\documentclass[reqno,10pt]{amsart}
\usepackage[off]{auto-pst-pdf}

\usepackage{amssymb}
\usepackage{amsmath}
\usepackage{amsthm}
\usepackage{mathrsfs}
\usepackage{pgf}
\usepackage{color}
\usepackage{graphicx}   

\usepackage{amsmath}
  \usepackage{paralist}
\usepackage{graphicx}  
\usepackage{psfrag}
\usepackage{comment}
\usepackage{esvect}

\newtheorem{theorem}{Theorem}[section]

\newtheorem{lemma}[theorem]{Lemma}
\newtheorem{proposition}[theorem]{Proposition}

\theoremstyle{definition}

\newcommand{\Rz}{\mathbb{R}}
\newcommand{\Nz}{\mathbb{N}}
\newcommand{\Zz}{\mathbb{Z}}

\newcommand{\eps}{\varepsilon}

\newcommand{\PPP}{\color{black}} 
\newcommand{\BBB}{\color{black}} 
 
\newcommand{\NNN}{\color{black}} 
\newcommand{\EEE}{\color{black}} 
\newcommand{\RRR}{\color{black}}

\usepackage{latexsym,amsfonts}
\usepackage{mathrsfs}
\usepackage{centernot}
\usepackage{paralist}

\usepackage{eso-pic}  
\usepackage{pifont}
\usepackage{fixmath} 

\usepackage{metalogo}
\usepackage{booktabs}
\makeatletter

\title[Characterization of optimal carbon nanotubes]{Characterization of optimal carbon nanotubes under stretching and validation of the Cauchy-Born rule}
\author{Manuel Friedrich} \author{Edoardo Mainini} \author{Paolo
  Piovano} \author{Ulisse Stefanelli} 


\address[Manuel Friedrich]{Applied Mathematics M\"unster, University of M\"unster\\
Einsteinstrasse 62, 48149 M\"unster, Germany}
\email{manuel.friedrich@uni-muenster.de}
\urladdr{https://www.uni-muenster.de/AMM/Friedrich/index.shtml}

\address[Edoardo Mainini]{Dipartimento di
   Ingegneria meccanica, energetica, gestionale e dei trasporti
   (DIME), Universit\`a degli Studi di Genova,  Via all'Opera Pia 15, I-16145 Genova, Italy 		}
\email{mainini@dime.unige.it}

\address[Paolo Piovano]{Faculty of Mathematics,  University of Vienna, Oskar-Morgenstern-Platz 1, A-1090 Vienna, Austria}
\email{paolo.piovano@univie.ac.at}
\urladdr{http://www.mat.univie.ac.at/~piovano/Paolo\texttt{Samp\_Dist\_Corr}Piovano.html}

\address[Ulisse Stefanelli]{Faculty of Mathematics,  University of Vienna, Oskar-Morgenstern-Platz 1, A-1090 Vienna, Austria $\&$ Istituto di Matematica Applicata e Tecnologie Informatiche ``E. Magenes'' - CNR, v. Ferrata 1, I-27100 Pavia, Italy}
\email{ulisse.stefanelli@univie.ac.at}
\urladdr{http://www.mat.univie.ac.at/~stefanelli}

\begin{document}

\maketitle


\begin{abstract}
Carbon nanotubes are modeled as point configurations and investigated by minimizing configurational
energies including two- and three-body interactions. Optimal
configurations are identified with local minima and their fine
geometry is fully characterized in terms of lower-dimensional problems. Under moderate
tension, \PPP we prove the existence of periodic local minimizers, \EEE which indeed
validates the so-called Cauchy-Born rule in this setting.
\end{abstract}

\tableofcontents 


\section{Introduction}

Nanostructured carbon has emerged over the last two decades as one of
the most promising materials available to mankind. The discovery of
fullerenes \cite{Kroto85,Kroto87}, followed by that of carbon
nanotubes \cite{Iijima91} and graphene \cite{Geim,Novoselov}, \PPP sparked an interest in \EEE   low-dimensional materials. 
The fascinating electronic and mechanical properties of
single-atom-thick surfaces and structures  are  believed to offer
unprecedented opportunities for innovative applications, ranging from
next-generation electronics to pharmacology, to batteries and
solar cells \cite{Gupta,Mannix,Mas-Balleste}.
New findings are emerging at an always increasing
pace, cutting across \PPP materials science, physics, and
chemistry, \EEE and extending from fundamental science to novel
applications \cite{Dresselhaus11,Morris13}. 

Carbon nanotubes are long, hollow  structures \PPP exhibiting cylindrical symmetry \EEE \cite{Dresselhaus92}. 
Their walls consist of a single (or multiple) one-atom-thick layer of
carbon atoms forming {\it $sp^2$ covalent} bonds
\cite{Clayden12} arranged in a hexagonal pattern. This molecular structure
is responsible for amazing mechanical properties: Carbon nanotubes are
presently among the strongest and stiffest known materials with a
nominal Young's modulus  \cite{KDEYT, TEG}  of 1 TPa and ideal strength greater than 100
MPa \cite{Arroyo05}. 
\PPP In addition, they are electrically and thermally conductive, chemically sensitive, transparent, and light weight \cite{Tuukanen}. \EEE
Nanotubes can be visualized
as the result of rolling up a patch of a regular hexagonal
lattice. Depending on the different possible realizations of this rolling-up, different
topologies may arise, giving rise to {\it zigzag, armchair},
and {\it chiral} nanotubes. These  topologies  are believed to have a specific impact on the
 mechanical and electronic properties of the nanotube, which can range from highly
  conducting to semiconducting \cite{Cao07,Charlier98}. 

In contrast to the ever-growing material knowledge, the rigorous mathematical
description of two-dimensional carbon systems is considerably less
developed. Ab initio atomistic models
are
believed to accurately describe some features of the carbon nanotube geometry
and mechanics \cite{Li07,Rochefort99,Yakobson96}. These methods are nevertheless computational in nature
and cannot handle a very large number of atoms due to the rapid
increase in 
computational complexity.
 On the other hand, a number
of continuum mechanics approaches have been proposed where carbon
nanotubes are modeled as rods \cite{Poncharal99}, shells
\cite{Arroyo05,Bajaj13,Favata12,Ru01}, or solids
\cite{Wang05}. These bring the advantage of possibly dealing with 
long structures, at the price however of a less accurate
description of the detailed microscopic behavior. 

The unique
mechanical behavior of nanotubes under {\it stretching} is a crucial feature
of these structures. As such, it has attracted attention from the
theoretical \cite{Bajaj13, Favata14, Ru01, Zhang08}, the computational
\cite{Agrawal, Cao07, Han14, Jindal}, and the experimental side
\cite{Demczyk, KDEYT,Warner11, YFAR}. \PPP Still, a reliable description of nanotubes under stretching requires one to correctly resolve \EEE   the atomic
scale and,  simultaneously, to rigorously deal with the whole
structure. We hence resort to 
   the classical frame of \PPP molecular
mechanics  \EEE \cite{Allinger,Lewars,Rappe} which identifies carbon
nanotubes 
with point configurations $\{x_1,
\dots, x_n\}\in \Rz^{3n}$ corresponding to their atomic
positions.  The atoms  are interacting via a {\it configurational
energy} $E=E(x_1,
\dots, x_n)$ given in terms of  classical potentials and taking 
into account both attractive-repulsive {\it two-body} interactions,
minimized at  a certain  bond length, and {\it three-body} terms
favoring specific angles between bonds
\cite{Brenner90,Stillinger-Weber85,Tersoff}.  The $sp^2$-type covalent bonding
implies that each atom has exactly three first neighbors
 and  that bond angles of
$2\pi/3$ are energetically preferred
\cite{Clayden12}.  The \PPP reader \EEE is referred to 
\cite{Davoli15,E-Li09,  FPS, Mainini-Stefanelli12,  stable} 
for a collection of results on local and global minimizers
in this setting and to \cite{Smereka15,cronut} for additional
results on carbon structures.  

The focus of this paper is to show the local minimality of 
periodic configurations, both in the unstreched case and under the effect of small
stretching.   More specifically, we prove that, by applying a small
stretching to a zigzag nanotube, the energy $E$ is locally strictly minimized by a  specific  periodic
configuration where all atoms see the same local configuration
(Theorem \ref{th: main3}). Local
minimality is here checked with respect to {\it all}  small  perturbations in
$\Rz^{3n}$, namely not restricting {\it a priori} to periodic
perturbations. On the contrary, periodicity is proved here to emerge as 
effect of the global variational nature of the problem.

The novelty of this result is threefold. At first, given the 
periodicity of the \PPP mentioned \EEE local minimizers, the actual configuration in $\Rz^{3n}$ can be
determined by solving a simple minimization problem in $\Rz^2$,  which consists in identifying the length of two specific bond lengths between neighboring atoms.  This
is indeed the standpoint of a number of contributions, see  
\cite{Agrawal,Budyka,Favata15,Favata16,Jiang,Jindal,Kanamitsu,Kurti}
among many others,  where nevertheless
periodicity is a priori {\it assumed}. In this regard, our result
offers a justification for these lower-dimensional approaches.
\PPP Our assumptions on $E$ are kept fairly general in order to 
include the \EEE menagerie of different
possible choices for energy
terms  which have   been implemented in \PPP computational chemistry \EEE codes \cite{Brooks83,Clark89,Gunsteren87,Mayo90,Weiner81}. A by-product of our
results is hence  the  
cross-validation of  these choices in view of their capability of
describing carbon nanotube geometries.

Secondly, we rigorously check that, also in presence of small
stretching, the geometrical model obtained via
local minimization corresponds neither to the classical {\it
  rolled-up} 
model  \cite{Dresselhaus92,Dresselhaus-et-al95, Jishi93},  where two out of three bond angles at each atom are
$2\pi/3$, nor to the {\it polyhedral} model   \cite{Cox-Hill07,Cox-Hill08,Lee}, where all  bond
angles are equal. The optimal configuration lies between these two
 (Proposition \ref{th: main2}),  a fact which remarkably corresponds
to measurements on very thin carbon nanotubes \cite{Zhao}.  Moreover, in accordance with the results in \cite{Jindal},  local minimizers are generically characterized by two different bond lengths. 
 
Finally, our result proves the validity of the so-called {\it
  Cauchy-Born rule} for carbon nanotubes: By imposing a small tension, the
periodicity cell deforms correspondingly and global periodicity is preserved. This
fact rests at the basis of a possible elastic theory for carbon
nanotubes. As a matter of fact, such periodicity is \PPP invariably \EEE {\it assumed} in a number of different
contributions, see \cite{Bajaj13,Favata14,Han14,Zhang08} among others, and then exploited in order to compute
tensile strength as well as stretched geometries. Here again our
results provide a theoretical justification of such approaches.

\PPP While the Cauchy-Born rule plays a pivotal role in mechanics \cite{Ericksen08, Ericksen84,Zanzotto92},  rigorous results are scarce.  \EEE
Among these we mention \cite{Friesecke02,Conti06}, which assess
its validity within two- and $d$-dimensional cubic  mass-spring systems,  respectively. More general interactions are considered in
\cite{E07,E07b}, where the Cauchy-Born rule is investigated under a
specific ellipticity condition applying to the triangular and
hexagonal lattice, both in the
static and the dynamic case. Our result is, to the best of our knowledge, the first one dealing
with a three-dimensional structure which is not a subset of a Bravais
lattice nor of a multilattice. Note  though  the Saint Venant principle
in \cite{Monneau14}, which corresponds to the validity of an approximate version of
the Cauchy-Born rule, up to a small error. However, the setting of
\cite{Monneau14} is  quite different from the present one,  where  
long-range purely two-body interactions are considered.

This work is the culmination of a series on the geometry and mechanics
of nanotubes \cite{MMPS,MMPS-new}. The theoretical outcomes of this
paper \PPP have been predicted computationally \EEE in \cite{MMPS}, where
stability of periodic configurations have been investigated with Monte
Carlo techniques, both for zigzag and armchair topologies under
moderate displacements.  A first step toward a rigorous analytical result
has been obtained in \cite{MMPS-new} for both zigzag and armchair
topologies under no stretching.  In \cite{MMPS-new},  stability is checked against a
number of non-periodic perturbations fulfilling a specific structural
constraint,  which is related to the nonplanarity of the hexagonal cells induced by the local geometry of the nanotube.  Here, we remove such constraint and consider all small
perturbations, even in presence of stretching.

 Indeed, removing the structural assumption and extending the result   of  \cite{MMPS-new} to the
present   fully general setting requires a remarkably deeper analysis.   In a nutshell,
one has to reduce to a cell problem and solve it. The actual
realization of this program poses however substantial technical
challenges and relies on a combination of perturbative and convexity techniques.

 Whereas the proof in \cite{MMPS-new} was essentially based on the convexity of the energy given by the three bond angles at one atom, in the present context we have to reduce to a  {\it cell}  which  includes eight atoms and is slightly nonplanar. The convexity of cell energies for various Bravais lattices has already been investigated in the literature \cite{Conti06, FriedrichSchmidt:2014.1, Friesecke02, Schmidt:2009}, particularly for problems related to the validation of the Cauchy-Born rule. In our setting, however, we need to deal with an almost planar structure embedded in the three-dimensional space and therefore, to confirm convexity of the cell energy, a careful analysis in terms of the  
nonplanarity is necessary, see Section \ref{sec: convexity} and Theorem \ref{th: cell convexity3}. \BBB In this context, an additional difficulty lies in the fact that the reference configuration of the cell is not a stress-free state. \EEE 

 The convexity is then  crucially exploited in order to obtain a quantitative control of the  \emph{energy defect} in terms of the \emph{symmetry defect} produced by symmetrizing a cell (Theorem \ref{th: Ered}). On the other hand, a
second quantitative estimate provides a bound on the  defect in
 the nonplanarity of the cell (called \emph{angle defect})  with respect to the  symmetry defect  of the cell (Lemma
\ref{lemma: sum}). The detailed
combination of these two estimates and a convexity and monotonicity argument  (Proposition \ref{th: mainenergy})  proves that ground states necessarily
have symmetric cells, from which our stability result follows (Theorem
\ref{th: main3}).

The validation of the Cauchy Born rule essentially relies on the application of a slicing technique which has also been used in \cite{FriedrichSchmidt:2014.1} in a more general setting: One reduces the problem to a chain of cells along the diameter of the structure and shows that identical deformation of each cell is energetically favorable. In the present context, however,  additional slicing arguments along the cross sections of the nanotube are necessary in order to identify correctly the nonplanarity of each hexagonal cell.

The paper is organized as follows. In Section \ref{Fsection} we introduce \PPP some notation \EEE and the mathematical setting. Section \ref{sec: mainresults} collects our main results. In Section \ref{sec: main proof} we present the proof strategy, the essential auxiliary statements (Lemma
\ref{lemma: sum} - Theorem \ref{th: Ered}),  and the proof of Theorem \ref{th: main3}. The proofs of the various necessary ingredients are postponed to  Sections \ref{sec: angles}-\ref{sec: cellenery}.

\section{Carbon-nanotube geometry}\label{Fsection}

\PPP The aim of this section is to introduce some notation and the nanotube configurational energy. \EEE
Let us start by introducing the mathematical setting as well as
some preliminary observations. 

As mentioned above, carbon nanotubes (nanotubes, in the following) are
modeled  by
{\it configurations} of atoms, \PPP i.e., \EEE collections of points in $\Rz^3$
representing the atomic sites. Nanotubes are very long
structures, measuring up to   $10^7$  times their diameter. 
As such, we shall not be concerned with  describing the  fine
 nanotube geometry close to their
ends.  \PPP We thus restrict our attention to \EEE  periodic configurations,  \PPP i.e., \EEE configurations that are invariant with respect to a translation of a certain period in the direction of the nanotube axis. Without loss of generality we consider only nanotubes with axis in the $e_1:=(1,0,0)$ direction. Therefore, a nanotube is identified with a configuration 
$$\mathcal{C}:=C_n+Le_1\Zz$$
where $L>0$ is the {\it period} of $\mathcal{C}$ and $C_n:=\{x_1,\dots,x_n\}$ is a collection of $n$ points $x_i\in\Rz^3$ such that $x_i\cdot e_1\in [0,L)$. In the following, we will refer to $C_n$ as the  {\it $n$-cell} of $\mathcal{C}$, and since $\mathcal{C}$ is characterized by its $n$-cell $C_n$ and its period $L$,  we will systematically identify the periodic configuration $\mathcal{C}$ with the couple $(C_n, L)$,  \PPP i.e., \EEE $\mathcal{C}=(C_n, L)$.

\subsection{Configurational energy}

We now introduce {\it the configurational energy $E$ of a nanotube
  $\mathcal{C}$}, and \PPP we detail the hypotheses 
 on $E$ that we assume throughout the paper. \EEE We aim here at minimal
assumptions in order \PPP to include \EEE in the analysis most of the
many different
 possible choices for energy
terms \PPP that have been \EEE successfully implemented in \PPP computational chemistry \EEE codes
\cite{Brooks83,Clark89,Gunsteren87,Mayo90,Weiner81}.

The energy $E$  is given by the sum of two contributions, respectively accounting for  {\it   two-body and  three-body  interactions among particles}  that are respectively \PPP modelled by \EEE the potentials $v_2$ and $v_3$, see \eqref{E}. 

We assume that the   {\it two-body potential} $v_2: \PPP (0,\infty) \EEE\to[-1,\infty)$ is smooth and attains its minimum value \RRR only \EEE at $1$ with $v_2(1) = -1$ and  $v''_2(1)>0$. Moreover, we ask $v_2$ to be {\it short-ranged}, that is to vanish 
shortly after $1$.  For the sake of definiteness, let us define $v_2(r)=0$
for $r\ge 1.1$. These assumptions reflect the nature of covalent
atomic bonding in carbon  favoring  a specific
interatomic distance, here normalized to $1$.  

 We say that two particles $x,y\in\mathcal{C}$ are {\it bonded} if
 $|x-y|<1.1$, and we refer to the graph formed by all the  bonds as
 the {\it bond graph} of $\mathcal{C}$.  Taking into account 
 periodicity, \PPP  this amounts to considering two particles $x_i$ and $x_j$ of the $n$-cell $C_n$ of $\mathcal{C}$ to be bonded if \EEE $|x_i-x_j|_L<1.1$, where $|\cdot|_L$ is the {\it distance modulo $L$} defined by
$$
|x_i-x_j|_L:=\min_{t\in\{-1,0,+1\}}|x_i-x_j+Lte_1|
$$
for every $x_i,x_j\in C_n$. Let us denote by $\mathcal{N}$ the set of all couples of indices corresponding to bonded particles,  \PPP i.e., \EEE
$$
\mathcal{N}:=\{(i,j)\,:\,\, \textrm{$x_i$, $x_j\in C_n$, $i\neq j$, and $|x_i-x_j|_L<1.1$}\}.
$$

The  {\it three-body potential}
$v_3: [0,2\pi]\to[0,\infty)$
is assumed to be smooth and  symmetric around $\pi$, namely
$v_3(\alpha)=v_3(2\pi{-}\alpha)$. Moreover, we suppose that the minimum value $0$
is attained only at $2\pi/3$ and $4\pi/3$ with $v_3''(2\pi/3)>0$. Let $\mathcal{T}$  be the index set  \PPP of \EEE the triples corresponding to first-neighboring particles,  \PPP i.e., \EEE
$$
\mathcal{T}:=\{(i,j,k)\,:\,\, \textrm{$i\neq k$, $(i,j)\in\mathcal{N}$ and $(j,k)\in\mathcal{N}$}\}.
$$ 
For all triples $(i,j,k)\in\mathcal{T}$ we denote by  $\alpha_{ijk} \in  \RRR [0,\pi] \EEE $
the {\it bond angle} formed by the vectors $x_i-x_j$ and
$x_k-x_j$. The assumptions on $v_3$ reflect the basic geometry of
carbon bonding in a nanotube: Each atom presents three $sp^2$-hybridized
orbitals, which tend to form $2\pi/3$ angles.

The configurational energy $E$ of a nanotube $\mathcal{C}=(C_n, L)$ is now defined by
\begin{equation}\label{E}
E(\mathcal{C})=E(C_n,L):=\frac12\sum_{(i,j)\in \mathcal{N}}v_2(|x_i{-}x_j|_L) + \frac12\sum_{(i,j,k)\in\mathcal{T}}v_3(\alpha_{ijk}),
\end{equation} 
\PPP where the factors $1/2$ are included to avoid double-counting the interactions among same atoms. \EEE
Let us mention  that the smoothness assumptions on $v_2$ and $v_3$ are for the 
sake of maximizing simplicity rather than generality and could be weakened. Observe that our  assumptions are generally satisfied by classical interaction potentials for carbon (see \cite{Stillinger-Weber85,Tersoff}). Since the energy $E$ is clearly rotationally and translationally invariant, in the following we will tacitly assume that all statements are to be considered up to isometries. We say that a nanotube $\mathcal{C}=(C_n,L)$ is {\it stable} if $(C_n,L)$ is a  strict  local minimizer of the interaction  energy   $E$.

\subsection{Geometry of zigzag nanotubes} We now  
introduce a specific two-parameter family of  nanotubes which will
play a crucial role in the following. This is the family of so-called
{\it zigzag nanotubes} having the 
 \emph{minimal period} $\mu>0$. The term {\it zigzag} refers to a
 specific topology of nanotubes, \PPP which can be visualized as the
 result of a rolling-up of a graphene sheet along a specific lattice
 direction, see Figure \ref{fig:nanotube}.
\begin{figure}[h]
    \pgfdeclareimage[width=0.5\textwidth]{nanotube}{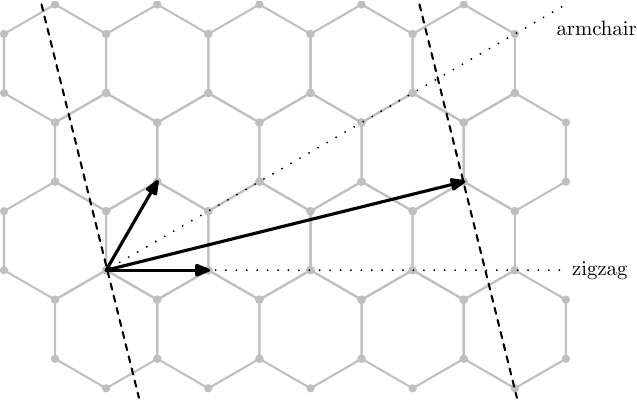}
    \pgfuseimage{nanotube}
    \caption{\PPP Rolling-up a graphene sheet to a zigzag nanotube: the vector illustrates the
      identification of the two dashed vertical lines. The term {\it
        zigzag} refers to the orientation of this vector with respect to
      bonds. Different vectors
      correspond indeed to different nanotube topologies. The dotted
     \PPP line indicates \PPP the identification direction for {\it armchair} nanotubes.}
    \label{fig:nanotube}
  \end{figure}
The resulting three-dimensional structure is depicted in Figure
\ref{Figure1}. Note that our
preference for the zigzag topology is solely motivated by the sake of definiteness. 
The \EEE other classical choice, \PPP namely the \EEE so-called {\it armchair}
 topology, could be considered as well. The \PPP reader \EEE is referred to
 \cite{MMPS-new} for some results on unstretched armchair geometries.

\begin{figure}[h]
    \pgfdeclareimage[width=0.9\textwidth]{tube}{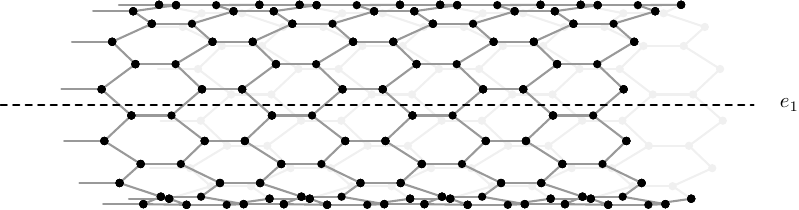}
    \pgfuseimage{tube}
    \caption{Zigzag nanotube.}
    \label{Figure1}
  \end{figure}

We let $\ell \in \Nz$, $\ell >3$, and  define the family $\mathscr{F}(\mu)$ as the
collection of all configurations that, up to isometries, coincide with
 \begin{align}
&\Bigg\{\left(k (\lambda_1+\sigma) +  j (2\sigma+2\lambda_1)  +
  l(2\sigma + \lambda_1), \rho \cos \left(\frac{\pi(2i+k)}{\ell}\right) , \rho\sin
    \left(\frac{\pi(2i+k)}{\ell}\right) \right)
 \ \Big| \ \nonumber
 \\
 &\hspace{80mm}
  i=1,\dots,\ell, \ j \in \Zz, \ k,l \in \lbrace 0,1 \rbrace \Bigg\}\label{zigzagfamilydefinition}
\end{align}
for some choice of 
$$\lambda_1 \in (0,\mu/2),\ \ \  \lambda_2 \in (0,\mu/2),\ \ \ \sigma \in (0,\mu/2), \ \ \   \text{and}  \ \ \ \rho\in \left(0,\,\frac{\mu}{4\sin(\pi/(2\ell))}\right)$$
 such that
\begin{align}\label{eq: basic constraints}
2\sigma + 2\lambda_1 = \mu, \ \ \ \ \ \ \ \ \ \sigma^2 +4\rho^2\sin^2\left(\frac{\pi}{2\ell}\right) =\lambda_2^2.
\end{align}
Of course, the configurations in $\mathscr{F}(\mu)$ are periodic with minimal period $\mu$. The parameter  $\rho$ indicates the diameter of the tube and $\lambda_1$, $\lambda_2$ are the two possibly different lengths of the covalent bonds in each hexagon of the tube, where the bonds of length $\lambda_1$ are oriented in the $e_1$ direction (see Figure \ref{extra}).  

These configurations are  {\it
  objective}  \cite{James}: They are obtained as
orbits of \NNN two points \EEE under the action of a
prescribed isometry group. \PPP 
The latter group is generated by a translation and by a translation combined with a rotation about the $e_1$-axis. \EEE
Notice that  our definition slightly differs from the one adopted in \cite{MMPS,MMPS-new} in the sense that for fixed $i$, $k$ the points \PPP identified \EEE by the quadruples $(i,j,k,l)$ for $j \in \Zz$, $l \in \lbrace 0,1 \rbrace$ lie on a line parallel to   $e_1$  (see Figure \ref{quadruples}).

  For fixed $\mu >0$,  $\mathscr{F}(\mu)$ is  a two-parameter smooth family of configurations since 
each configuration in $\mathscr{F}(\mu)$ is uniquely determined by
$\lambda_1$ and $\lambda_2$ by taking relation \eqref{eq: basic
  constraints} into account. Later we will consider different values
for the minimal period $\mu$ in order to model nanotubes under
stretching. 

 We state the following basic geometric properties of configurations in $\mathscr{F}(\mu)$ \PPP (see Figure \ref{quadruples}). \EEE The analogous properties in the case $\lambda_1 = \lambda_2 = 1$ have already been discussed in \cite{MMPS}. 

\begin{proposition}[Geometric structure of zigzag nanotubes]\label{basiczigzag}
Let $\mathcal{F}\in\mathscr{F}(\mu)$. Then
 \begin{enumerate}
\item[\rm (a)] Atoms in $\mathcal{F}$ lie on the surface of a cylinder with radius $\rho$ and axis  $e_1$. 
\item[\rm (b)] Atoms in $\mathcal{F}$ are arranged in planar {sections}, perpendicular to $e_1$, obtained by fixing $j$, $k$, and $l$ in \eqref{zigzagfamilydefinition}. Each of the sections \PPP contains  exactly \EEE $\ell$ atoms, arranged at the vertices of a regular $\ell$-gon. For each section, the two closest sections are at distance $\sigma$ and $\lambda_1$, respectively.
\item[\rm (c)] The configuration $\mathcal{F}$ is invariant under a rotation of
  $2\pi/\ell$ around $e_1$, under the translation  $\mu e_1$, and under
  a \PPP transformation consisting of a rotation of $\pi/\ell$ around $e_1$ and a translation along the vector $(\lambda_1+\sigma)e_1$ \emph{(}see Figure  \ref{extra}\emph{)}.  \EEE
\item [\rm (d)] Let $i\in\{1,\ldots,\ell\} $, $j\in\mathbb{Z}$ and $k,l\in\{0,1\}$: the quadruple $(i,j,k,l)$ \PPP identifies \EEE points of $\mathcal{F}$, denoted by $x_{i,k}^{j,l}$, where $(0,j,k,l)$ \PPP is identified \EEE with $(\ell,j,k,l)$.   Given $x_{i,0}^{j,0}\in \mathcal{F}$, the two points $x_{i,1}^{j-1,1}$, $x_{i-1, 1}^{j-1,1}$ have distance $\lambda_2$ and  $x_{i, 0}^{j-1,1}$ has distance $\lambda_1$ from $x_{i,0}^{j,0}$.
 For $x_{i,0} ^{j,1}$, the distance of $x_{i,1}^{j,0}$ and  $x_{i-1,1}^{j,0}$ is $\lambda_2$ and the distance from $x_{i,0}^{j+1,0}$ is $\lambda_1$. See {\rm Figure \ref{quadruples}} for the analogous notation of $x_{i,1}^{j,0}$ and $x_{i,1}^{j,1}$.  
  \end{enumerate} 
 \end{proposition}
 \begin{figure}[htp]
\pgfdeclareimage[width=0.7\textwidth]{quadruples}{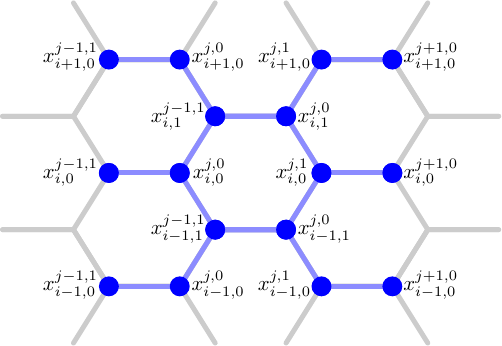}
    \pgfuseimage{quadruples}
\caption{ Configuration points are \PPP identified \EEE by quadruples $(i,j,k,l)$ for $i=1,\dots,\ell$, $j \in \Zz$, and $k,l \in \lbrace 0,1 \rbrace$.}
\label{quadruples} 
\end{figure}

Notice that for fixed $\lambda_1$ and $\lambda_2$ the other parameters
range between two degenerate cases: $\rho=0$ (the cylinder is reduced
to its axis) and $\sigma=0$ (sections collide).  We shall however
 impose further restrictions,  for each atom should have three
bonds.  In particular, the only three bonds per atom should be the ones \PPP identified \EEE by point (d) of Proposition \ref{basiczigzag}.
By recalling that two particles are bonded if their distance is less than the reference value $1.1$, since the distance between two consecutive sections is either $\lambda_1$ or $\sigma$, we require $\lambda_1 > 0.9$ and  $\sigma>0.2$. Additionally, we require $\lambda_1,\lambda_2 < 1.1$, which also implies $\sigma < 1.1$ by \eqref{eq: basic constraints}.  On the other hand, on each section, the edge of the regular $\ell$-gon should be greater than $1.1$.
Such length is given by $
2\rho\sin\gamma_\ell,
$
 where $\gamma_\ell$ is the internal angle of a regular $2\ell$-gon,  \PPP i.e., \EEE
 \begin{equation}\label{gamma}
\gamma_\ell:=\pi\left(1-\frac{1}{\ell}\right).
\end{equation}  
Therefore, we need to impose $\rho>\rho^-:=0.55/\sin\gamma_\ell$. With these restrictions we have the following

\begin{proposition}[Parametrization of the family]\label{betaproperties} Let $ \mathcal{F}\in\mathscr{F}(\mu)$ with $\rho > \rho^-$, $\sigma >0.2$ and $\lambda_1,\lambda_2 \in (0.9, 1.1)$. Then, all atoms in $\mathcal{F}$ have exactly three (first-nearest) neighbors, two at distance $\lambda_2$ and one at distance $\lambda_1$, where the bond corresponding to the latter neighbor is \PPP parallel to $e_1$. \EEE Among the corresponding  \PPP 
 three bond angles, which are smaller than $\pi$, \EEE two have amplitude
$\alpha$ (the ones involving atoms in three different sections), and the third has amplitude $\beta$ \PPP \emph{(}see Figure \ref{extra}\emph{)}, \EEE
  where $\alpha\in(\pi/2,\pi)$ is obtained from
  \begin{equation}\label{alphars}
  \sin\alpha=\sqrt{1-(\sigma/\lambda_2)^2}=2 (\rho/\lambda_2)\sin\left(\frac{\pi}{2\ell}\right)
  \end{equation}
  and $\beta\in(\pi/2,\pi)$ is given by
  \begin{equation}\label{betaz}
\beta=\beta(\alpha,\gamma_\ell):=2\arcsin\left(\sin\alpha\sin\frac{\gamma_\ell}{2}\right).
\end{equation}
\end{proposition}
 \begin{figure}[htp]
\begin{center}
\pgfdeclareimage[width=0.45\textwidth]{cell}{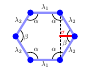}
    \pgfuseimage{cell}
\caption{\PPP The bond lengths and the angles for the hexagon of a configuration in \PPP $\mathscr{F}(\mu)$ are represented. A segment representing $\sigma$ is drawn in red. \EEE}
\label{extra}
\end{center}
\end{figure}

The proof for the case $\lambda_1 = \lambda_2=1$ was  detailed  in
\cite{MMPS}. The extension to our setting  is a straightforward
adaption  and is therefore omitted.  As already mentioned,   the collection $\mathscr{F}(\mu)$ is a two-parameter family where all its configurations are  uniquely determined by the specification of $\lambda_1$ and $\lambda_2$. The corresponding element will be denoted by $\mathcal{F}_{\lambda_1,\lambda_2,\mu}$. Restricting the minimal period $\mu$ to the interval $(2.6,3.1)$ we observe by \eqref{eq: basic constraints} and an elementary computation that the constraints $\lambda_1,\lambda_2 \in (0.9,1.1)$ and $\ell >3$ automatically imply $0.2 < \sigma <0.65$ and $\rho>\rho^-$. Therefore, the assumptions of {\BBB Proposition} \ref{betaproperties} hold.  

\section{Main results}\label{sec: mainresults}

In this section we collect  our main results. The corresponding
proofs will then be presented in Sections \ref{sec: main proof}-\ref{sec:
  cellenery}.

For  \PPP a \EEE fixed integer $\ell>3$, let us consider a configuration $\mathcal{F}$  in the family
$\mathscr{F}(\mu)$. \PPP As $\mathcal{F}$ is periodic, it can be identified with the couple \EEE $(F_n, L)$, where $F_n$ is the corresponding $n$-cell ($n=4m\ell$ for some $m\in \mathbb{N}$),
and
\begin{equation}\label{zigzagperiod}
L = L^\mu_m:= m\mu
\end{equation}
is the period parameter, corresponding to the cell length (notice that for $m=1$ we get the minimal period of the configuration). In view of \eqref{E} and the properties stated in  Proposition   \ref{betaproperties}, the energy can be written as 
\begin{align}\label{basicenergy}
E(\mathcal{F}) = E(F_n, L^\mu_m) = \frac{n}{2} \big( v_2(\lambda_1) + 2v_2(\lambda_2)  \big) + n\big(2v_3(\alpha) + v_3(\beta(\alpha,\gamma_\ell))\big).
\end{align}

\subsection{Unstrechted nanotubes} 
{\BBB A first natural problem to be considered is the energy minimization restricted to the families $\mathscr{F}(\mu)$, with the values of $\mu$ in the reference interval $\mu\in (2.6, 3.1)$. Let us denote by $\mathcal{F}_{\lambda_1,\lambda_2,\mu}$ an element of $\mathscr{F}(\mu)$ with bond lengths $\lambda_1,\lambda_2$. If we minimize among nanotubes $\mathcal{F}_{\lambda_1,\lambda_2,\mu}$ with respect to $\mu\in (2.6,3.1)$   and  $\lambda_1,\lambda_2$ in a neighborhood of $1$,  we  reduce to the case $\lambda_1=\lambda_2=1$}. \RRR Indeed,   we can replace   $\lambda_1,\lambda_2$ by $1$, leave $\alpha$ unchanged, and choose $\mu$ according to  \eqref{eq: basic constraints} and \eqref{alphars}    such that the energy \eqref{basicenergy} decreases\EEE .

We notice that $\lbrace \mathcal{F}_{1,1,\mu}| \  \mu \in (2.6,3.1) \rbrace $ is a one-parameter family.  It
follows from Proposition \ref{betaproperties} and \eqref{eq: basic constraints} that this family can also be parametrized in terms of the  bond angle $\alpha$ introduced in Proposition \ref{betaproperties} using the relation $\mu = 2 (1-\cos\alpha)$.  We indicate these configurations by $\mathcal{G}_\alpha$. 

As  already  discussed in \cite{MMPS}, there are two specific angles $\alpha^{\rm
  ch}_\ell < \alpha^{\rm ru}$ corresponding to the \emph{rolled-up}
 \cite{Dresselhaus92,Dresselhaus-et-al95} 
and  \emph{polyhedral} \cite{Cox-Hill07,Cox-Hill08}  configuration, respectively, with $\alpha^{\rm ru}
= 2\pi/3$ and $\alpha^{\rm ch}_\ell$  being  the unique solution of the equation  $\beta(\alpha^{\rm ch}_\ell, \gamma_\ell) = \alpha^{\rm ch}_\ell$ in $(\arccos(-0.4), \arccos(-0.6))$. The {\BBB one variable minimization problem for the map $\alpha\mapsto E(\mathcal{G}_\alpha)$ has} been investigated in \cite[Theorem 4.3]{MMPS}:

\begin{proposition}[Existence and uniqueness of minimizer: Unstretched case]\label{eq: old main result}
There exist an open interval $A$ and $\ell_0 \in \Nz$ only depending on  $v_3$  such \PPP  that \EEE the following holds for all $\ell \ge \ell_0$:  There is a unique angle $\alpha^{\rm us}_\ell \in A$ such that $\mathcal{G}_{\alpha^{\rm us}_\ell}$ minimizes the energy $E$ in the class $\lbrace  \mathcal{G}_\alpha| \ \alpha \in A \rbrace$. Moreover, one has $\alpha^{\rm us}_\ell \in  (\alpha^{\rm ch}_\ell,\alpha^{\rm ru}) \subset A$.
\end{proposition}  

Let us report the idea of the proof. Exploiting the monotonicity
properties of $v_3$ and $\beta$ (the latter being decreasing as a
function of $\alpha$), one derives that the minimum  is  attained for $\alpha$ in a small left neighborhood $I$ of $2\pi/3$, \PPP e.g., $I:=(2\pi/3-\sigma, 2\pi/3]$ for some small $\sigma>0$. \EEE Using in addition the convexity of $v_3$ and the concavity of $\beta$, it follows that $\alpha \mapsto E(\mathcal{F}) = -3n/2  + n\big(2v_3(\alpha) + v_3(\beta(\alpha,\gamma_\ell))\big)$ is strictly convex in $I$, which implies the assertion.

The result in particular shows that neither the \PPP polyhedral \EEE nor the
rolled-up configuration is a local minimizer of the energy $E$. The
corresponding minimal period of the nanotube is given by 
\begin{align}\label{eq:muellneu}
\mu^{\rm
  us}_\ell := 2 - 2\cos\alpha^{\rm us}_\ell,
  \end{align}
  cf. \eqref{eq: basic
  constraints} and \eqref{alphars}, and we notice
$\mathcal{G}_{\alpha^{\rm us}_\ell} = \mathcal{F}_{1,1,\mu^{\rm
    us}_\ell}$.  {\BBB Nanotubes with $\mu=\mu_\ell^{\rm us}$ will be referred to as  {\it unstretched} nanotubes.
    }

The aim of   \cite{MMPS,MMPS-new} was  to prove  that 
$\mathcal{G}_{\alpha^{\rm us}_\ell}$ is a local minimizer. This has
been illustrated numerically in \cite{MMPS} and checked analytically
in \cite{MMPS-new},
for a restricted class of perturbations. Our stability result Theorem
\ref{th: main3} below delivers an analytical proof of stability with
respect to {\it all} small perturbations. As such, it generalizes
 and improves  known results, even in the unstreched case.

\subsection{Nanotubes under stretching}  Let us now move forward
to the case of {\it stretched} nanotubes. This corresponds to
 choosing  $\mu \neq \mu^{\rm us}_\ell$. Indeed, we impose a
tensile  or compressive  stress on the nanotube by simply modifying its minimal
period. Given the role of periodicity in the definition of the energy
$E$,  see \eqref{E},   this has the net effect of stretching/compressing the
structure. Note that this action on the structure is very general. In
particular, it includes, without reducing to, imposed Dirichlet boundary
conditions, where only the first coordinate of the boundary atoms is prescribed.  For   fixed $\mu \in (2.6,3.1)$ we consider the minimization problem 
\begin{align}\label{min2} 
E_{\rm min}(\mu)  = \min\big\{ E(\mathcal{F}_{\lambda_1,\lambda_2,\mu})| \ \mathcal{F}_{\lambda_1,\lambda_2,\mu} \in \mathscr{F}(\mu), \ \lambda_1,\lambda_2 \in   (0.9,1.1) \big\}.
\end{align}
We obtain the following existence result.

\begin{theorem}[Existence and uniqueness of minimizer: General case]\label{th: main1}
There  exist $\ell_0 \in \Nz$ and, for each $\ell \ge \ell_0$, an open interval $M^\ell$ only depending on $v_2$, $v_3$, and $\ell$, with $\mu^{\rm us}_\ell \in M^\ell$, such that for all $\mu \in M^\ell$ there is a unique pair of bondlengths $(\lambda^\mu_1,\lambda^\mu_2)$ such that $\mathcal{F}_{\lambda^\mu_1,\lambda^\mu_2,\mu}$ is a solution of the problem \eqref{min2}.
\end{theorem}

In the following the minimizer is denoted by $\mathcal{F}_\mu^*$. Note
that we have  $\mathcal{F}_{\mu^{\rm us}_\ell}^* =
\mathcal{G}_{\alpha^{\rm us}_\ell}$  by Proposition \ref{eq: old main
  result}.

 Our aim is to investigate the local stability of $\mathcal{F}_\mu^*$. To this end, we consider \emph{general} small perturbations $\tilde{\mathcal{F}}$ of $\mathcal{F}_\mu^*$ with the same bond graph,  \PPP i.e., \EEE each atom keeps three and only three bonds, and we can identify the three neighboring atoms of the perturbed configurations with the ones for the configuration $\mathcal{F}_\mu^*$. By $F^\mu_n = \lbrace x^\mu_1,\ldots, x^\mu_n \rbrace$ denote the $n$-cell of $\mathcal{F}_\mu^*$ so that $\mathcal{F}_\mu^* = (F_n^\mu, L^\mu_m)$ with $L^\mu_m$ as defined in \eqref{zigzagperiod} for $m \in \Nz$ with $n = 4m\ell$. We define \emph{small perturbations} $\mathscr{P}_\eta(\mu)$ of $\mathcal{F}_\mu^*$ by 
\begin{align}\label{eq: bc}
\begin{split}
\mathscr{P}_\eta(\mu) = \lbrace \tilde{\mathcal{F}} = (F_n,L^\mu_m)| \ F_n := \lbrace x_1,\ldots,x_n \rbrace \ \text{ with } |x_i - x_i^\mu| \le \eta \rbrace.
\end{split}
\end{align}
The parameter $\eta>0$ will always be chosen sufficiently small such
that the topology of the bond graph remains invariant. 
$\eta$ will in general also depend on $\ell$. Moreover, we recall
$E(\tilde{\mathcal{F}}) = E(F_n, L^\mu_m)$. We obtain our main
result, concerning local stability under small stretching.

\begin{theorem}[Local stability of minimizers]\label{th: main3}
There  exist $\ell_0 \in \Nz$ and for each $\ell \ge \ell_0$ some $\mu^{\rm crit}_\ell > \mu^{\rm us}_\ell$ and $\eta_\ell >0$ only depending on $v_2$, $v_3$, and $\ell$  such that  for all $\ell \ge \ell_0$ and  for all  $\mu \in [\mu_\ell^{\rm us},\mu_\ell^{\rm crit}]$   we have
$$E(\tilde{\mathcal{F}})>E(\mathcal{F}_\mu^*) $$
for any nontrivial perturbation $\tilde{\mathcal{F}} \in \mathscr{P}_{\eta_\ell}(\mu)$  of the configuration $\mathcal{F}_\mu^*$.
\end{theorem}

 The theorem asserts that, under prescribed and small stretchings \PPP (i.e., the value of $L_m^\mu$ is prescribed), \EEE
 \PPP there exists a periodic strict-local minimizer  $\mathcal{F}_\mu^*$ that  belongs \EEE  to the
family $\mathscr{F}(\mu)$. \PPP In other words, given $\mu>\mu^{\rm us}$, the $\mu$-periodic configuration $\mathcal{F}_\mu^*$  is a local minimizer among 
configurations subject to the same macroscopic stretching, i.e., the atoms follow the
macroscopic deformation. This can be seen as a validation of the
Cauchy-Born rule  in  this specific setting. \EEE
 Especially,  the
result justifies the reduction of the $3n$-dimensional minimization problem
$\min \{E(\mathcal{F}) | \ \mathcal{F} \in \mathscr{P}_{\eta_\ell}(\mu)\}$ to the two-dimensional
problem \eqref{min2}.  

In the following statement we collect    the main properties of the local  minimizer.

\begin{proposition}[Properties of minimizer]\label{th: main2}
There  exist $\ell_0 \in \Nz$ and for each $\ell \ge \ell_0$ an open interval $M^\ell$ only depending on $v_2$, $v_3$, and $\ell$, with $\mu^{\rm us}_\ell \in M^\ell$, such that:
\begin{itemize}
\item[1.] The mapping $\mu \mapsto E(\mathcal{F}_\mu^*) = E_{\rm min}(\mu)$ is smooth, strictly convex on $M^\ell$ and attains its minimum in $\mu_\ell^{\rm us}$. Particularly, $\frac{d^2}{d\mu^2} E_{\rm min}(\mu_\ell^{\rm us}) \ge cn$ for $c>0$ only depending on $v_2$,  $v_3$. 
\item[2.]  The lengths  $\lambda^\mu_1,\lambda^\mu_2$ increase continuously for $\mu \in M^\ell$. In particular, we have $\lambda^\mu_1,\lambda^\mu_2>1$ for $\mu > \mu_\ell^{\rm us}$ and  $\lambda^\mu_1,\lambda^\mu_2<1$ for $\mu < \mu_\ell^{\rm us}$.
\item[3.] The angle $\alpha^\mu$ corresponding to $\lambda^\mu_1,\lambda^\mu_2$ given by the relations  \eqref{eq: basic constraints} and \eqref{alphars} satisfies $\alpha^\mu \in (\alpha^{\rm ch}_\ell,\alpha^{\rm ru})$ for all $\mu \in M^\ell$.
\item[4.] Whenever $v_2''(1) \neq 6v_3''(2\pi/3)$, the radius $\rho^\mu$ corresponding to $\lambda^\mu_1,\lambda^\mu_2$ given by relation \eqref{eq: basic constraints} is continuously increasing or decreasing for $\mu \in M^\ell$, respectively, depending on whether $v_2''(1) < 6v_3''(2\pi/3)$ or $v_2''(1) > 6v_3''(2\pi/3)$.
\end{itemize}

\end{proposition}

Properties 1 and 2 imply that that the nanotubes show elastic 
response  for small extension and compression. Property 3  
reconfirms  that neither the polyhedral  nor the rolled-up
configuration is a local minimizer of the energy,  for all $\mu$
  near $\mu^{\rm us}_\ell$.   Finally, Property 4 implies that under stretching or
compressing the radius of the nanotube \PPP changes whenever $v_2''(1) \neq 6v_3''(2\pi/3)$. \EEE In particular, if $v_2''(1) > 6v_3''(2\pi/3)$, the radius of the nanotube decreases as changing the angles is energetically more convenient.

Notice that Theorem \ref{th: main3} provides a stability result only for the case of expansion $\mu \ge \mu^{\rm us}_\ell$ and for values $\mu$ near $\mu^{\rm us}_\ell$. The situation for compression is more subtle from an analytical point of view and our proof techniques do not apply in this case. However, we expect stability of nanotubes also for small compression and refer to \cite{MMPS}  for some numerical results in this direction.  Let us complete the picture in the tension regime by \BBB
 discussing briefly the fact that for larger stretching cleavage \EEE along a section is energetically favored. More precisely, we have the following result.

\begin{theorem}[Fracture]\label{fracture} 
Let \RRR $\mathcal{H}_\mu$ \EEE be the configuration
\begin{align*}
x_{i,k}^{j,l} = \begin{cases}   \bar{x}_{i,k}^{j,l} &  \RRR j \EEE \in [0,m/2) + m\Zz, \\ 
\bar{x}_{i,k}^{j,l} + m(\mu - \mu_\ell^{\rm us})  & \text{else}
\end{cases}
\end{align*}
for $i=1,\ldots,\ell$ and  $k,l \in \lbrace 0,1 \rbrace$, where $\bar{x}_{i,k}^{j,l}$ denote the atomic positions of the configuration $\mathcal{F}_{1,1,\mu^{\rm us}_\ell}$ \RRR (see Proposition \ref{basiczigzag}(d)). \EEE  Then there \BBB are an open interval $M^\ell$ containing $\mu_\ell^{\rm us}$ and a constant $c>0$ only depending on $v_2$ and $v_3$ such that for  all $\mu  \in M^\ell$, $\mu \ge \mu^{\rm frac}_{\ell,m} := \mu_\ell^{\rm us} + c/\sqrt{m}$, \EEE one has  $E(\mathcal{H}_\mu) < E(\mathcal{F}^*_\mu)$.
\end{theorem}
Notice that the configuration $\mathcal{H}_\mu$ corresponds to a brittle nanotube cleaved along a cross-section. The energy is given by $E(\mathcal{H}_\mu) = E(\mathcal{F}_{1,1,\mu^{\rm us}_\ell}) + \NNN 4\ell \EEE $ since in the configuration $\mathcal{H}_\mu$ there are $4\ell$ less active bonds {\BBB per $n$-cell}  than in $\mathcal{F}_{1,1,\mu^{\rm us}_\ell}$. Moreover,  $\mathcal{H}_\mu$ is a stable configuration in the sense of Theorem \ref{th: main3} \BBB for all $\mu \ge \mu^{\rm
   us}_\ell$, \EEE which can be seen by applying Theorem \ref{th: main3} separately on the two parts of $\mathcal{H}_\mu$, \RRR consisting of the  points $x_{i,k}^{j,l}$ with $j < m/2$ and $j \ge m/2$, respectively. \EEE

  As mentioned, nanotubes are long structures. In particular, $m$
 should be expected to be many orders of magnitude larger than
 $\ell$. The case of large $m$ is hence a sensible one and for  $m$ large enough we have $\mu^{\rm frac}_{\ell,m} <
 \mu^{\rm crit}_\ell$, with $\mu^{\rm crit}_\ell$ from Theorem
 \ref{th: main3}. Hence, by  combining Theorem \ref{th: main3}
 with Theorem \ref{fracture}, for {\it all} $\mu \ge \mu^{\rm
   us}_\ell$   we obtain a stability result for an elastically
 stretched or cleaved nanotube, respectively.

The proof of Theorem \ref{fracture} is elementary and relies on the fact that the difference of the energy associated to $\mathcal{F}^*_\mu$ and $\mathcal{H}_\mu$ can be expressed by
\begin{align*}
E(\mathcal{H}_\mu) -   E(\mathcal{F}^*_\mu)& = 4\ell + E(\mathcal{F}_{1,1,\mu^{\rm us}_\ell}) - E(\mathcal{F}^*_\mu) = 4\ell + E_{\rm min}(\mu^{\rm us}_\ell) - E_{\rm min}(\mu)\\
& = 4\ell -  \frac{1}{2}\frac{d^2}{d^2\mu}E_{\rm min}(\mu^{\rm us}_\ell) (\mu -\mu^{\rm us}_\ell)^2  + {\rm O}((\mu -\mu^{\rm us}_\ell)^3)\\
&  \le   4\ell -    \frac{1}{4}  \frac{d^2}{d^2\mu}E_{\rm min}(\mu^{\rm us}_\ell) (\mu -\mu^{\rm us}_\ell)^2  \le   4\ell  - m\ell c (\mu -\mu^{\rm us}_\ell)^2
\end{align*}
for $\mu$ in a small neighborhood around $\mu^{\rm us}_\ell$, where we
used Property 1 in Proposition \ref{th: main2} \RRR and $n = 4m\ell$. \EEE 

We close the section by noting that the scaling of $\mu^{\rm frac}_{\ell, m} - \mu^{\rm us}_\ell$ in $m$ is typical for atomistic \PPP  systems with pairwise interactions of Lennard-Jones type \EEE and has also been obtained in related models, cf. \cite{Braides-Lew-Ortiz:06, FriedrichSchmidt:2011, FriedrichSchmidt:2014.1}.

%
%

\section{Existence and stability:  Proof of Theorem \ref{th: main1} and Theorem \ref{th: main3}}\label{sec: main proof}

In this section we consider small perturbations $\tilde{\mathcal{F}}$ of configurations in $\mathscr{F}(\mu)$ with the same bond graph, {\BBB as defined in \eqref{eq: bc}}. The atomic positions of $\tilde{\mathcal{F}}$ will be indicated by $x_{i,k}^{j,l}$ and are labeled as for a configuration $\mathscr{F}(\mu)$, \PPP cf.~Proposition \EEE \ref{basiczigzag}(d). We first introduce some further notation needed for the proof of our main result. In particular, we introduce a \emph{cell energy} corresponding to the energy contribution of a specific {\BBB basic} cell.

\noindent
\textbf{Centers and dual centers.} We introduce the \emph{cell centers}
\begin{align}\label{eq: centers}
z_{i,j,k} = \frac{1}{2}\Big(x_{i,k}^{j,0} +  x_{i,k}^{j,1}\Big)
\end{align}
and the \emph{dual cell centers}
$$z^{\rm dual}_{i,j,k} = \frac{1}{2}\Big(x_{i,k}^{j,1} +  x_{i,k}^{j+1,0}\Big).$$
Note that for a configuration in $\mathcal{F}(\mu)$ for fixed $j$ the $2\ell$ points $z_{i,j,0}$ and $z^{\rm dual}_{i,j-1,1}$ for $i=1,\ldots,\ell$ lie in a plane perpendicular to $e_1$. Likewise,  $z_{i,j,1}$ and $z^{\rm dual}_{i,j,0}$ for $i=1,\ldots,\ell$ lie in a plane perpendicular to $e_1$. 

\noindent
\textbf{Cell energy.} The main strategy of our proof will be to reduce the investigation of \eqref{min2}  to a cell problem. {\BBB In order to correctly   capture the contribution of all bond lengths and angles to the energy, } it is not enough to consider a hexagon as a basic cell, but two additional atoms have to be taken into account.

 \begin{figure}[htp]
\begin{center}
\pgfdeclareimage[width=0.7\textwidth]{cell}{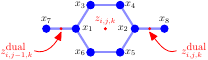}
    \pgfuseimage{cell}
\caption{ Notation for the points and the centers in the basic cell.}
\label{cell}
\end{center}
\end{figure}
Let be given a center $z_{i,j,k}$ and number the atoms of the corresponding hexagon by $x_1 = x_{i,k}^{j,0}$, $ x_2 = x_{i,k}^{j,1}$ and the remaining  clockwisely  by  $x_3,x_4,x_5,x_6$ as indicated in  Figure \ref{cell}, {\BBB such that $x_3$ is consecutive to $x_1$,} see also \eqref{kink}  below. Additionally, the atoms bonded to $x_1$ and $x_2$, respectively, which are not contained in the hexagon, are denoted by  $x_7$ and  $x_8$. Note that $z_{i,j-1,k}^{\rm dual} = (x_7 + x_1)/2$ and $z_{i,j,k}^{\rm dual} = (x_2 + x_8)/2$. For $i=1,\ldots,6$ we define the bondlengths $b_i$ as indicated in  Figure \ref{cellangles}    and $b_7 = |x_1 - x_7|$,  $b_8 = |x_2 - x_8|$, where
$$2|z_{i,j-1,k}^{\rm dual} - x_1| = b_7, \ \ \ \  2|z_{i,j,k}^{\rm dual} - x_2| = b_8.$$
 By $\varphi_i$ we denote the interior angle of the hexagon at
 $x_i$. By $\varphi_7,\varphi_8$ we denote the remaining two angles at
 $x_1$ and by $\varphi_9,\varphi_{10}$ we denote the remaining two
 angles at $x_2$, see again  Figure \ref{cellangles}.

 \begin{figure}[htp]
\begin{center}
\pgfdeclareimage[width=0.45\textwidth]{cellangles}{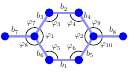}
    \pgfuseimage{cellangles}
\caption{ Notation for the bond lengths and angles in the basic cell.}
\label{cellangles}
\end{center}
\end{figure}
We define the \emph{cell energy} 
by 
\begin{align} 
E_{{\rm cell}}(z_{i,j,k}) & = \frac{1}{4} \big(v_2(b_1) + v_2(b_2) \big) + \frac{1}{2}\sum^6_{\PPP h\EEE=3} v_2(b_{\PPP h\EEE})    + \frac{1}{4} \big( v_2(b_7) + v_2(b_8) \big) \nonumber\\ 
&   + v_3(\varphi_1) + v_3(\varphi_2) + \frac{1}{2}\sum^6_{\PPP h\EEE=3}
v_3(\varphi_{\PPP h\EEE})  +\frac{1}{2} \sum^{10}_{\PPP h\EEE=7}
v_3(\varphi_{\PPP h\EEE}). \label{eq: cell} 
\end{align}
\PPP Notice that the cell energy is a function depending on the bond lengths and angles in the cell. However, as we identify each cell with its center $z_{i,j,k}$, for simplicity we use the notation $E_{{\rm cell}}=E_{{\rm cell}}(z_{i,j,k})$. Furthermore, also for notational convenience we do not put    indices $i,j,k$ on bond lengths and angles.  \EEE
To derive convexity properties of $E_{{\rm cell}}$ it is convenient to take also the contribution of the angles $\varphi_7, \ldots, \varphi_{10}$ into account.  Observe that 
\begin{align}\label{eq: sumenergy}
E(\tilde{\mathcal{F}}) = \sum_{i=1}^\ell \sum_{j=1}^m \sum_{k=0,1} E_{{\rm cell}}(z_{i,j,k}). 
\end{align}
Indeed, each bond not (approximately) parallel to $e_1$ is contained exactly in two cells. Each bond  (approximately) parallel to $e_1$ is contained in four cells, twice in form of a bond in a hexagon, once as a bond left of a hexagon and once as a bond right of a hexagon. Moreover, angles with index $\lbrace 1,2\rbrace$ are contained exactly in one cell and  angles with index $\lbrace 3,\ldots,10\rbrace$ are contained in exactly two cells.

\noindent
\textbf{Symmetrization of cells.} 
\BBB A basic cell is a configuration of eight points of $\mathbb{R}^3$.  By $\boldsymbol{x}_{\rm kink}^\ell\in\mathbb{R}^{3\times 8}$ we denote the \emph{unstretched kink configuration}:  a basic   cell as found in the unstretched configuration  $\mathcal{G}_{\alpha_\ell^{\rm us}}$ from Section \ref{sec: mainresults}, see \eqref{kink} below for the exact definition. \BBB  
\PPP Notice that the coordinates given in \eqref{kink} correspond to a convenient choice of a new reference orthonormal system in $\mathbb{R}^3$. \EEE
  Indeed, \PPP consider \EEE a cell of the nanotube $\mathcal{G}_{\alpha_\ell^{\rm us}}$, where the eight points are ordered from  $x_1$ to $x_8$ according to the convention of the previous subsection (see Figure \ref{cell}), \RRR in particular  the points  $x_3,x_4,x_5,x_6$ are numbered clockwisely    with respect to an observer lying in the interior of the tube. \EEE  \PPP We fix \EEE a new reference coordinate system as follows: we let the center of the cell be the origin,  $e_1$ (axis direction) be the direction of $x_2-x_1$,
 $e_2$  the direction of $x_3-x_6$,  and $e_3=e_1\wedge e_2$. Sometimes we will write $\mathbb{R}^2\times\{0\}$ for the plane generated by  $e_1,e_2$. If $\boldsymbol{x}\in \mathbb{R}^{3\times 8}$ denotes a generic cell, possibly after a rigid motion we may always assume that, {\BBB with respect to the new reference system}, the second and third components of \RRR $(x_1+x_7)/2$, $(x_2+x_8)/2$ are zero \EEE and the points $x_4$, $x_5$ lie in a plane parallel to $\Rz^2 \times \lbrace 0 \rbrace$.

A key step in our analysis will be to show that the minimization of the cell energy \eqref{eq: cell} can be reduced to a special situation with high symmetry. To this end, we introduce the \emph{symmetrization} of a cell.
For  $y = (y^1, y^2,y^3) \in \Rz^3$ we let $r_1 (y) := (-y^1,y^2,y^3)$ and $r_2 (y) := (y^1,-y^2,y^3)$. For the generic cell $\boldsymbol{x}= (x_1,\ldots,x_8) \in \Rz^{3 \times 8}$ we define the reflections
\begin{equation}\label{reflexion}
\begin{aligned}
S_1(\boldsymbol{x})& = ( r_2(x_1) \, | \,  r_2(x_2)   \, | \, r_2(x_6)  \, | \,  r_2(x_5) \, | \, r_2(x_4) \, | \, r_2(x_3) \, | \, r_2( x_7) \, | \, r_2(  x_8)),\\ 
S_2(\boldsymbol{x}) &= ( r_1( x_2) \, | \, r_1( x_1) \, | \,  r_1(x_4 )\,  | \,  r_1( x_3)  \, | \,r_1( x_6) \, | \,  r_1(x_5) \, | \, r_1( x_8) \, | \, r_1( x_7)).
\end{aligned}
\end{equation}
$S_1$ interchanges the pair of points $(x_3, x_6)$ and $(x_4,x_5)$, and changes the sign of the second components of all points. On the other hand, $S_2$ interchanges the pair of points $(x_1, x_2)$, $(x_3,x_4)$, $(x_5,x_6)$, and  $(x_7,x_8)$, and changes the    sign of the first components of all points. \EEE

{\BBB
We  let 
\begin{equation}\label{s1s2}
\boldsymbol{x}_{S_1}:=\boldsymbol{x}_{\rm kink}^\ell+S_1(\boldsymbol{x}-\boldsymbol{x}_{\rm kink}^\ell),\quad
\boldsymbol{x}_{S_2} : = \boldsymbol{x}_{\rm kink}^\ell  + S_2(\boldsymbol{x}- \boldsymbol{x}_{\rm kink}^\ell).
\end{equation}
 If $\boldsymbol{x}$ is seen as  a perturbation of  $\boldsymbol{x}_{\rm kink}^\ell$, $\boldsymbol{x}_{S_1}$ \PPP(resp.~$\boldsymbol{x}_{S_2}$) \EEE is the \PPP \emph{reflected} \EEE perturbation with respect to the plane generated by ${e_1, e_3}$ \PPP (resp.~$e_2, e_3$). \EEE The symmetry of the configurations implies therefore  $E_{\rm cell}(\boldsymbol{x}_{S_2} ) =  E_{\rm cell}(\boldsymbol{x}_{S_1} ) =  E_{\rm cell}(\boldsymbol{x})$.

}

  We define \PPP the \emph{symmetrized} perturbations \EEE
\begin{subequations}\label{reflection2}
  \begin{align}
    \boldsymbol{x}' :=  \boldsymbol{x}_{\rm kink}^\ell + \frac{1}{2} \Big((\boldsymbol{x}  - \boldsymbol{x}_{\rm kink}^\ell)   + S_1(\boldsymbol{x}- \boldsymbol{x}_{\rm kink}^\ell) \Big),\label{reflection2-a}\\
    \mathcal{S}(\boldsymbol{x}): = \boldsymbol{x}_{\rm kink}^\ell +
    \frac{1}{2} \Big((\boldsymbol{x}' - \boldsymbol{x}_{\rm
      kink}^\ell) + S_2(\boldsymbol{x}'- \boldsymbol{x}_{\rm
      kink}^\ell) \Big).\label{reflection2-b}
  \end{align}
\end{subequations}

We also introduce the \emph{symmetry defect}
\begin{align}\label{delta}
\Delta(z_{i,j,k}) := |\boldsymbol{x} - \boldsymbol{x}'|^2 + |\boldsymbol{x}' - \mathcal{S}(\boldsymbol{x})|^2.
\end{align}
\PPP Notice that for notational simplicity in \eqref{delta} we do not put indices $i,j,k$ on $\boldsymbol{x}$, $\boldsymbol{x}'$, and $\mathcal{S}(\boldsymbol{x})$. \EEE
{\BBB A property that we remark is that    for a basic cell $\boldsymbol{x}$ with center $z_{i,j,k}$ the quantity $|z^{\rm dual}_{i,j,k} - z^{\rm dual}_{i,j-1,k}|$   does not change when passing to $\mathcal{S}(\boldsymbol{x})$ \RRR since the second and third component of $z^{\rm dual}_{i,j,k}, z^{\rm dual}_{i,j-1,k}$ are assumed to be zero. \EEE  }
 Below we will see that the difference of the cell energy of $\mathcal{S}(\boldsymbol{x})$ and $\boldsymbol{x}$ can be controlled in terms of $\Delta(z_{i,j,k})$ due to strict convexity of the energy.

\noindent
\textbf{Angles between planes.} \PPP In the following we denote the plane through three points $p_1$, $p_2$, and $p_3$ by $\{p_1 p_2 p_3\}$, i.e., $$\{p_1 p_2 p_3\}:=\textrm{span}_{\Rz} \{p_1-p_2, p_3-p_2 \}.$$ Furthermore, for  each $y = x_{i,k}^{j,l}$  we denote by $y_1,y_2,y_3$  \EEE the three atoms that are bonded with $y$, where the three points are numbered such that $y_3 - y$ is (approximately) parallel to {\BBB the axis direction} $e_1$.  Let $\theta = \theta(x)\le\pi$ denote the angle between the planes defined by $\{y_3 y y_1\}$ and $\{y_3 y y_2\}$. More precisely, let $n_{13}$, $n_{23}$ denote  unit normal  vectors to the planes $\{y_3 y y_1\}$ and $\{y_3 y y_2\}$, respectively.  Then we have 
\begin{align}\label{eq: thetaangle}
\theta(\PPP y\EEE) = \max \big\{\pi - \arccos ( n_{13} \cdot n_{23}  ), \ \arccos ( n_{13} \cdot n_{23} ) \big\}
\end{align}
as represented in Figure \ref{angletheta}. With these preparations we will now define angles corresponding to centers and dual centers. Let  \PPP $z_{i,j,k} = \frac{1}{2}(x_{i,k}^{j,0} +  x_{i,k}^{j,1})$ be a center of a given \EEE hexagon. As before we denote the points of the hexagon by $x_1,\ldots,x_6$. By $\theta_l(z_{i,j,k})$ we denote the angle between the planes $\{x_1 x_3 x_4\}$ and $\{x_1 x_6 x_5\}$. By $\theta_r(z_{i,j,k})$ we denote the angle between the planes $\{x_3 x_4 x_2\}$ and $\{x_2 x_5 x_6\}$.  For a dual center $z^{\rm dual}_{i,j,k} = (x_{i,k}^{j,1} +  x_{i,k}^{j+1,0})/2$ we introduce $\theta_l(z^{\rm dual}_{i,j,k}) = \theta(x_{i,k}^{j,1})$ and $\theta_r(z^{\rm dual}_{i,j,k}) = \theta(x_{i,k}^{j+1,0})$.   
 \begin{figure}[htp]
\begin{center}
\pgfdeclareimage[width=0.5\textwidth]{angletheta}{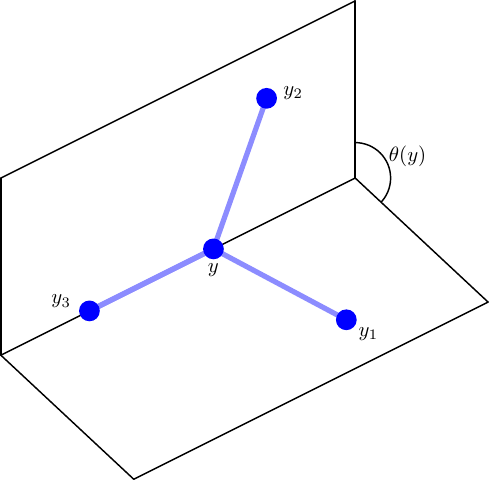}
    \pgfuseimage{angletheta}
\caption{ The angle between the planes \PPP $\{y_3 y y_1\}$ \EEE and \PPP $\{y_3 y y_2\}$ \EEE is denoted by \PPP $\theta(y)$. \EEE}
\label{angletheta}
\end{center}
\end{figure}

In Section \ref{sec: angles} we prove the following lemma which provides a linear control for the oscillation of plane angles of a perturbed configuration $\tilde{\mathcal{F}}$  \RRR with respect to those of a configuration in $\mathscr{F}(\mu)$ \EEE  in terms of the symmetry defect from \eqref{delta}.

\begin{lemma}[Symmetry defect controls angle defect]\label{lemma: sum}
There is a universal constant $c>0$ such that for $\eta>0$ small enough \RRR for  all \PPP $\tilde{\mathcal{F}}\in \mathscr{P}_\eta(\mu) $ \EEE with $\Delta(z_{i,j,k}) \le \eta$ for all centers $z_{i,j,k}$ \EEE we have
\begin{align*}
\sum_{j=1}^m\sum_{i=1}^\ell \sum_{k=0,1}\Big(\theta_l(z_{i,j,k}) + \theta_l(z^{\rm dual}_{i,j,k}) & + \theta_r(z_{i,j,k}) + \theta_r(z^{\rm dual}_{i,j,k})  \Big) \\ & \le    4m(2\ell - 2)\pi + c\sum_{j=1}^m\sum_{i=1}^\ell\sum_{k=0,1} \Delta(z_{i,j,k}).
\end{align*}

\end{lemma}

Note that the sum on the left equals exactly  $4m(2\ell - 2)\pi$ if $\tilde{\mathcal{F}} \in \mathscr{F}(\mu)$.

\noindent
\textbf{Reduced energy.} A key step in our analysis will be to show that the minimization of the cell energy \eqref{eq: cell} can be reduced to a special situation with high symmetry. As represented in Figure \ref{reducedenergy}, this  corresponds to the conditions
\begin{equation}\label{sym-assumption} 
  \begin{aligned}
    &b_1 = b_2 = \lambda_1, \ \  \ \ b_3 = b_4 = b_5 =  b_6 = \lambda_2, \ \  \ \  b_7 = b_8 = \lambda_3,\\
    &z^{\rm dual}_{i,j,k} - z^{\rm dual}_{i,j-1,k} = {\BBB \widetilde\mu }e_1, \ \ \ \  x_2-x_1 = \lambda_4 e_1,\\
    &\varphi_1 = \varphi_2 = \beta, \ \  \ \ \varphi_3 = \varphi_4 = \varphi_5 = \varphi_6 = \alpha_1, \ \ \ \ \varphi_7 = \varphi_8 = \varphi_9=  \varphi_{10} = \alpha_2, \\
    &{\theta_l}(z_{i,j,k}) = {\theta_r}(z_{i,j,k}) = \gamma_1, \ \ \ \
    \ {\theta_l}(z^{\rm dual}_{i,j,k}) = {\theta_r}(z^{\rm
      dual}_{i,j-1,k}) = \gamma_2
  \end{aligned}
\end{equation}
with $\lambda_1,\lambda_2,\lambda_3 \in (0.9,1.1)$, $\lambda_4 \in (0.9,3.3)$, ${\BBB \widetilde\mu} \in (2.6,3.1)$\PPP, \EEE  $\alpha_1,\alpha_2,\beta \in (\arccos(-0.4),\arccos(-0.6))$, $\gamma_1,\gamma_2 \in [\frac{3}{4}\pi,\pi]$.  Note that ${\theta_r}(z^{\rm dual}_{i,j-1,k})= \theta(x_1)$ and ${\theta_l}(z^{\rm dual}_{i,j,k})= \theta(x_2)$ with the angles introduced in \eqref{eq: thetaangle}. {\BBB The notation $\tilde \mu$ is reminiscent of the fact that we have indeed $\widetilde \mu=\mu$ for a basic cell of a nanotube in $\mathscr{F}(\mu)$.}   {\BBB Under \eqref{sym-assumption}},  arguing along the lines  of Proposition \ref{betaproperties},   we obtain 
\begin{align}\label{eq: constraint2}
\beta= \beta(\alpha_1,\gamma_1) =  2\arcsin\left(\sin\alpha_1\sin\frac{\gamma_1}{2}\right) = \beta(\alpha_2,\gamma_2) =   2\arcsin\left(\sin\alpha_2\sin\frac{\gamma_2}{2}\right).
\end{align}
 By elementary trigonometry, cf.  Figure \ref{reducedenergy},   we also get
\begin{align}\label{lambda4}
\lambda_4   = \lambda_1 - 2\lambda_2\cos\alpha_1.
\end{align}

We  now introduce the  \emph{symmetric energy} by 
\begin{align}\label{symmetric-cell}
\begin{split}
E_{\mu,\gamma_1,\gamma_2}^{{\rm sym}}(\lambda,\alpha_1,\alpha_2) &=  2v_2(\lambda)    + \frac{1}{2} {v}_2  \big(\mu/2 + \lambda\cos\alpha_1 \big) + \frac{1}{2} {v}_2  \big(\mu/2 + \lambda\cos\alpha_2 \big)  \\
& \ \ \ \  +  2   v_3(\alpha_1)  +  2 v_3(\alpha_2)   +  v_3(\beta(\alpha_1,\gamma_1)) +  v_3(\beta(\alpha_2,\gamma_2)).
\end{split}
\end{align}
{\BBB Notice that   $E_{\rm cell}(z_{i,j,k})=E^{\rm sym}_{\widetilde\mu,\gamma_1,\gamma_2}(\lambda,\alpha_1,\alpha_2)$ if the conditions \eqref{sym-assumption} hold with $\alpha_1=\alpha_2$,  $\gamma_1=\gamma_2$, \RRR $\lambda_1=\lambda_3= \mu/2 + \lambda\cos\alpha_1$, and $\lambda_2=\lambda$. 
In general, } we show that, up to a small perturbation, the symmetric energy {\BBB $E_{\widetilde\mu,\gamma_1,\gamma_2}^{{\rm sym}}$ delivers a lower bound for  $E_{\rm cell}$ for  cells satysfying \eqref{sym-assumption}}.
 \begin{figure}[htp]
\begin{center}
\pgfdeclareimage[width=0.7\textwidth]{reducedenergy}{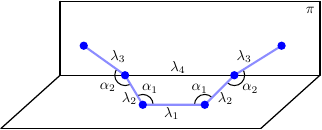}
    \pgfuseimage{reducedenergy}
\caption{Half of a cell configuration kinked at the plane $\pi$ and satisfying conditions \eqref{sym-assumption}. The other half of the cell configuration can be determined by symmetry with respect to the plane $\pi$.}
\label{reducedenergy}
\end{center}
\end{figure}

\begin{lemma}[Cell energy and symmetric energy]\label{lemma: sym-energy}
There exist a constant $c_0>0 $ and $\ell_0 \in \Nz$ only depending on $v_2$ and $v_3$  \RRR such that   for each \PPP $\tilde{\mathcal{F}}\in \mathscr{P}_\eta(\mu) $ \EEE  and all centers $z_{i,j,k}$ satisfying  conditions \eqref{sym-assumption} with $|\lambda_1 - 1| + |\lambda_3 - 1| \le \ell^{-4}$ and $|\gamma_1 - \gamma_2| \le \ell^{-2}$ \EEE  we have  
$$
E_{{\rm cell}}(z_{i,j,k}) \ge E_{{\BBB \widetilde\mu},\gamma_1,\gamma_2}^{{\rm sym}}(\lambda_2,\alpha_1,\alpha_2) - c_0 \ell^{-4} (\gamma_1 - \gamma_2)^2.
$$ 
\end{lemma}

  This lemma will be proved in Section \ref{sec: reduced-energy}. The idea in the proof  is to express $\lambda_3$ in terms of the relations \eqref{sym-assumption} and \eqref{lambda4} to find $\lambda_3 = {\BBB \widetilde\mu} - \lambda_1 + 2\lambda\cos\alpha_1 + {\rm O}( (\gamma_1 - \gamma_2)^2)$, where we set $\lambda=\lambda_2$. Here the term $ {\rm O}( (\gamma_1 - \gamma_2)^2)$ appears as the points $x_7,x_1,x_2,x_8$ in general do not lie on a line. Likewise, we obtain  $\lambda_1 = {\BBB  \widetilde\mu} - \lambda_3 + 2\lambda\cos{\BBB \alpha_2} + {\rm O}( (\gamma_1 - \gamma_2)^2)$. Finally, we use $v_2(\lambda_1) + v_2(\lambda_3) \ge 2v_2( (\lambda_1 + \lambda_3)/2 )$ by convexity of $v_2$.  
  
    We also introduce the \emph{reduced energy}
\begin{align}\label{red}
E_{\rm red}(\mu,\gamma_1,\gamma_2) &= \min\lbrace E_{\mu,\gamma_1,\gamma_2}^{{\rm sym}}(\lambda,\alpha_1,\alpha_2)| \ \lambda \in (0.9,1.1), \ \alpha_1,\alpha_2 \in (\arccos(-0.4),\arccos(-0.6)) \rbrace.
\end{align}
Since $E_{\mu,\gamma_1,\gamma_2}^{{\rm sym}}$ is symmetric in  $(\alpha_1,\gamma_1)$ and $(\alpha_2,\gamma_2)$, we observe that $E_{\rm red}$ is symmetric in $\gamma_1$ and  $\gamma_2$, \PPP i.e., \EEE $E_{\rm red}(\mu,\gamma_1,\gamma_2)  = E_{\rm red}(\mu,\gamma_2,\gamma_1)$.  The following result, which is proved in Section \ref{sec: reduced-energy}, collects the fundamental properties of $E_{\rm red}$.

\begin{proposition}[Properties of $E_{\rm red}$]\label{th: mainenergy}
There exists $\ell_0 \in \Nz$ and for each $\ell \ge \ell_0$ there are open intervals $M^\ell$, $G^\ell$ only depending on $v_2$,  $v_3$ and $\ell$ with $\mu^{\rm us}_\ell \in M^\ell$, $\gamma_\ell \in G^\ell$ \PPP \emph{(}where we recall that $\mu^{\rm us}_\ell$ and $\gamma_\ell$ were defined in \eqref{eq:muellneu} and \eqref{gamma}, respectively\emph{)} \EEE such that the following holds:   
\begin{itemize}
\item[1.] \BBB (Unique minimizer) For each   $(\mu, \gamma_1,\gamma_2) \in M^\ell \times G^\ell \times G^\ell$  there exists a unique triple $(\lambda^\mu, \alpha^\mu_1,\alpha^\mu_2)$ solving the minimization problem \eqref{red}. {\BBB   {Moreover, $\alpha_1^\mu=\alpha_2^\mu$ if $\gamma_1=\gamma_2$.}} \EEE
\RRR (For simplicity, the dependence of the triple on $\gamma_1,\gamma_2$ is not included in the notation.)\EEE
\item[2.] (Strict convexity) $E_{\rm red}$ is strictly convex on  $M^\ell \times G^\ell \times G^\ell$, in particular there is a  constant $c_0'>0$ only depending on $v_2$ and $v_3$  such that
$$E_{\rm red}(\mu,\gamma_1,\gamma_2) \ge E_{\rm red}(\mu,\bar{\gamma},\bar{\gamma}) + c_0'\ell^{-2} (\gamma_1 - \gamma_2)^2$$ with $\bar{\gamma} = (\gamma_1 + \gamma_2)/2$ for all $\mu \in M^\ell$ and $\gamma_1,\gamma_2 \in G^\ell$.
\item[3.] (Monotonicity in $\gamma$) For each $\mu \in M^\ell$, the mapping $g(\gamma):= E_{\rm red}(\mu,\gamma,\gamma)$ is decreasing on $G^\ell$ with $|g'(\gamma)| \le C\ell^{-3}$ for all $\gamma \in G^\ell$ for some $C>0$ depending only on  $v_3$.
\item[4.] (Monotonicity in $\mu$) The mapping $h(\mu):= E_{\rm red}(\mu,\gamma_\ell,\gamma_\ell)$ is strictly convex on $M^\ell$ with $h''(\mu^{\rm us}_\ell)>0$ and strictly increasing on $M^\ell \cap \lbrace \mu \ge \mu^{\rm us}_\ell \rbrace$. 
\item[5.] \BBB (Minimization) For each   $\mu \in M^\ell$ and $\gamma_1 = \gamma_2 = \gamma_\ell$, letting $\lambda_1^\mu = \mu/2 + \lambda^\mu \cos\alpha^\mu_1$ and $\lambda_2^\mu = \lambda^\mu$ with $\lambda^\mu$ and $\alpha^\mu_1$ from 1., \EEE the configuration $\mathcal{F}_{\lambda_1^\mu,\lambda_2^\mu,\mu}$ is the unique minimizer of the problem \eqref{min2} with 
$$E(\mathcal{F}_\mu^*)  = E(\mathcal{F}_{\lambda_1^\mu,\lambda_2^\mu,\mu}) = 2m\ell E_{\rm red}(\mu,\gamma_\ell,\gamma_\ell).$$
\end{itemize}

\end{proposition}

\RRR
\noindent\textbf{Proof of Theorem \ref{th: main1} and Theorem \ref{th: main3}.}  We postpone the proofs of the auxiliary results Lemma \ref{lemma: sum}, Lemma \ref{lemma: sym-energy}, and Proposition \ref{th: mainenergy}  to the next sections and now proceed with the proof of Theorem \ref{th: main1} and Theorem \ref{th: main3}. For the proof of Proposition \ref{th: main2} we refer to Section \ref{sec: reduced-energy}. Moving from the properties of the reduced energy $E_{\rm red}$, we directly obtain Theorem \ref{th: main1}.

\begin{proof}[Proof of  Theorem \ref{th: main1}]
Theorem \ref{th: main1} follows from Property 5 of Proposition \ref{th: mainenergy}.
\end{proof}

We denote the unique minimzer again by $\mathcal{F}_\mu^*$ and recall the definition of small perturbations $\mathscr{P}_\eta(\mu)$ in \eqref{eq: bc}.  Based on the  properties of the reduced energy $E_{\rm red}$, \EEE we are able to show that, up to a linear perturbation in terms of the symmetry defect $\Delta$ defined in \eqref{delta}, $E_{\rm red}$ bounds the cell energy $E_{\rm cell}$ from below. More precisely, we have the following.

\begin{theorem}[Energy defect controls symmetry defect]\label{th: Ered}
There exist   $C>0 $ and $\ell_0 \in \Nz$ only depending on $v_2$ and $v_3$, and for each $\ell \ge \ell_0$ there are $\eta_\ell > 0$ and an open interval $M^\ell$ containing $\mu^{\rm us}_\ell $ such that for all $\mu \in M^\ell$,  $\tilde{\mathcal{F}} \in  \mathscr{P}_{\eta_\ell}(\mu)$, and centers $z_{i,j,k}$   we have
\begin{align*}
 E_{{\rm cell}}(z_{i,j,k}) \ge  E_{\rm red}\big(| z^{\rm dual}_{i,j,k} - z^{\rm dual}_{i,j-1,k}| , \bar{\theta}(z_{i,j,k}), \bar{\theta}(z_{i,j,k})\big)   + C \ell^{-2}\Delta(z_{i,j,k}),
\end{align*}
where $\bar{\theta}(z_{i,j,k}) := \big(\theta_l(z_{i,j,k}) + \theta_r(z_{i,j,k}) + \theta_l(z^{\rm dual}_{i,j,k}) + \theta_r(z^{\rm dual}_{i,j-1,k}) \big)/4$.
\end{theorem}

\RRR
 We postpone the proof of  Theorem \ref{th: Ered} to Section \ref{sec: cellenery} and close this section with the proof of our main stability result Theorem \ref{th: main3}.  \EEE

 \begin{proof}[Proof of  Theorem \ref{th: main3}]
Let $M^\ell$ be an open interval containing $\mu^{\rm us}_\ell$ such
that Proposition \ref{th: mainenergy} and  Theorem \ref{th: Ered} hold
for all $\mu \in M^\ell$ and let $G^\ell$ be the interval from
Proposition \ref{th: mainenergy}. Then choose $\mu^{\rm crit}_\ell >
\mu^{\rm us}_\ell$ such that $[\mu^{\rm us}_\ell, \mu^{\rm crit}_\ell]
\subset \subset M^\ell$.  Let  $\ell \ge \ell_0$ and  $\mu \in
[\mu^{\rm us}_\ell, \mu^{\rm crit}_\ell]$ be given. Consider a
nontrivial perturbation $\tilde{\mathcal{F}} \in
\mathscr{P}_{\eta_\ell}(\mu)$ with $\eta_\ell$ as in Theorem \ref{th:
  Ered}. We denote the atomic positions by $x_{i,k}^{j,l}$ and the
centers by $z_{i,j,k}$, $z_{i,j,k}^{\rm dual}$ as introduced at the
beginning of the section, see \eqref{eq: centers} and  Figure \ref{cell}. Define 
\begin{align}\label{thetabar}
\bar{\theta}(z_{i,j,k}) = \frac{1}{4}\big(\theta_l(z_{i,j,k}) + \theta_r(z_{i,j,k})+\theta_l(z^{\rm dual}_{i,j,k}) + \theta_r(z^{\rm dual}_{i,j-1,k}) \big)
\end{align}
and also  
$$\bar{\mu} = \frac{1}{2m\ell}\sum_{j=1}^m\sum_{i=1}^\ell \sum_{k=0,1} |z^{\rm dual}_{i,j,k} - z^{\rm dual}_{i,j-1,k}|,  \  \ \ \ \ \bar{\theta} = \frac{1}{2m\ell}\sum_{j=1}^m\sum_{i=1}^\ell \sum_{k=0,1} \bar{\theta}(z_{i,j,k}).$$
Possibly passing to a smaller $\eta_\ell$, we get $|z^{\rm dual}_{i,j,k} - z^{\rm dual}_{i,j-1,k}| \in M^\ell$ and $\bar{\theta}(z_{i,j,k}) \in G^\ell$ for all $i,j,k$. By Theorem \ref{th: Ered}   we have for each cell 
\begin{align}\label{mainproof1}
E_{{\rm cell}}(z_{i,j,k}) \ge  E_{\rm red}\Big( |z^{\rm dual}_{i,j,k} - z^{\rm dual}_{i,j-1,k}|, \bar{\theta}(z_{i,j,k}),  \bar{\theta}(z_{i,j,k}) \Big)  + C \ell^{-2}\Delta(z_{i,j,k}) 
\end{align}
if $\ell_0$ is chosen sufficiently large. Then, taking the sum over all cells and using  Property {\it 2.} of Proposition \ref{th: mainenergy}, we get by \eqref{eq: sumenergy}
\begin{align*}
E(\tilde{\mathcal{F}}) &= \sum_{i=1}^\ell \sum_{j=1}^m \sum_{k=0,1}  E_{{\rm cell}}(z_{i,j,k})\ge 2m\ell  E_{\rm red}(\bar{\mu}, \bar{\theta},\bar{\theta}) + C \ell^{-2}\sum_{i=1}^\ell \sum_{j=1}^m \sum_{k=0,1} \Delta(z_{i,j,k}).
\end{align*}
\RRR Possibly passing to a smaller $\eta_\ell$, we can assume that $\Delta(z_{i,j,k}) \le \eta$ for all centers with $\eta$ from Lemma \ref{lemma: sum}. Then \EEE using Lemma \ref{lemma: sum} and recalling \eqref{thetabar} we find 
$$\bar{\theta} \le \frac{1}{8m\ell}\Big( 4m(2\ell - 2)\pi + C\sum_{j=1}^m\sum_{i=1}^\ell\sum_{k=0,1} \Delta(z_{i,j,k}) \Big) \le   \gamma_\ell + \frac{c}{2m\ell}\sum_{j=1}^m\sum_{i=1}^\ell\sum_{k=0,1} \Delta(z_{i,j,k}),$$
where in the last step we have used the fact that $\gamma_\ell = \pi(1-1/\ell)$, see \eqref{gamma}. This together with  Property 3 of Proposition \ref{th: mainenergy} yields   
\begin{align*}
 E(\tilde{\mathcal{F}})   \ge 2m\ell  E_{\rm red}(\bar{\mu}, \gamma_\ell,\gamma_\ell) + \big(C \ell^{-2} - C'\ell^{-3} \big)\sum_{j=1}^m \sum_{i=1}^\ell \sum_{k=0,1} \Delta(z_{i,j,k})
\end{align*}
for some $C'>0$ only depending { on $v_3$}.  Recalling the constraint in definition \eqref{eq: bc}, we get for fixed $i$ and $k$ that
$$m\mu = L^\mu_m = \Big|\sum_{j=1}^m z^{\rm dual}_{i,j,k} - z^{\rm dual}_{i,j-1,k} \Big| \le  \sum_{j=1}^m |z^{\rm dual}_{i,j,k} - z^{\rm dual}_{i,j-1,k}| $$
and therefore, by taking the sum over all $i$ and $k$, we get $\bar{\mu} \ge \mu \ge \mu^{\rm us}_\ell$. Then we  derive by Property 4 and 5  of Proposition \ref{th: mainenergy}
\begin{align}
E(\tilde{\mathcal{F}}) &\ge 2m\ell E_{\rm
  red}(\mu,\gamma_\ell,\gamma_\ell) + C''\ell^{-2}\sum_{i=1}^\ell
\sum_{j=1}^m \sum_{k=0,1} \Delta(z_{i,j,k})\nonumber \\ &= E(\mathcal{F}_\mu^*) + C''\ell^{-2}\sum_{i=1}^\ell \sum_{j=1}^m \sum_{k=0,1} \Delta(z_{i,j,k}) \label{mainproof3}
\end{align}
for $\ell_0$ sufficiently large and a possibly smaller constant $C''>0$. Note that in this step of the proof we have fundamentally used that $\mu \ge \mu^{\rm us}_\ell$,  \PPP i.e., \EEE the nanotube is stretched, so that a monotonicity argument can be applied. 

 It remains to confirm the strict inequality $E(\tilde{\mathcal{F}}) > E(\mathcal{F}_\mu^*)$. If $\Delta(z_{i,j,k})>0$ for some center $z_{i,j,k}$, this follows directly from the previous estimate. Otherwise, as  $\tilde{\mathcal{F}}$ is a nontrivial perturbation, one of the angles in \eqref{thetabar} or one of the lengths $|z^{\rm dual}_{i,j,k} - z^{\rm dual}_{i,j-1,k}|$ does not coincide with the corresponding mean value and then at least one of the inequalities \eqref{mainproof1}-\eqref{mainproof3} is strict due to the strict convexity and monotonicity of the mappings considered in Proposition \ref{th: mainenergy}.  \end{proof}

 \section{Symmetry defect controls angle defect: Proof of Lemma \ref{lemma: sum}}\label{sec: angles}

This short section  is devoted to the proof of Lemma \ref{lemma: sum}. Recall the definition of the centers in \eqref{eq: centers}, the angles \eqref{eq: thetaangle}, and the symmetry defect \eqref{delta}. 

\begin{proof}[Proof of Lemma \ref{lemma: sum}]
\PPP Let  $\tilde{\mathcal{F}}$ be a small perturbation of \EEE  $\mathcal{F}' \in \mathscr{F}(\mu)$, with $\Delta(z_{i,j,k}) \le \eta$ for all centers $z_{i,j,k}$. \EEE Due to the symmetry of the problem it suffices to show
$$ \sum_{j=1}^m\sum_{i=1}^\ell  \Big(\theta_l(z_{i,j,0}) + \theta_l(z^{\rm dual}_{i,j-1,1}) \Big) \le    m(2\ell - 2)\pi + c\sum_{j=1}^m\sum_{i=1}^\ell\sum_{k=0,1} \Delta(z_{i,j,k}). $$
For brevity we write $\theta'_i = \theta_l(z_{\frac{i+1}{2},j,0})$ for $i =1,3,\ldots, 2\ell-1$ and  $\theta'_i = \theta_l(z^{\rm dual}_{\frac{i}{2},j-1,1})$ for $i=2,4,\ldots,2\ell$. (Note that for convenience we do not include the index $j$ in the notation.)  

 Let $n_i,n_{i+1}$ be unit normal vectors as introduced before \eqref{eq: thetaangle} such that $n_i \cdot n_{i+1}$ is near $1$ and \PPP the smallest angle between them, which we denote by $\sphericalangle(n_i,n_{i+1})$, is given by $$\sphericalangle(n_i,n_{i+1}) = \pi - \theta'_i$$ 
 \EEE  for $i= 1,3,\ldots, 2\ell-1$. For a suitable ordering of $n_i$ and $n_{i+1}$  we then also obtain $\sphericalangle(n_{i},n_{i+1}) = \pi- \theta'_i$ for $i=2,4,\ldots,2\ell$.
Fix a center $x_0 \in \Rz^3$ and let $P$ be the $2\ell$-gon with vertices $v_i := x_0 + n_i$, $i=1,\ldots,2\ell$. Denote the interior angles accordingly by $\varphi_i$. Note that each edge of $P$ forms a triangle with $x_0$ with angles $\pi - \theta'_i$, $\psi_i^1$, and $\psi_i^2$, where $\psi_i^1$ is the angle at the vertex $v_i$ and $\psi_{i}^2$ is the angle  at $v_{i+1}$.  The key ingredient in the proof is now the observation that there exists a  universal $c>0$  such that
\begin{subequations}\label{psi-phi}
  \begin{align}
    &\psi_{i+1}^1 + \psi_{i}^2  - \varphi_{i+1} \le c\Delta(z_{\frac{i+1}{2},j,0}) + c\Delta(z_{\frac{i+3}{2},j,0})  , \label{psi-phi-a} \\
    &\psi_{i}^1 + \psi_{i-1}^2 - \varphi_{i} \le
    c\Delta(z_{\frac{i-1}{2},j,0}) + c\Delta(z_{\frac{i+1}{2},j,0})\label{psi-phi-b}
  \end{align}
\end{subequations}
for $i=1,3\ldots,2\ell-1$, {\BBB where it is understood that $\psi_0^2=\psi^2_{2\ell}$ and $z_{0,j,0}=z_{\ell,j,0}$}. We defer the derivation of this property to the end of the proof. Notice that  $\theta_i' = \psi_i^1 + \psi_i^2$ for $i=1,\ldots, 2\ell$ and that $\sum_{i=1}^{2\ell} \varphi_i \le (2\ell-2) \pi$ since $P$ is a $2\ell$-gon. We now obtain by \eqref{psi-phi}  
\begin{align*}
\sum_{i=1}^{2\ell}  \theta_i' = \sum_{i=1}^{2\ell} (\psi_i^1 + \psi_i^2) \le (2 \ell - 2)\pi + c \sum_{i=1}^\ell \Delta(z_{i,j,0}).
\end{align*}
The assertion then follows by taking the sum over all $j=1,\ldots,m$. 

It remains to confirm \eqref{psi-phi}. Fix $i=1,3,\ldots,2\ell-1$ and let $N_{i+1}$ be the plane containing the points $v_i, v_{i+1}$, and $v_{i+2}$. By $d_{i+1}$ we denote the distance of $x_0$ from $N_{i+1}$ and  by $n'_{i+1}$ the orthogonal projection of the vector $n_{i+1}$ onto $N_{i+1}$. Note that $d_{i+1} \le \delta$ for $\delta$ small, depending only on the choice of $\eta$, and that $|n_{i+1}'| = |n_{i+1}| + {\rm O}(d_{i+1}^2)$. The segments $v_{i+2} - v_{i+1},n'_{i+1}$ and $v_{i}-v_{i+1},n_{i+1}'$ enclose two angles, denoted by  $\hat{\psi}_{i+1}^1$ and $\hat{\psi}_{i}^2$, so that $\varphi_{i+1} = \hat{\psi}_{i+1}^1 + \hat{\psi}^2_{i}$. Observe that  $\hat{\psi}_{i+1}^1$ and $\hat{\psi}_{i}^2$ are the projections of $\psi_{i+1}^1$, $\psi_{i}^2$, respectively, onto  $N_{i+1}$.  For notational convenience suppose  $(v_{i+2} - v_{i+1}) \cdot n_{i+1}'>0$ and  $(v_{i+2} - v_{i+1}) \cdot n_{i+1}>0$, which holds  after possibly changing the signs of the vectors. Using that $(v_{i+2} - v_{i+1}) \cdot (n_{i+1} - n_{i+1}') = 0$ and recalling that $d_{i+1}$ is small, we calculate by a Taylor expansion 
\begin{align*}
\hat{\psi}_{i+1}^1 & = \arccos\Big( \frac{(v_{i+2} - v_{i+1}) \cdot n_{i+1}'}{|v_{i+2} - v_{i+1}||n_{i+1}'|} \Big) = \arccos\Big( \frac{(v_{i+2} - v_{i+1}) \cdot n_{i+1}}{|v_{i+2} - v_{i+1}|(|n_{i+1}| + {\rm O}(d_{i+1}^2))} \Big) \\& = \psi_{i+1}^1 + {\rm O}(d_{i+1}^2),
\end{align*}
where ${\rm O}(\cdot)$ is universal. \RRR Likewise, we have $\hat{\psi}_i^2 = \psi_i^2 + {\rm O}(d_{i+1}^2)$. Since $\varphi_{i+1} = \hat{\psi}_{i+1}^1 + \hat{\psi}^2_{i}$,  \EEE to conclude  \eqref{psi-phi-a}, it therefore remains to show 
\begin{align}\label{di+1}
d^2_{i+1} \le c\big(\Delta(z_{\frac{i+1}{2},j,0})  + \Delta(z_{\frac{i+3}{2},j,0}) \big)
\end{align}
for a universal constant $c>0$.  To see this, we first note that we have $d_{i+1} = 0 $ whenever $\Delta(z_{\frac{i+1}{2},j,0})  + \Delta(z_{\frac{i+3}{2},j,0})  = 0$. Indeed, if  $\Delta(z_{\frac{i+1}{2},j,0})  + \Delta(z_{\frac{i+3}{2},j,0})  = 0$, the high symmetry of the atoms in the cells with centers $z_{\frac{i+1}{2},j,0}$ and $z_{\frac{i+3}{2},j,0}$ (cf.  \eqref{delta}) implies that the three normal vectors $n_i$, $n_{i+1}$, and $n_{i+2}$ are coplanar. Thus, $x_0$ is contained in $N_{i+1}$ and therefore $d_{i+1} =0$. 

\BBB Note that $d^2_{i+1}$, $\Delta(z_{\frac{i+1}{2},j,0})$, and $\Delta(z_{\frac{i+3}{2},j,0})$ are functions of the positions of the atoms contained in the adjacent cells with center $z_{\frac{i+1}{2},j,0},  z_{\frac{i+3}{2},j,0}$, denoted by $\tilde{\boldsymbol{y}}=(\tilde{y}_1,\ldots,\tilde{y}_{14}) \in \Rz^{3 \times 14}$.  By \eqref{delta} we find that $\Delta(z_{\frac{i+1}{2},j,0})  + \Delta(z_{\frac{i+3}{2},j,0}) = (\tilde{\boldsymbol{y}} - {\boldsymbol{y}}^0)^T \mathcal{Q} (\tilde{\boldsymbol{y}} - {\boldsymbol{y}}^0)  $ is quadratic with $\mathcal{Q}\in \Rz^{42 \times 42}$, where ${\boldsymbol{y}}^0$ denotes the atomic  positions of \RRR $\mathcal{F}' \in \mathscr{F}(\mu)$. \EEE \BBB Moreover, the fact that  $d^2_{i+1}$ is smooth as a function in $\tilde{\boldsymbol{y}}$, a Taylor expansion, and $d_{i+1} \le \delta $ yield  $d^2_{i+1} \le C |\tilde{\boldsymbol{y}} - {\boldsymbol{y}}^0|^2$ for a universal constant $C>0$. Now    \eqref{di+1} follows from the property that \EEE  $d_{i+1} = 0 $ whenever $\Delta(z_{\frac{i+1}{2},j,0})  + \Delta(z_{\frac{i+3}{2},j,0})  = 0$. 

The second estimate \eqref{psi-phi-b} can be shown along similar lines. This concludes the proof. 
\end{proof}

\section{Properties of the reduced energy: Proof of Lemma \ref{lemma: sym-energy}, Proposition \ref{th: mainenergy}, and Proposition \ref{th: main2}}\label{sec: reduced-energy}

In this section we investigate the properties of the  symmetric energy  and the reduced energy as introduced in \eqref{symmetric-cell} and \eqref{red}, respectively.

\subsection{Proof of Lemma \ref{lemma: sym-energy}} We {\RRR start with} the relation of the cell energy \eqref{eq: cell} and the symmetric energy \eqref{symmetric-cell}. 

\begin{proof}[Proof of Lemma \ref{lemma: sym-energy}]   {\BBB In the proof we let $\lambda=\lambda_2$}.
Given the cell energy, the symmetric energy,  and the constraints \eqref{sym-assumption}-\eqref{eq: constraint2}, we observe that it suffices to show 
\begin{align}\label{v_2--lambda_3}
v_2(\lambda_1) + v_2(\lambda_3) \ge  2{v}_2  \big(\widetilde\mu/2 +2\lambda\cos\alpha_i \big) -c_0  \ell^{-4} (\gamma_1 - \gamma_2)^2 \ \ \ \text{for} \ \ \ i=1,2
\end{align}
for a constant $c_0$ only depending on $v_2$ and $v_3$. First, with the notation of \eqref{sym-assumption}, particularly recalling $\lambda_3 = | x_8 - x_2| = |2(z^{\rm dual}_{i,j,k} - x_2)|$, we see 
$$ \lambda^2_3 = (\widetilde\mu - \lambda_4)^2 + 4|(x_2 - z^{\rm dual}_{i,j,k}) \cdot e_2|^2 + 4|(x_2 - z^{\rm dual}_{i,j,k}) \cdot e_3|^2.$$
As in the special case $\gamma_1 = \gamma_2$ the points $x_1,x_2,z^{\rm dual}_{i,j,k}$ are contained in one line and thus the latter two terms vanish, we obtain by a Taylor expansion $\lambda_3 = \widetilde\mu - \lambda_4+ {\rm O}((\gamma_1 - \gamma_2)^2)$, which together with \eqref{lambda4} gives
$$ \lambda_1 + \lambda_3 = \widetilde\mu + 2\lambda\cos\alpha_1  + {\rm O}((\gamma_1 - \gamma_2)^2).$$
By a similar argument, interchanging the roles of $\lambda_1$ and $\lambda_3$, we also get
$$\lambda_1 + \lambda_3  = \widetilde\mu + 2\lambda\cos\alpha_2  + {\rm O}((\gamma_1 - \gamma_2)^2).$$
\RRR Recall that  $|\lambda_1 - 1|  + |\lambda_3 - 1| \le \ell^{-4}$ and  $|\gamma_1 - \gamma_2| \le \ell^{-2}$ by assumption. \EEE   Then by the convexity of $v_2$ in a neighborhood of $1$  and a Taylor expansion we derive
\begin{align*}
v_2(\lambda_1)  + v_2(\lambda_3) &\ge  2v_2(\widetilde\mu/ 2 + \lambda\cos\alpha_i + {\rm O}((\gamma_1 - \gamma_2)^2)) \\ 
& \ge 2v_2(\widetilde\mu/ 2 + \lambda\cos\alpha_i) - C|v'_2(\widetilde\mu/2   + \lambda\cos\alpha_i)| (\gamma_1 - \gamma_2)^2 - C(\gamma_1 - \gamma_2)^4 
\end{align*}
for $i=1,2$. We recall that  $|v'_2(\widetilde\mu/2 + \lambda\cos\alpha_i)|  = {\rm O}(\ell^{-4})$ since $|\lambda_1 - 1|  + |\lambda_3 - 1| + |\gamma_1 - \gamma_2|^2 \le 2\ell^{-4}$, and $v_2$ is smooth and attains its minimum in $1$. Moreover, observe that by  $|\gamma_1 - \gamma_2|	 \le \ell^{-2}$ we get $|\gamma_1 - \gamma_2|^4 \le \ell^{-4}|\gamma_1 - \gamma_2|^2$. This concludes the proof of \eqref{v_2--lambda_3}. 
\end{proof}

\subsection{Convexity of the reduced energy}
Let us now concentrate on the symmetric energy $E_{\mu,\gamma_1,\gamma_2}^{{\rm sym}}$ introduced in \eqref{symmetric-cell}. We recall the definition of the angle $\beta=\beta(\alpha,\gamma) = 2\arcsin\left(\sin\alpha\sin\frac{\gamma}{2}\right)$ in \eqref{eq: constraint2} and for later use we note that the function $\beta$ is smooth on $[\frac{1}{2}\pi,\frac{3}{4}\pi] \times [\frac{3}{4}\pi,\pi]$ and satisfies
\begin{subequations}\label{zigder2}
\begin{align}
\partial_\alpha\beta(2\pi/3,\pi)&=-2,\quad \partial^2_{\alpha\alpha}\beta(2\pi/3,\pi)=0, \quad \partial_\gamma\beta(2\pi/3,\pi)=0, \label{zigder2-a}\\
 \partial^2_{\gamma\gamma}\beta(2\pi/3,\pi)&=-\sqrt{3}/2,  \quad \partial^2_{\alpha\gamma} \beta(2\pi/3,\pi)=0. \label{zigder2-b}
 \end{align}
\end{subequations}
More precisely, a Taylor expansion also shows
\begin{equation}\label{zigder3}
\lim_{\ell \to \infty} \ell\partial_\gamma\beta(2\pi/3,\gamma_\ell)=  \frac{\sqrt{3}}{2} \pi, \ \ \ \ \ \ \lim_{\ell \to \infty} \ell^2\partial^2_{\alpha\alpha}\beta(2\pi/3,\gamma_\ell)=  -2\sqrt{3} \pi^2,
\end{equation}
where \PPP $\gamma_\ell$ was defined in \eqref{gamma}. \EEE For the exact expressions of the derivatives of the function $\beta$ we refer the \PPP 
reader \EEE to \cite[Section 4]{MMPS-new}. Recall the definition of $\alpha_\ell^{\rm us}$ in Proposition \ref{eq: old main result}.

\begin{lemma}[Angles of unstretched nanotubes]\label{aus}
There \PPP exist \EEE $0 < c_1 < c_2$ and $\ell_0 \in \Nz$ only depending on $v_3$ such that for all $\ell \ge \ell_0$
$$
\alpha^{\rm us}_\ell, \beta(\alpha_\ell^{\rm us}, { \RRR \gamma_\ell \EEE }) \in (2\pi/3 - c_2\ell^{-2} , 2\pi/3 - c_1\ell^{-2}).
$$
\end{lemma}

\begin{proof}
By Proposition \ref{eq: old main result} and the fact that $\alpha \mapsto \beta(\alpha,\gamma_\ell)$ is decreasing,  we obtain \RRR $\alpha_\ell^{\rm us} \ge \alpha_\ell^{\rm ch}$ and \EEE  $\beta(\alpha^{\rm us}_\ell,\gamma_\ell) \le \alpha^{\rm us}_\ell \le 2\pi/3$. \RRR By \cite[(11)]{MMPS} we have $2\pi/3 - \alpha_\ell^{\rm ch} = {\rm O}(\ell^{-2})$. Moreover, in view of \eqref{gamma}, \eqref{betaz} and a Taylor expansion, we find  $ \alpha^{\rm us}_\ell -  \beta(\alpha^{\rm us}_\ell,\gamma_\ell)  \ge C\ell^{-2}$.  \RRR Summarizing, we get   
\begin{align}\label{sumatjunction}
\RRR 2\pi/3 - \alpha^{\rm us}_\ell \le C \ell^{-2},  \ \ \ \ \ \  2\pi - 2\alpha^{\rm us}_\ell - \beta(\alpha^{\rm us}_\ell,\gamma_\ell)\ge C\ell^{-2} 
\end{align}
\EEE for some universal $C>0$. As $2v_3(\alpha) + v_3(\beta(\alpha,\gamma_\ell))$ is minimized at $\alpha = \alpha^{\rm us}_\ell$ (see Proposition \ref{eq: old main result}), we get $2v'_3( \alpha^{\rm us}_\ell) + v'_3(\beta( \alpha^{\rm us}_\ell,\gamma_\ell))\partial_\alpha \beta( \alpha^{\rm us}_\ell,\gamma_\ell) = 0$. Using \eqref{zigder2-a} and a Taylor expansion of \RRR $v'_3$ \EEE around $2\pi/3$, we deduce that for $\ell_0$ large enough and all $\ell \ge \ell_0$
$$ \frac{2\pi/3 - \alpha^{\rm us}_\ell}{ 2\pi/3 -  \beta( \alpha^{\rm us}_\ell,\gamma_\ell)} \in [C',1]$$
for a constant $0 < C'<1$ only depending on $v_3$. This together with \eqref{sumatjunction} concludes the proof. 
\end{proof}

 Recall the minimization problem \eqref{red} for the symmetric energy introduced in \eqref{symmetric-cell}.   We proceed with the identification of the minimizers of \eqref{red}.

\begin{proposition}[Existence and uniqueness of minimizers]\label{prop1}
There exists $\delta>0$ depending only on $v_2$, $v_3$  such that, for any fixed $\mu\in[3-\delta, 3+\delta]$ and $\gamma = (\gamma_1,\gamma_2)\in [\pi-\delta,\pi]^2$,  the minimization problem  \eqref{red} has a unique solution $(\lambda^*(\mu,\gamma),\alpha_1^*(\mu,\gamma), \alpha_2^*(\mu,\gamma))$, which satisfies 
\begin{align}\label{eq: firstorder-opt}
\nabla E_{\mu,\gamma_1,\gamma_2}^{{\rm sym}}(\lambda^*(\mu,\gamma),\alpha_1^*(\mu,\gamma), \alpha_2^*(\mu,\gamma)) = 0,
\end{align}
where $\nabla$ denotes the derivative with respect to $(\lambda, \alpha_1,\alpha_2)$.
\end{proposition}

\begin{proof}
 We start the proof with a direct computation of the derivatives. Replace  $E_{\mu,\gamma_1,\gamma_2}^{{\rm sym}}$ by $\tilde{E}$ for notational convenience. We obtain 
\begin{subequations}\label{first order}
\begin{align}\displaystyle 
{\partial_{\lambda}}\tilde{E}(\lambda, \alpha_1, \alpha_2)&=
2v_2'(\lambda) +  \sum_{i=1,2} \Big( \frac{1}{2} \cos\alpha_i \, v_2'(\mu/2  + \lambda\cos\alpha_i) \Big)\PPP,\EEE \label{first order-a} \\\displaystyle
{\partial_{\alpha_i}}\tilde{E}(\lambda,\alpha_1, \alpha_2)&=- \frac{1}{2}\lambda \sin\alpha_i \,  v_2'(\mu/2  + \lambda\cos\alpha_i) \notag \\ &  \ \ \  + v'_3(\beta(\alpha_i,\gamma_i))\partial_\alpha\beta(\alpha_i,\gamma_i) + 2 v'_3(\alpha_i), \ \  \ \ \ \  i =1,2.\label{first order-b}
\end{align}
\end{subequations}
Moreover, for $i=1,2$
\[\begin{aligned}
{\partial^2_{\lambda\lambda}}\tilde{E}(\lambda,\alpha_1,\alpha_2)&=
2v_2''(\lambda) +  \sum_{j=1,2} \Big( \frac{1}{2}\cos^2\alpha_j \ v_2''(\mu/2 + \lambda\cos\alpha_j) \Big),\\
{\partial^2_{\alpha_i\alpha_i}}\tilde{E}(\lambda,\alpha_1,\alpha_2)&= 
\frac{1}{2}\lambda^2 \sin^2\alpha_i  \, v_2''(\mu/2  + \lambda\cos\alpha_i)   - \frac{1}{2}\lambda \cos\alpha_i \, v_2'(\mu/2  + \lambda\cos\alpha_i)  + 2 v_3''(\alpha_i)\\
& \ \ \ +  v''_3(\beta(\alpha_i,\gamma_i))\,(\partial_\alpha\beta(\alpha_i,\gamma_i))^2 + v'_3(\beta(\alpha_i,\gamma_i))\partial^2_{\alpha\alpha}\beta(\alpha_i,\gamma_i),
\\
 {\partial^2_{\lambda\alpha_i}}\tilde{E}(\lambda,\alpha_1,\alpha_2)&=  -\frac{1}{2}\sin\alpha_i \,  v_2'(\mu/2 + \lambda\cos\alpha_i)   - \frac{1}{2}\lambda \sin\alpha_i \cos\alpha_i  \, v_2''(\mu/2  + \lambda\cos\alpha_i ), 
 \\ 
  {\partial^2_{\alpha_1\alpha_2}}\tilde{E}(\lambda,\alpha_1,\alpha_2)&= 0. 
\end{aligned}
\]
For notational convenience we define $s_{\rm ref} :=(1,2\pi/3,2\pi/3)$. Recall that $\partial_\alpha \beta(2\pi/3,\pi)= -2  $ by \eqref{zigder2-a}, $\beta(2\pi/3,\pi)=2\pi/3$ by  \eqref{eq: constraint2},  $v_3'(2\pi/3) =0$,  $\cos(2\pi/3) = - 1/2$, $\sin(2\pi/3) =\sqrt{3}/2$. At the planar reference configuration $\mu = 3$, $\gamma_1=\gamma_2 = \pi$, $\alpha_1= \alpha_2=2\pi/3$, $\lambda  = 1$ the derivative then reads after some computation  
\begin{align*}
{\partial^2_{\lambda\lambda}}E_{3,\pi,\pi}^{{\rm sym}}(s_{\rm ref}) &= \frac{9}{4} v_2''(1), \ \ \ \ \ \
{\partial^2_{\alpha_i\alpha_i}}E_{3,\pi,\pi}^{{\rm sym}}(s_{\rm ref})=\frac{3}{8} v_2''(1) + 6 v''_3(2\pi/3), \ \ i=1,2,\\
{\partial^2_{\lambda\alpha_i}}E_{3,\pi,\pi}^{{\rm sym}}(s_{\rm ref}) &=  \frac{\sqrt{3}}{8}v_2''(1), \ \ i=1,2,\ \ \ \ \ 
{\partial^2_{\alpha_1\alpha_2}}E_{3,\pi,\pi}^{{\rm sym}}(s_{\rm ref})=  0.
\end{align*}
We shall check the positivity of the Hessian matrix in a neighborhood of the reference configuration. Since
\begin{align*}
{\rm det} \Big( D^2_{\alpha_1\alpha_2}E_{3,\pi,\pi}^{{\rm sym}}(s_{\rm ref}) \Big) &= \big({\partial^2_{\alpha_1\alpha_1}}E_{3,\pi,\pi}^{{\rm sym}}(s_{\rm ref})\big)^2, \\ 
{\rm det} \big(D^2E_{3,\pi,\pi}^{{\rm sym}}(s_{\rm ref})\big) &=
  \big({\partial^2_{\alpha_1\alpha_1}}E_{3,\pi,\pi}^{{\rm sym}}(s_{\rm
  ref})\big)^2  {\partial^2_{\lambda\lambda}}E_{3,\pi,\pi}^{{\rm
  sym}}(s_{\rm ref}) \\
&- 2\big({\partial^2_{\lambda\alpha_1}}E_{3,\pi,\pi}^{{\rm sym}}(s_{\rm ref})\big)^2 {\partial^2_{\alpha_1\alpha_1}}E_{3,\pi,\pi}^{{\rm sym}}(s_{\rm ref})
\end{align*}
are positive,  the principal minors of the Hessian matrix $D^2E_{3,\pi,\pi}^{{\rm sym}}(1,2\pi/3,2\pi/3)$ are positive. Due to the smoothness of the potentials $v_2$, $v_3$ and the mapping $(\alpha,\gamma) \mapsto \beta(\alpha,\gamma)$, we get that for $\delta'>0$ sufficiently small the principal  minors of the Hessian matrix $D^2E_{\mu,\gamma_1,\gamma_2}^{{\rm sym}}(\lambda,\alpha_1,\alpha_2)$ are positive for all $(\lambda,\alpha_1,\alpha_2) \in D_{\delta'}$ and
 for all  $\mu\in[3-\delta',3+\delta']$, $(\gamma_1,\gamma_2) \in [\pi-\delta',\pi]^2$, where 
 $$D_{\delta'} := [1-\delta',1+\delta'] \times  [2\pi/3-\delta',2\pi/3 + \delta']^2.$$ 
Since we have shown that $E_{\mu,\gamma_1,\gamma_2}^{{\rm sym}}$ is strictly convex on $D_{\delta'}$, it follows that it has a unique minimizer $(\lambda^*(\mu,\gamma),\alpha_1^*(\mu,\gamma), \alpha_2^*(\mu,\gamma) )$ for all  $\mu\in[3-\delta',3+\delta']$ and $\gamma = (\gamma_1,\gamma_2) \in [\pi-\delta',\pi]^2$. Moreover, a continuity argument shows that
\begin{align}\label{eq:continuity}
(\lambda^*(\mu,\gamma),\alpha_1^*(\mu,\gamma), \alpha_2^*(\mu,\gamma)) & \to (\lambda^*(3,\pi,\pi), \alpha_1^*(3,\pi,\pi), \alpha_2^*(3,\pi,\pi))  = (1,2\pi/3, 2\pi/3) 
\end{align}
 as $\gamma \to  (\pi,\pi)$ and $\mu \to 3$.  \BBB Recalling \eqref{symmetric-cell} and the fact that $v_2$ and $v_3$ attain their minimum exactly at $1$ and $2\pi/3$, respectively, we find $\inf_{(\lambda,\alpha_1,\alpha_2)\notin D_{\delta'}}E_{\mu,\gamma_1,\gamma_2}^{{\rm sym}}(\lambda,\alpha_1,\alpha_2) > - 3$. On the other hand, by \eqref{eq: constraint2}, \eqref{symmetric-cell}, and \eqref{eq:continuity} we get $E_{\mu,\gamma_1,\gamma_2}^{{\rm sym}}(\lambda^*(\mu,\gamma),\alpha_1^*(\mu,\gamma), \alpha_2^*(\mu,\gamma)) \to -3$  as $\gamma \to  (\pi,\pi)$ and $\mu \to 3$. This shows that  for all $\mu\in[3-\delta'',3+\delta'']$ and $\gamma \in [\pi-\delta'',\pi]^2$, for some small $\delta''>0$, the triple $(\lambda^*(\mu,\gamma),\alpha_1^*(\mu,\gamma), \alpha_2^*(\mu,\gamma))$ is the unique solution of the minimization problem  \eqref{red}. Moreover, if $\delta''>0$ is chosen small enough, the triple \EEE lies in the interior of $D_{\delta'}$ and the first order optimality conditions \eqref{eq: firstorder-opt} follow.  We conclude the proof by setting $\delta = \min\lbrace \delta',\delta''\rbrace$. 
 \end{proof}

 We now study convexity properties of the reduced energy $E_{\rm red}$ defined in \eqref{red}.  Recall the definition of $\gamma_\ell$ in \eqref{gamma} and the \RRR definition of  $\mu^{\rm us}_\ell$ in \eqref{eq:muellneu}. \EEE

\begin{proposition}[Convexity of reduced energy]\label{convexenergy}
There exists $\ell_0 \in \Nz$ and  for each $\ell \ge \ell_0$ there exits $\eps=\eps(\ell)>0$ such that  $E_{\rm red}$ is strictly convex on $D^\ell_\eps :=[\mu_\ell^{\rm us} - \eps, \mu_\ell^{\rm us}+ \eps] \times [\gamma_\ell - \eps,\gamma_\ell+\eps]^2$. Moreover, there \PPP exists \EEE  $c_0'>0$ depending only on $v_2$ and $v_3$ such that for all $\ell \ge \ell_0$ and $(\mu,\gamma_1,\gamma_2) \in D_\eps^\ell$  
\begin{align}\label{eq:convexenergy}
E_{\rm red}(\mu,\gamma_1,\gamma_2) = E_{\rm red}(\mu,\gamma_2,\gamma_1) \ge E_{\rm red} \Big( \mu, \frac{\gamma_1 + \gamma_2}{2},\frac{\gamma_1 + \gamma_2}{2} \Big) +  c_0' \ell^{-2} (\gamma_1-\gamma_2)^2.
\end{align}
\end{proposition}
 
\begin{proof}
 Choosing $\ell$ sufficiently large and $\eps>0$ small we can suppose that $D^\ell_\eps \subset [3- \delta, 3+ \delta]  \times [\pi- \delta,\pi]^2$ with  $\delta$ from Proposition \ref{prop1} since $\mu^{\rm us}_\ell = 2-2\cos\alpha^{\rm us}_\ell \to 3$ as $\ell \to\infty$.  \RRR Then \eqref{eq: firstorder-opt} holds for $(\mu,\gamma_1,\gamma_2) \in  D^\ell_\eps$. \EEE

We drop the brackets $(\mu,\gamma_1,\gamma_2)$ and indicate the unique solution at $(\mu,\gamma_1,\gamma_2)$ by $(\lambda^*,\alpha_1^*, \alpha_2^*)$ for notational convenience. Taking the partial derivatives and  making use of the first order optimality conditions \eqref{eq: firstorder-opt}, we get
\begin{align}\label{derivative1}
\partial_\mu E_{\rm red}(\mu,\gamma_1,\gamma_2)&=\frac d{d\mu}E_{\mu,\gamma_1,\gamma_2}^{{\rm sym}}(\lambda^*,\alpha_1^*, \alpha_2^*)\notag\\
&=\frac{\partial E_{\mu,\gamma_1,\gamma_2}^{{\rm sym}}}{\partial \mu}(\lambda^*,\alpha_1^*, \alpha_2^*)  + \nabla E_{\mu,\gamma_1,\gamma_2}^{{\rm sym}}(\lambda^*,\alpha_1^*, \alpha_2^*)\cdot (\partial_\mu\lambda^*,\partial_\mu\alpha_1^*, \partial_\mu\alpha_2^*)\notag\\
&=\frac{\partial E_{\mu,\gamma_1,\gamma_2}^{{\rm sym}}}{\partial \mu}(\lambda^*,\alpha_1^*, \alpha_2^*) 
=\sum_{j=1,2}\frac{1}{4}{v}'_2  \big(\mu/2+\lambda^*\cos\alpha^*_j \big),
\end{align}
where $\nabla$ denotes the derivative with respect to $(\lambda,\alpha_1,\alpha_2)$. Likewise, we get for $i=1,2$
\begin{equation}\label{derivative1.2}
\begin{aligned}
\partial_{\gamma_i} E_{\rm red}(\mu,\gamma_1,\gamma_2) &=\frac{\partial E_{\mu,\gamma_1,\gamma_2}^{{\rm sym}}}{\partial \gamma_i}(\lambda^*,\alpha_1^*, \alpha_2^*)  =  v'_3(\beta(\alpha_i^*,\gamma_i))\, \partial_\gamma \beta(\alpha_i^*,\gamma_i).
\end{aligned}
\end{equation}
Next we compute the second derivatives and obtain
  \begin{align}
    \partial^2_{\mu\mu} E_{\rm red}(\mu,\gamma_1,\gamma_2)&=\sum_{j=1,2}\frac{1}{4} {v}''_2  \big(\mu/2 +\lambda^*\cos\alpha^*_j \big) \, w_{j,\mu}(\mu,\gamma_1,\gamma_2), \label{convexity7-a}\\
    \partial^2_{\gamma_i\gamma_i} E_{\rm red}(\mu,\gamma_1,\gamma_2)&= v'_3(\beta(\alpha_i^*,\gamma_i))\, \big( \partial^2_{\gamma\gamma} \beta(\alpha_i^*,\gamma_i) + \partial^2_{\gamma\alpha} \beta(\alpha_i^*,\gamma_i) \,\partial_{\gamma_i} \alpha^*_i \big)\notag\\
    & \hspace{-10mm}+  v''_3(\beta(\alpha_i^*,\gamma_i))\, \partial_\gamma \beta(\alpha_i^*,\gamma_i) \cdot \big(\partial_\gamma\beta(\alpha_i^*,\gamma_i) +  \partial_\alpha\beta(\alpha_i^*,\gamma_i)\, \partial_{\gamma_i} \alpha_i^* \big), \  i=1,2,\label{convexity7-b}\\
    \partial^2_{\mu\gamma_i} E_{\rm red}(\mu,\gamma_1,\gamma_2)&= \sum_{j=1,2}\frac{1}{4} {v}''_2  \big(\mu/2 + \lambda^*\cos\alpha^*_j \big) \, w_{j,\gamma_i}(\mu,\gamma_1,\gamma_2), \ \ i=1,2,\label{convexity7-c}\\
    \partial^2_{\gamma_1\gamma_2} E_{\rm red}(\mu,\gamma_1,\gamma_2) &  =  v'_3(\beta(\alpha_1^*,\gamma_1))\,  \partial^2_{\gamma\alpha} \beta(\alpha_1^*,\gamma_1) \,\partial_{\gamma_2} \alpha^*_1 \notag \\
    &  \ \ \ +
    v''_3(\beta(\alpha_1^*,\gamma_1))\, \partial_\gamma
    \beta(\alpha_1^*,\gamma_1)
    \, \partial_\alpha\beta(\alpha_1^*,\gamma_1)\, \partial_{\gamma_2}
    \alpha_1^*,\label{convexity7-d}
  \end{align} 
where for brevity we have introduced
\begin{subequations}\label{convexity2}
\begin{align}
w_{j,\mu}(\mu,\gamma_1,\gamma_2) & = 1/2  + \partial_\mu\lambda^*\cos\alpha^*_j - \lambda^*\sin\alpha^*_j\,\partial_\mu \alpha^*_j,  \ \ \ \ j =1,2, \label{convexity2-a}\\
w_{j, \gamma_i}(\mu,\gamma_1,\gamma_2) & =  \partial_{\gamma_i}\lambda^*\cos\alpha^*_j - \lambda^*\sin\alpha^*_j\, \partial_{\gamma_i} \alpha^*_j, \ \ \  \ i,j=1,2. \label{convexity2-b}
\end{align}
\end{subequations}
We now exploit the identity  $\nabla E^{\rm sym}_{\mu,\gamma_1,\gamma_2}(\lambda^*,\alpha_1^*, \alpha_2^*) =0$: differentiating \eqref{first order} with respect to $\mu$, $\gamma_1$ or $\gamma_2$, respectively, we obtain
\begin{align}
0&=2v_2''(\lambda^*) \, \partial_X \lambda^* +  \sum_{j=1,2} \Big( -\frac{1}{2}\sin\alpha^*_j \, \partial_X \alpha^*_j \, v_2'(\mu/2 + \lambda^*\cos\alpha^*_j) \Big) \notag \\
&  \ \  \ + \sum_{j=1,2} \Big( \frac{1}{2}\cos\alpha^*_j \, v_2''(\mu/2  + \lambda^*\cos\alpha^*_j) \, w_{j,X}(\mu,\gamma_1,\gamma_2)  \Big), \label{eq: long-est-a} \\
0 & = -\frac{1}{2}v_2'(\mu/2 + \lambda^*\cos\alpha^*_j)  \Big( \partial_X\lambda^* \sin\alpha^*_j   + \lambda^* \cos\alpha^*_j \partial_X \alpha_j^* \Big)\notag \\
& \ \ \ -\frac{1}{2}\lambda^* \,  \sin\alpha^*_j \,  v_2''(\mu/2 + \lambda^*\cos\alpha^*_j)\, w_{j,X}(\mu,\gamma_1,\gamma_2) + v'_3(\beta(\alpha^*_j,\gamma_j))\partial^2_{\alpha\alpha}\beta(\alpha^*_j,\gamma_j)\partial_X\, \alpha^*_j \notag\\
& \ \ \ + v''_3(\beta(\alpha^*_j,\gamma_j))\big(\partial_\alpha \beta(\alpha^*_j,\gamma_j)\big)^2\, \partial_X \alpha^*_j + 2v''_3(\alpha^*_j) \,  \partial_X \alpha^*_j + z_{j,X}(\mu,\gamma_1,\gamma_2), \ \ \ j=1,2,   \label{eq: long-est-b}
\end{align}
where $X \in \lbrace \mu, \gamma_1, \gamma_2 \rbrace$ and where we have defined for brevity
\begin{align*}
z_{j,\gamma_j}(\mu,\gamma_1,\gamma_2) &= v'_3(\beta(\alpha^*_j,\gamma_j))\partial_{\alpha\gamma}\beta(\alpha^*_j,\gamma_j) +  v''_3(\beta(\alpha^*_j,\gamma_j))\partial_\alpha\beta(\alpha^*_j,\gamma_j)  \partial_\gamma \beta(\alpha^*_j,\gamma_j),\\
z_{j,\gamma_i}(\mu,\gamma_1,\gamma_2) &= z_{j,\mu}(\mu,\gamma_1,\gamma_2) = 0, \ \ \ i \neq j.
\end{align*}
\RRR For brevity let $t_{\rm ref}^\ell := (\mu^{\rm us}_\ell,\gamma_\ell,\gamma_\ell)$ and $t_{\rm ref} := (3,\pi,\pi)$. Observe that $t_{\rm ref}^\ell \to t_{\rm ref}$ as $\ell \to \infty$ by  \eqref{gamma}, \eqref{eq:muellneu}, and Lemma \ref{aus}. Moreover, by \eqref{eq:continuity} we get that the unique solution of the problem \eqref{red} corresponding to $t_{\rm ref}^\ell$ converges to $(1,2\pi/3,2\pi/3)$, in particular $\alpha^*_j(t_{\rm ref}^\ell)   \to 2\pi/3$ for $j=1,2$.  We also recall $\beta(\alpha^*_j(t_{\rm ref}^\ell),\gamma_\ell) \to 2\pi/3$ for $j=1,2$  (see \eqref{eq: constraint2}). \EEE  Using $v_2'(1) = v_3'(2\pi/3) = 0$, $\cos(2\pi/3) = -1/2$, $\sin(2\pi/3)=  \sqrt{3}/2$ and \eqref{zigder2}  we then deduce from \eqref{eq: long-est-a}-\eqref{eq: long-est-b} 
\begin{subequations}\label{convexity1}
\begin{align}
0 &= 2 v_2''(1)\, \partial_X \lambda^*({\RRR t_{\rm ref} \EEE})-  \frac{1}{4} v_2''(1)\sum_{j=1,2}w_{j,X}({\RRR t_{\rm ref} \EEE}), \label{convexity1-a}\\
0 & = -v_2''(1) \, w_{j,X}({\RRR t_{\rm ref} \EEE})
+ 8\sqrt{3} v_3''(2\pi/3) \, \partial_X\alpha^*_j({\RRR t_{\rm ref} \EEE}), \ \ \ j=1,2, \label{convexity1-b}
 \end{align} 
\end{subequations}
as $\ell \to \infty$ , where $X \in \lbrace \mu, \gamma_1, \gamma_2 \rbrace$. Inserting the identities into \eqref{convexity2}, we obtain, after some elementary but tedious calculations, 
\begin{subequations}\label{convexity3}
\begin{align}
&w_{1,\mu}({\RRR t_{\rm ref} \EEE }) = w_{2,\mu}({\RRR t_{\rm ref} \EEE }) = 4/K, \ \ \  w_{1,\gamma_i}({\RRR t_{\rm ref} \EEE }) = w_{2,\gamma_i}({\RRR t_{\rm ref} \EEE })=0, \ \ i=1,2,\\
& \partial_\mu \lambda^*({\RRR t_{\rm ref} \EEE }) = 1/K, \ \ \ \partial_\mu \alpha_1^*({\RRR t_{\rm ref} \EEE }) =\partial_\mu \alpha_2^*({\RRR t_{\rm ref} \EEE })= v_2''(1) / (2\sqrt{3}K v_3''(2\pi/3)), \label{convexity3-b}
\end{align}
\end{subequations} 
where $K:= 9 + v_2''(1)/(2 v_3''(2\pi/3))$. In particular, the last two equalities of the first line together with \eqref{convexity1} yield that
$\partial_{\gamma_i}\lambda^*$,   $\partial_{\gamma_i} \alpha^*_1$, and $\partial_{\gamma_i} \alpha^*_2$  vanish at ${\RRR t_{\rm ref} \EEE }$. Thus, by a Taylor expansion in terms of $1/\ell$ the limits $w_{j,\gamma_i}^\infty := \lim_{\ell \to \infty} \ell w_{j, \gamma_i}({\RRR t^\ell_{\rm ref} \EEE })$,  $\partial_{\gamma_i}\lambda^\infty := \lim_{\ell \to \infty} \ell \partial_{\gamma_i} \lambda^*({\RRR t^\ell_{\rm ref} \EEE })$, and $\partial_{\gamma_i} \alpha^\infty_j := \lim_{\ell \to \infty} \ell \partial_{\gamma_i} \alpha^*_j({\RRR t^\ell_{\rm ref} \EEE })$ for $i,j=1,2$  exist {\BBB and are finite}.  

By Lemma \ref{aus} and the fact that $v_3$ is smooth with minimum at
$2\pi/3$  we note that one has $|v'_3(\beta(\alpha^{\rm us}_\ell,\gamma_\ell))| \le C\ell^{-2}$ for a constant only depending on $v_3$.   Consequently, multiplying the  estimates in \eqref{eq: long-est-a}-\eqref{eq: long-est-b} by $\ell$ and letting $\ell \to \infty$ we get using \eqref{zigder2} and \eqref{zigder3}
\begin{align*}
\begin{split}
0 & = 2 v_2''(1) \partial_{\gamma_i} \, \lambda^\infty - \frac{1}{4} v_2''(1)\sum_{j=1,2}w_{j,{\gamma_i} }^\infty,  \ \ \ i=1,2,\\
0&= - \frac{1}{4} v_2''(1)w_{j,\gamma_i}^\infty  + 2\sqrt{3} v_3''(2\pi/3) \,  \partial_{\gamma_i} \alpha^\infty_j - v_3''(2\pi/3)\pi \, \delta_{ij}, \ \ i,j=1,2,
\end{split}
\end{align*}
where $\delta_{ij}$ denotes the Kronecker delta.  As before,  inserting the identities into \eqref{convexity2-b}, we obtain after some tedious calculations 
\begin{subequations}\label{convexity6}
\begin{align}
\sum_{j=1,2}w_{j,\gamma_i}^\infty &  = - \frac{2\pi}{K}, \ \ \ \ \ \ \ \  \ \ \   \ \ \ \ \ \ \ \  \ \ \ \ \ \ \ \ \  \sum_{j=1,2}  \partial_{\gamma_i} \alpha_j^\infty = \frac{\pi}{2\sqrt{3}} - \frac{\pi v_2''(1)}{4\sqrt{3}K v_3''(2\pi/3)},\\
 \partial_{\gamma_i} \alpha_i^\infty & = \frac{\pi}{2\sqrt{3}} - \frac{\pi v_2''(1)}{4\sqrt{3}K v_3''(2\pi/3)} - \frac{\pi}{KK^\infty}, \ \ \ \ \ \ \ \ \ \partial_{\gamma_i} \alpha_j^\infty =\frac{\pi}{KK^\infty}, \ \ i \neq j,
 \end{align}
\end{subequations} 
for $i=1,2$ with $K$ as defined after \eqref{convexity3} and $K^\infty := 64\sqrt{3} v_3''(2\pi/3)/v_2''(1) + 4\sqrt{3}$. Moreover, \PPP we notice by \eqref{zigder2-b}  and Lemma \ref{aus} that \EEE
$$v'_3(\beta(\alpha_\ell^{\rm us}, \gamma_\ell)) \partial^2_{\gamma\gamma} \beta(\alpha_\ell^{\rm us},\gamma_\ell) \ge 0$$ 
for $\ell$ sufficiently large. With this at hand, we go back to
\eqref{convexity7-a}-\eqref{convexity7-d} and derive as $\ell \to \infty$ by \eqref{zigder2}, \eqref{zigder3},
\eqref{convexity3}, and \eqref{convexity6}
 
  \begin{align}
    \partial^2_{\mu\mu} E_{\rm red}({\RRR t^\ell_{\rm ref} \EEE })&=\frac{2v''_2(1)}{K}+ {\rm O}(\ell^{-1}),\label{convexity-last-a}\\
    \partial^2_{\gamma_i\gamma_i} E_{\rm red}({\RRR t^\ell_{\rm ref} \EEE })& \ge  \ell^{-2} \Big(
    v''_3(2\pi/3) \frac{3}{4}\pi^2 -  v''_3(2\pi/3)
    \,\sqrt{3}\pi \partial_{\gamma_i}\alpha^\infty_i  \Big)+ {\rm
      O}(\ell^{-3}) \notag \\
    & \RRR = \EEE \ell^{-2}  v''_3(2\pi/3)\pi^2 \Big( \frac{1}{4} + \frac{v_2''(1)}{4Kv_3''(2\pi/3)} + \frac{\sqrt{3}}{KK^\infty}\Big) + {\rm O}(\ell^{-3}), \  i=1,2,  \notag\\
    \partial^2_{\mu\gamma_i} E_{\rm red}({\RRR t^\ell_{\rm ref} \EEE })&= - \ell^{-1}\frac{\pi v_2''(1)}{2K}  +  {\rm O}(\ell^{-2}), \ \ i=1,2, \notag  \\
    \partial^2_{\gamma_1\gamma_2} E_{\rm red}({\RRR t^\ell_{\rm ref} \EEE })& = -\ell^{-2} v''_3(2\pi/3)\,
    \sqrt{3}\pi \partial_{\gamma_1} \alpha_2^\infty + {\rm
      O}(\ell^{-3})  = -\ell^{-2} v''_3(2\pi/3)\,
    \frac{\sqrt{3}\pi^2}{KK^\infty} + {\rm O}(\ell^{-3}). \notag 
  \end{align} 
We now check the positivity of the Hessian $D^2 E_{\rm red}$ by considering the  minors $H_1 = \partial^2_{\gamma_2\gamma_2} E_{\rm red}$, $H_2 = \det  (D^2_{\gamma_1\gamma_2} E_{\rm red}  )$ and $H_3 = \det (D^2  E_{\rm red})$. First, we get for $\ell \in  \Nz$ sufficiently large  
\begin{align*}
&H_1({\RRR t^\ell_{\rm ref} \EEE })  \ge  \ell^{-2}  v''_3(2\pi/3) \frac{\pi^2}{4}>0, \ \ \ H_2({\RRR t^\ell_{\rm ref} \EEE }) \ge \ell^{-4} (v_3(2\pi/3)'')^2 \pi^4  (1/4)^2  >0 
  \end{align*}
and finally  for $\ell$ large enough 
\begin{align*}
H_3({\RRR t^\ell_{\rm ref} \EEE }) &=  \Big(\partial^2_{\gamma_2\gamma_2} E_{\rm red} -  \partial^2_{\gamma_1\gamma_2} E_{\rm red}\Big)   \cdot \Big(  \partial^2_{\mu\mu} E_{\rm red}\big(\partial^2_{\gamma_2\gamma_2} E_{\rm red} + \partial^2_{\gamma_1\gamma_2} E_{\rm red}\big)  - 2\big( \partial^2_{\mu\gamma_1} E_{\rm red}\big)^2 \Big)\\
& \ge \ell^{-4}  v''_3(2\pi/3) \frac{\pi^2}{4} \Big( \frac{\pi^2 v_2''(1) v_3''(2\pi/3)}{2K} + \frac{\pi^2(v_2''(1))^2}{2K^2} - 2 \frac{\pi^2(v_2''(1))^2}{4K^2}\Big)>0.
\end{align*}
 Due to the smoothness of the potentials $v_2$, $v_3$, the mapping $(\alpha,\gamma) \mapsto \beta(\alpha,\gamma)$, and the solutions $(\lambda^*,\alpha^*_1,\alpha_2^*)$ as functions of $(\mu,\gamma_1,\gamma_2)$, we get that for $\ell_0 \in \Nz$ sufficiently large and  $\eps>0$ small (depending on $\ell$)    
$H_i(\mu,\gamma_1,\gamma_2) >0$ for $i=1,2,3$ for all $(\mu,\gamma_1,\gamma_2) \in [\mu^{\rm us}_\ell - \eps, \mu^{\rm us}_\ell + \eps] \times [\gamma_\ell - \eps, \gamma_\ell + \eps]^2 $.  

It remains to confirm \eqref{eq:convexenergy}. The first identity is a consequence of the fact that  $E^{\rm sym}_{\mu,\gamma_1,\gamma_2}$ is  symmetric in $(\alpha_1,\gamma_1)$ and $(\alpha_2,\gamma_2)$. \BBB Recalling \eqref{convexity-last-a} and the fact that  $D^2E_{\rm red}$ is positive definite, we can control the eigenvalues of $\ell^2 D^2E_{\rm red}$ from below and find $\ell^2 D^2E_{\rm red} \ge  8c_0' \mathbf{I} + {\rm O}(\ell^{-1})$ for some constant $c_0'$ depending only on $v_2''(1)$ and $v_3''(2\pi/3)$, where $\mathbf{I}$ denotes the identity matrix. This implies the  second estimate of \eqref{eq:convexenergy}. \EEE
\end{proof}

\subsection{Proof of  Proposition \ref{th: mainenergy} and Proposition \ref{th: main2}}
We are now in the position to show the main properties of $E_{\rm red}$.

\begin{proof}[Proof of Proposition \ref{th: mainenergy}]
Property 2 follows directly from Proposition \ref{convexenergy} if the intervals $M^\ell, G^\ell$ are chosen appropriately  depending on $\eps$, with $\eps$ from Proposition \ref{convexenergy}. 

In Proposition \ref{prop1} we have seen that for given $(\mu,\gamma_1,\gamma_2) \in M^\ell \times G^\ell \times G^\ell$ there is a unique solution $(\lambda^*,\alpha_1^*, \alpha_2^*)$ of the minimization problem  \eqref{red}. In particular, if $\gamma_1 = \gamma_2$ we obtain $\alpha^* := \alpha_1^* = \alpha_2^*$ as then \eqref{red} is completely symmetric in $\alpha_1$ and $\alpha_2$.  \RRR This  \PPP proves \EEE Property 1. \EEE

We now specifically consider the case $\gamma_1 = \gamma_2 = \gamma_\ell$ and denote the minimizer in \eqref{red} by  $(\lambda^\mu,\alpha^\mu, \alpha^\mu)$. We observe that  $\lambda_1^\mu := \mu/2 +\lambda^\mu \cos\alpha^\mu$, $\lambda_2^\mu := \lambda^\mu$, and $\sigma^\mu : = -\lambda^\mu\cos\alpha^\mu$ satisfy the relations \eqref{eq: basic constraints} and \eqref{alphars}. Then   by  \eqref{basicenergy}, \eqref{symmetric-cell}, and the fact that $n= 4m\ell$ we derive
\begin{align*}
E_{\rm red}(\mu,\gamma_\ell,\gamma_\ell) &=  2v_2(\lambda^\mu)    +   {v}_2  \big(\mu/2 +\lambda^\mu\cos\alpha^\mu \big) +  4   v_3(\alpha^\mu) +  2v_3(\beta(\alpha^\mu,\gamma_\ell))\\
& =  2v_2(\lambda^\mu_2) +  v_2(\lambda^\mu_1)     +  4   v_3(\alpha^\mu) +  2v_3(\beta(\alpha^\mu,\gamma_\ell))  = \frac{1}{2m\ell} E(\mathcal{F}_{\lambda_1^\mu,\lambda_2^\mu,\mu}),
\end{align*}
which concludes the proof of Property 5.

To see Property 3, we introduce $g(\gamma) = E_{\rm red}(\mu,\gamma,\gamma)$ for $\mu \in M^\ell$. By \eqref{derivative1.2} we have $$g'(\gamma) = \sum_{i=1,2}\partial_{\gamma_i} E_{\rm red}(\mu,\gamma,\gamma) = 2v_3'(\beta(\alpha^*,\gamma)) \partial_\gamma \beta(\alpha^*,\gamma),$$
where $\alpha^* = \alpha^*(\mu,\gamma,\gamma)$. Using \eqref{zigder3} and the fact that $v_3'(\beta(\alpha^*,\gamma)) < 0$ since  $\beta(\alpha^*,\gamma) < 2\pi/3$, we get $g'(\gamma) < 0$. Moreover, taking again  \eqref{zigder3} and  Lemma \ref{aus} into account, a Taylor expansion shows $|g'(\gamma)| \le C\ell^{-3}$ for some $C>0$ only depending on $v_3$. This  shows Property 3.

Finally, we show Property 4. The strict convexity of $\mu \mapsto E_{\rm red}(\mu,\gamma_\ell,\gamma_\ell)$  follows from \eqref{convexity-last-a} and  a continuity argument, exactly as in the proof of Proposition \ref{convexenergy}. To show that the mapping is \RRR strictly \EEE increasing for $\mu > \mu^{\rm us}_\ell$, we have to show that for $\mu > \mu^{\rm us}_\ell$
\begin{align}\label{larger than one}
\mu/2 + \lambda^\mu \cos\alpha^\mu >1 
\end{align}
as then the property follows from \eqref{derivative1}.  Using the monotonicity properties of $v_2$ we see that the first-order optimality conditions \eqref{eq: firstorder-opt} and \eqref{first order-a} imply
\begin{align}\label{eq: sgn}
\begin{split}
 \mu/2 + \lambda^\mu\cos\alpha^\mu> 1 \ \ \  \Leftrightarrow \ \ \ \lambda^\mu >1.
\end{split}
\end{align}
We prove \eqref{larger than one} by contradiction. Suppose  $\lambda^\mu \le  1$. This together with the fact $\mu > \mu^{\rm us}_\ell = 2 - 2\cos\alpha^{\rm us}_\ell$ \RRR (see \eqref{eq:muellneu}) \EEE and $\cos\alpha^\mu<0$  would imply by  \eqref{eq: sgn}
\begin{align}\label{eq: sgn2}
2\cos\alpha^\mu - 2\cos\alpha^{\rm us}_\ell +1 = \mu^{\rm us}_\ell -1 +2\cos\alpha^\mu <  \mu -1 +2\lambda^\mu\cos\alpha^\mu \le 1
\end{align}
and thus $\alpha^\mu > \alpha^{\rm us}_\ell$. By the  optimality condition in the unstretched case (see \eqref{first order-b} and recall that bond lengths are all equal to $1$) we get 
$$v_3'(\beta(\alpha^{\rm us}_\ell,\gamma_\ell))\, \partial_\alpha \beta(\alpha^{\rm us}_\ell,\gamma_\ell) + 2v_3'(\alpha^{\rm us}_\ell) = 0.$$
Consider the mapping $\alpha \mapsto v_3'(\beta(\alpha,\gamma_\ell))\, \partial_\alpha \beta(\alpha,\gamma_\ell) + 2v_3'(\alpha)$ and \PPP observe that  its derivative is \EEE 
$$v_3'(\beta(\alpha,\gamma_\ell))\, \partial^2_{\alpha\alpha} \beta(\alpha,\gamma_\ell) + v_3''(\beta(\alpha,\gamma_\ell))\, (\partial_\alpha \beta(\alpha,\gamma_\ell))^2  + 2v_3''(\alpha).$$
Thus, the mapping is strictly increasing in a left neighborhood of $2\pi/3$ by \eqref{zigder3} and the fact that $\beta(\alpha,\gamma_\ell) < 2\pi/3$. Since $\alpha^\mu > \alpha^{\rm us}_\ell$, this gives 
$$
v'_3(\beta(\alpha^\mu,\gamma_\ell)) \, \partial_\alpha\beta(\alpha^\mu,\gamma_\ell)   + 2 v'_3(\alpha^\mu)  >0.
$$
In view of \eqref{first order-b} and the first order optimality conditions \eqref{eq: firstorder-opt},    we get $\mu/2 + \lambda^\mu\cos\alpha^\mu > 1$, which contradicts the last inequality in  \eqref{eq: sgn2}. Consequently,   \eqref{larger than one} holds, which concludes the proof. 
\end{proof}

\RRR We close this section with the proof of Proposition \ref{th: main2}.\EEE

\begin{proof}[Proof of Proposition \ref{th: main2}]
Let $M^\ell$ be the interval given by Proposition \ref{th: mainenergy}. The strict convexity of the mapping $\mu \mapsto  E_{\rm min}(\mu)$ on $M^\ell$ as well as $\frac{d^2}{d\mu^2} E_{\rm min}(\mu_\ell^{\rm us}) \ge c2m\ell \ge cn$  follow from Properties 4 and 5 of Proposition \ref{th: mainenergy}. The fact that the energy minimum is attained at $\mu^{\rm us}_\ell$ follows from the definition of $\mu^{\rm us}_\ell$, see Proposition \ref{eq: old main result} \RRR and \eqref{eq:muellneu}. \EEE This shows Property 1.

Now consider Property 2. We define $\lambda^\mu_1=\mu/2 +\lambda^\mu \cos\alpha^\mu$, $\lambda^\mu_2 = \lambda^\mu$ \RRR with $\lambda^\mu$, $\alpha^\mu$ being the solution of \eqref{red} for $\mu$ and $\gamma_1 = \gamma_2 = \gamma_\ell$ (cf. Proposition   \ref{th: mainenergy}(v)) \EEE and use  \eqref{convexity3-b} to obtain $\partial_\mu \lambda_2^\mu({\RRR t_{\rm ref} \EEE }) = \partial_\mu \lambda^*({\RRR t_{\rm ref} \EEE }) = 1/K$ and $\partial_\mu \lambda_1^\mu({\RRR t_{\rm ref} \EEE }) = 1/2 - \partial_\mu \lambda^*({\RRR t_{\rm ref} \EEE })/2- \sqrt{3}\partial_\mu \alpha^*_1({\RRR t_{\rm ref} \EEE })/2 =  4/K$ with $K = 9 + v_2''(1)/(2 v_3''(2\pi/3))$. \RRR (Recall the definition $t_{\rm ref} = (3,\pi,\pi)$.) \EEE Consequently, by a standard continuity argument we see that $\lambda^\mu_1$ and $\lambda^\mu_2$ increase continuously for $\mu \in M^\ell$, possibly passing to a smaller (not relabeled)  open interval $M^\ell$ containing $\mu^{\rm us}_\ell$. The proof of the fact that $\mu > \mu^{\rm us}_\ell$ implies $\lambda^\mu_1,\lambda^\mu_2 >1$ is already contained in the proof of Proposition \ref{th: mainenergy}, see particularly \eqref{larger than one} and \eqref{eq: sgn}. The fact that  $\mu < \mu^{\rm us}_\ell$ implies $\lambda^\mu_1,\lambda^\mu_2 <1$ can be proved along similar lines.  

To see Property 3, recall that by Proposition \ref{eq: old main result} we have $ \RRR \alpha^{\rm us}_\ell  \EEE = \alpha^{\mu^{\rm us}_\ell} \in (\alpha^{\rm ch}_\ell,\alpha^{\rm ru})$ in the unstretched case. By a continuity argument we particularly obtain the convergence of minimizers,  \PPP i.e., \EEE $\alpha^\mu \to \alpha^{\mu^{\rm us}_\ell}$ as $\mu \to \mu_\ell^{\rm us}$. Consequently, again possibly passing to a smaller interval $M^\ell$,  Property 3 follows. We finally concern ourselves with Property 4. Recall by \eqref{alphars} that the radius of the nanotube is given by
$$\rho^\mu = \lambda_2^\mu \sin\alpha^\mu / (2\sin(\pi/(2\ell))).$$
We compute the derivative and obtain
$$\partial_\mu \rho^\mu = \big(\lambda_2^\mu \cos\alpha^\mu \, \partial_\mu \alpha^\mu + \partial_\mu \lambda_2^\mu\sin\alpha^\mu \big)/(2\sin(\pi/(2\ell))). $$
By \eqref{convexity3-b} the derivative at the unstrechted planar reference configuration  \PPP is \EEE
\begin{align*}
\lim_{\ell \to \infty} \partial_\mu \rho^{\mu^{\rm us}_\ell} \cdot (2\sin(\pi/(2\ell))) & =  - \frac{1}{2}  \partial_\mu \alpha_1^*({\RRR t_{\rm ref} \EEE }) + \frac{1}{2}\sqrt{3}\partial_\mu \lambda^*({\RRR t_{\rm ref} \EEE })  = \frac{\sqrt{3}}{2K}   \Big( 1 - \frac{v_2''(1)}{6v_3''(2\pi/3)} \Big).
\end{align*}
Consequently, whenever $v_2''(1) \neq 6v_3''(2\pi/3))$, by a continuity argument the sign of $\partial_\mu \rho^{\mu}$ for $\ell \in \Nz$ large in a small neighborhood of $\mu^{\rm us}_\ell$ only depends on the sign of $v_2''(1) - 6v_3''(2\pi/3)$.
\end{proof}

\section{Energy defect controls symmetry defect: Proof of Theorem \ref{th: Ered}}\label{sec: cellenery}
This section is devoted to the proof of Theorem \ref{th: Ered}. The
fact that the minimum of the cell energy is attained for a special
configuration with  high symmetry (see \eqref{sym-assumption})
essentially relies on convexity properties of the cell energy $E_{\rm
  cell}$ defined in \eqref{eq: cell}.  Throughout the section we
consider a cell consisting of eight points $\boldsymbol{x} =
(x_1,\ldots,x_8) \in \Rz^{3 \times 8}$ as defined before \eqref{eq:
  cell}, see Figure \ref{cell}. Likewise, the bond lengths are again denoted by   $b_1,\ldots,b_8$ and the angles by $\varphi_1,\ldots, \varphi_{10}$, \RRR see Figure \ref{cellangles}. \EEE With a slight abuse of notation we denote the cell energy for a given configuration $\boldsymbol{x}$ by $E_{\rm cell}(\boldsymbol{x})$.

 \subsection{Relation between atomic positions, bonds, and angles} 
 We will investigate the convexity properties of $E_{{\rm cell}}$ near the \emph{planar reference configuration} $\boldsymbol{x}^0 = (x^0_1,\ldots,x^0_8) \in \Rz^{3 \times 8}$ defined by  
 \begin{align*} 
 \begin{split}
&x_1^0 = (-1,0,0), \ \ \ x^0_2 = (1,0,0), \ \ \    x_3^0 =(-1/2,\sqrt{3}/2,0), \\ & x_4^0 =(1/2,\sqrt{3}/2,0),  \ \ \ x_5^0 =(1/2,-\sqrt{3}/2,0), \ \ x_6^0 =(-1/2,-\sqrt{3}/2,0), \\& x_7^0 = (-2,0,0), \ \ x_8^0 = (2,0,0).
\end{split}
 \end{align*}
Moreover, we introduce the \emph{unstretched kink configuration}  $\boldsymbol{x}^\ell_{\rm kink} = (x^{\rm kink}_1,\ldots,x^{\rm kink}_8) \in \Rz^{3 \times 8}$ by 
 \begin{equation}\label{kink} 
 \begin{aligned}
&x_1^{\rm kink} = (-1/2-\sigma^{\rm us},0,0), \\ &  x^{\rm kink}_2 = (1/2 + \sigma^{\rm us},0,0), \\
&  x_3^{\rm kink} =(-1/2,\sin\alpha^{\rm us}_\ell \sin(\gamma_\ell/2),\sin\alpha^{\rm us}_\ell \cos(\gamma_\ell/2)), \\ & x_4^{\rm kink} =(1/2,\sin\alpha^{\rm us}_\ell \sin(\gamma_\ell/2),\sin\alpha^{\rm us}_\ell \cos(\gamma_\ell/2)),  \\
&  x_5^{\rm kink} =(1/2,-\sin\alpha^{\rm us}_\ell \sin(\gamma_\ell/2),\sin\alpha^{\rm us}_\ell \cos(\gamma_\ell/2)), \\ & x_6^{\rm kink} =(-1/2,-\sin\alpha^{\rm us}_\ell \sin(\gamma_\ell/2),\sin\alpha^{\rm us}_\ell \cos(\gamma_\ell/2)), \\& x_7^{\rm kink} = (-3/2 - \sigma^{\rm us},0,0), \\ & x_8^{\rm kink} = (3/2 + \sigma^{\rm us},0,0),
\end{aligned}
 \end{equation}
where $\gamma_\ell=\pi(1-1/\ell)$ and $\sigma^{\rm us} = - \cos\alpha^{\rm us}_\ell$ with $\alpha^{\rm us}_\ell$ as given by Proposition \ref{eq: old main result} (cf. also \eqref{alphars}). Note that $\boldsymbol{x}^\ell_{\rm kink}$ represents the mutual position of atoms in a cell for the unstretched nanotube $\mathcal{G}_{\alpha^{\rm us}_\ell}$ found   in  Proposition \ref{eq: old main result}.  For later use we note that by  Lemma \ref{aus} and a Taylor expansion we find 
\begin{align}\label{eq: distance} 
|\boldsymbol{x}^0 -\boldsymbol{x}^\ell_{\rm kink}| \le C \ell^{-1}
\end{align}
for some universal $C>0$  large enough. In order to discuss the convexity properties of $E_{\rm cell}$ we need to introduce a specific basis of $\Rz^{3 \times 8}$,  \PPP i.e., \EEE  the space of cell configurations. This will consist of three collections of vectors, denoted by $\mathcal{V}_{\rm degen}$, $\mathcal{V}_{\rm good}$, and $\mathcal{V}_{\rm bad}$, where the sets are defined as follows:   We introduce the translations and infinitesimal rotations
\begin{align*}
\mathcal{V}_{\rm trans} &= \Big\{ ( {e}_1,\ldots, {e}_1),  ( e_2,\ldots, {e}_2), ( {e}_3,\ldots, {e}_3)\Big\} \subset \Rz^{3 \times 8} \\
\mathcal{V}_{\rm rot} &=  \left\{ 
\boldsymbol{v}_1 := \begin{pmatrix}
  0 & 1   & 0 \\
  -1 & 0  & 0 \\
  0  & 0 & 0  
 \end{pmatrix} \, \boldsymbol{x}^0, \ 
\boldsymbol{v}_2 :=\begin{pmatrix}
  0 & 0   & 1 \\
  0 & 0  & 0 \\
  -1  & 0 & 0   
 \end{pmatrix} \, \boldsymbol{x}^0, \ 
\boldsymbol{v}_3 :=\begin{pmatrix}
  0 & 0   & 0 \\
  0 & 0  & 1 \\
  0  & -1 & 0   
 \end{pmatrix} \, \boldsymbol{x}^0 \right\} \subset \Rz^{3 \times 8}
\end{align*}
and set $\mathcal{V}_{\rm degen}= \mathcal{V}_{\rm trans}  \cup
\mathcal{V}_{\rm rot}$.  The family $\mathcal{V}_{\rm good}$ contains
the 13 vectors 
\begin{align*}
\boldsymbol u_1=&( -1, 0, 0  |  1 , 0 , 0 | - 1/2, \sqrt{3}/2 , 0  |  1/2 , \sqrt{3}/2, 0  | 1/2 , -\sqrt{3}/2 , 0|  - 1/2 , -\sqrt{3}/2 , 0 | 0 , 0 ,  0  |  0 , 0 , 0),\\
\boldsymbol u_2=& (0 , 0 , 0 \, | \,  0 , 0 , 0  \, | \,  1/2 , \sqrt{3}/2 , 0  \, | \,  - 1/2 , \sqrt{3}/2 , 0   \, | \,  0 , 0 , 0  \, | \,  0 , 0 , 0  \, | \,  0 , 0 , 0  \, | \,  0 , 0 , 0),\\ 
\boldsymbol u_3=&(0 , 0 , 0  \, | \, 1 , 0 , 0 \, | \, 0 , 0 , 0 \, | \, 1 , 0 , 0  \, | \, 1 , 0 , 0 \, | \, 0 , 0 , 0 \, | \, 0 , 0 , 0 \, | \, 0 , 0 , 0),\\ 
\boldsymbol u_4=&(0 , 0 , 0 \, | \, 1/2 , -\sqrt{3}/2 , 0 \, | \, 1/2 , \sqrt{3}/2 , 0 \, | \, - 1/2 , \sqrt{3}/2 , 0  \, | \, 1 , 0 , 0 \, | \, 0 , 0 , 0 \, | \, 0 , 0 , 0 \, | \, 0 , 0 , 0),\\ 
\boldsymbol u_5=&(0 , 0 , 0 \, | \, 0 , 0 , 0 \, | \, 0 , 0 , 0 \, | \, 0 , 0 , 0 \, | \, 0 , 0 , 0 \, | \, 0 , 0 , 0 \, | \, -1 , 0 , 0 \, | \, 0 , 0 , 0),\\
\boldsymbol u_6=&(0 , 0 , 0 \, | \, 0 , 0 , 0 \, | \, 0 , 0 , 0 \, | \, 0 , 0 , 0 \, |
\, 0 , 0 , 0 \, | \, 0 , 0 , 0 \, | \, -1 , 0 , 0 \, | \, 1 , 0 ,
0),\\
\boldsymbol u_7=&(\sqrt{3} , 0 , 0 \, | \, 0 , 0 , 0 \, | \, 0 , 1 , 0 \, | \, 0 , 0 , 0  \, | \, 0 , 0 , 0 \, | \, 0 , -1 , 0 \, | \, 0 , 0 , 0  \, | \, 0 , 0 , 0),\\ 
\boldsymbol u_8=&(0 , 0 , 0 \, | \, 0 , 0 , 0 \, | \, \sqrt{3}/2 , - 1/2 , 0 \, | \, \sqrt{3}/2 , 1/2 , 0  \, | \, 0 , 0 , 0 \, | \, 0 , 0 , 0 \, | \, 0 , 0 , 0 \, | \, 0 , 0 , 0),\\ 
\boldsymbol u_9=&(\sqrt{3}/2, 1/2 , 0 \, | \, -\sqrt{3}/2 , 1/2 , 0 \, | \, 0 , 1 , 0 \, | \, 0 , 1 , 0  \, | \, 0 , 0 , 0 \, | \, 0 , 0 , 0 \, | \, 0 , 0 , 0 \, | \, 0 , 0 , 0),\\ 
\boldsymbol u_{10}=&(0 , 0 , 0 \, | \, 0 , 0 , 0 \, | \, 0 , 0 , 0 \, | \, 0 , 0 , 0 \, | \, 0 , 0 , 0 \, | \, 0 , 0 , 0 \, | \, 0 , 1 , 0 \, | \, 0 , 0 , 0),\\
\boldsymbol u_{11}=&(0 , 0 , 0 \, | \, 0 , 0 , 0 \, | \, 0 , 0 , 0 \, | \, 0 , 0 , 0 \, |
\, 0 , 0 , 0 \, | \, 0 , 0 , 0 \, | \, 0 , 1 , 0 \, | \, 0 , 1 , 0),\\
\boldsymbol u_{12}=&(1 , 0 , 0 \, | \, 0 , 0 , 0 \, | \, 0 , 0 , 0 \, | \, 0 , 0 , 0 \, | \, 0 , 0 , 0 \, | \, 0 , 0 , 0 \, | \, 0 , 0 , 0 \, | \, 0 , 0 , 0),\\ 
\boldsymbol u_{13}=&(0 , 1 , 0 \, | \, 0 , 0 , 0 \, | \, 0 , 0 , 0 \, | \, 0 , 0 , 0 \, | \, 0 , 0 , 0 \, | \, 0 , 0 , 0 \, | \, 0 , 0 , 0 \, | \, 0 , 0 , 0).
\end{align*}
The first 6 vectors keep the angles fixed and \PPP modify only the bond lengths, \EEE
see Figure \ref{Eigenvectors_angle_fixed}. The vectors $\boldsymbol
u_8, \dots , \boldsymbol u_{11}$ keep the bond lengths fixed \RRR to first order \EEE and change the
angles, see Figure \ref{Eigenvectors_sides_fixed}. Eventually, the
remaining  vectors $\boldsymbol
u_{12}$ and $ \boldsymbol u_{13}$  modify both angles and bonds as in Figure \ref{Eigenvectors_nofixed}.

 \begin{figure}[htp]
\begin{center}
\pgfdeclareimage[width=0.9\textwidth]{Eigenvectors_angle_fixed}{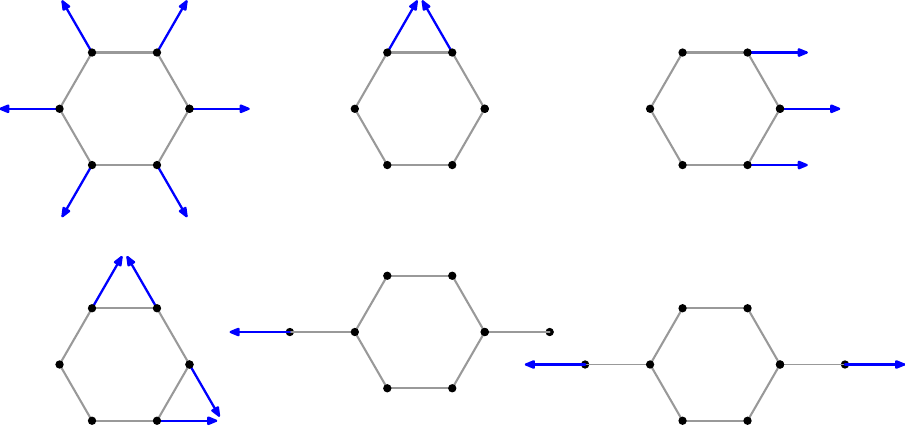}
    \pgfuseimage{Eigenvectors_angle_fixed}
\caption{Vectors $\boldsymbol u_1, \dots, \boldsymbol u_6$ in  $\mathcal{V}_{\rm good}$ keep the angles fixed (ordered from left to right  both in the first and in the second line).}
\label{Eigenvectors_angle_fixed}
\end{center}
\end{figure}

 \begin{figure}[tp]
\begin{center}
\pgfdeclareimage[width=0.8\textwidth]{Eigenvectors_sides_fixed}{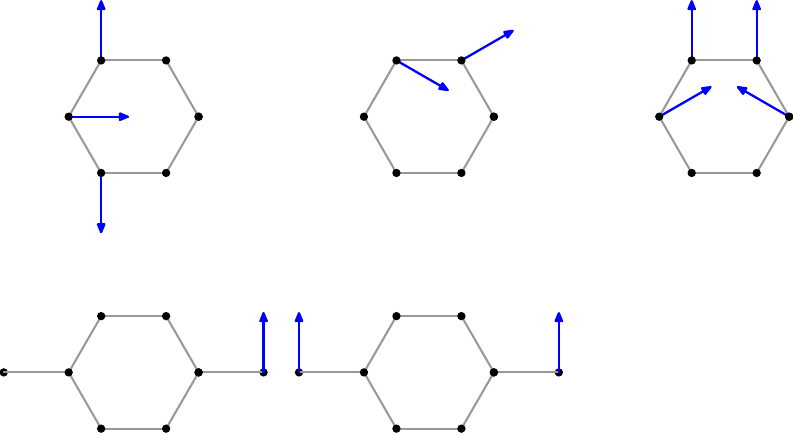}
    \pgfuseimage{Eigenvectors_sides_fixed}
\caption{Vectors $\boldsymbol u_{7}, \dots, \boldsymbol u_{11}$ in  $\mathcal{V}_{\rm good}$  keep the bond lengths fixed  (ordered from left to right  both in the first and in the second line).}
\label{Eigenvectors_sides_fixed}
\end{center}
\end{figure}
 
\begin{figure}[htp]
\begin{center}
\pgfdeclareimage[width=0.45\textwidth]{Eigenvectors_nofixed}{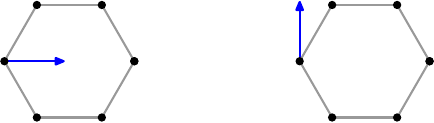}
    \pgfuseimage{Eigenvectors_nofixed}
\caption{Vectors $\boldsymbol u_{12} $ and $\boldsymbol u_{13}$ in
  $\mathcal{V}_{\rm good}$   keep neither angles nor bond lengths fixed (ordered from left to right).}
\label{Eigenvectors_nofixed}
\end{center}
\end{figure}

 By $\mathcal{V}_{\rm bad}$ we denote the collection of the vectors
\begin{align*}
& (0 , 0 , 1 \, | \, 0 , 0 , 0 \, | \, 0 , 0 , 0 \, | \, 0 , 0 , 0  \, | \, 0 , 0 , 0 \, | \, 0 , 0 , 0 \, | \, 0 , 0 , 0 \, | \, 0 , 0 , 0),\\ 
&(0 , 0 , 1 \, | \, 0 , 0 , 0 \, | \, 0 , 0 , 1 \, | \, 0 , 0 , 0  \, | \, 0 , 0 , 0 \, | \, 0 , 0 , 0 \, | \, 0 , 0 , 0 \, | \, 0 , 0 , 0),\\ 
&(0 , 0 , 1 \, | \, 0 , 0 , 0  \, | \, 0 , 0 , 0 \, | \, 0 , 0 , 1 \, | \, 0 , 0 , 1 \, | \, 0 , 0 , 0 \, | \, 0 , 0 , 0 \, | \, 0 , 0 , 0),\\
&(0 , 0 , 0 \, | \, 0 , 0 , 0 \, | \, 0 , 0 , 0 \, | \, 0 , 0 , 0 \, | \, 0 , 0 , 0 \, | \, 0 , 0 , 0 \, | \, 0 , 0 , 1 \, | \, 0 , 0 , 0),\\
&(0 , 0 , 0 \, | \, 0 , 0 , 0 \, | \, 0 , 0 , 0 \, | \, 0 , 0 , 0 \, | \, 0 , 0 , 0 \, | \, 0 , 0 , 0 \, | \, 0 , 0 , 1 \, | \, 0 , 0 , 1).
\end{align*}
It is elementary to check that the vectors $\mathcal{V}_{\rm degen} \cup \mathcal{V}_{\rm good} \cup \mathcal{V}_{\rm bad}$ are linearly independent and thus form a basis of $\Rz^{3 \times 8}$. Note that    the vectors in $\mathcal{V}_{\rm good}$ are perpendicular to the vectors in $\mathcal{V}_{\rm bad}$.

Clearly, the cell energy is strictly convex as a function of the bond lengths and angles by the assumptions on the potentials $v_2$ and $v_3$. Our goal is to show that the same property holds if the cell energy is given as a function of the atomic positions. To this end, we introduce the mapping $T = (T^a,T^b): \Rz^{3 \times 8} \to \Rz^{18}$ defined by
\begin{align*}
T^a_i(\boldsymbol{x}) = \varphi_{i} \ \text{ for } i= 1,\ldots,10, \ \ \ \ T^b_i(\boldsymbol{x}) &= b_i \ \text{ for } i= 1,\ldots,8.
\end{align*}
Then the cell energy reads as 
\begin{align}\label{eq: cell-energy}
E_{{\rm cell}}(\boldsymbol{x}) = \sum_{i=1}^8 \kappa^b_i v_2(T^b_i(\boldsymbol{x}))  +\sum_{i=1}^{10} \kappa^a_i v_3(T^a_i(\boldsymbol{x}))  
\end{align}
with the factors $\kappa^b_1 = \kappa^b_2 = \kappa^b_7= \kappa^b_8 = 1/4$, $\kappa^b_3 = \kappa^b_4 = \kappa^b_5=  \kappa^b_6 = 1/2$, $\kappa^a_1 = \kappa^a_{2} = 1$, $\kappa^a_{3} = \ldots =  \kappa^a_{10} =  1/2$.

{\BBB Before analyzing the mapping $T$, we need to introduce some more notation for the sum of angles $\varphi_i$. From here on, we denote by $\boldsymbol{e}_1,\ldots,\boldsymbol{e}_{10}$ the canonical basis of $\mathbb{R}^{10}$, and we let
$$\boldsymbol{a}_1 := \boldsymbol{e}_1 + \ldots + \boldsymbol{e}_6,\quad \boldsymbol{a}_2 := \boldsymbol{e}_1 + \boldsymbol{e}_7 + \boldsymbol{e}_8,\quad \boldsymbol{a}_3: = \boldsymbol{e}_2 + \boldsymbol{e}_9 + \boldsymbol{e}_{10}$$ be vectors in $\Rz^{10}$.  Elementary geometry yields $T^a(\boldsymbol{x}^0)\cdot  \boldsymbol{a}_1 = 4\pi$ and $T^a(\boldsymbol{x}^0) \cdot \boldsymbol{a}_j = 2\pi$ for $j=2,3$ as well as  $T^a(\boldsymbol{x})\cdot \boldsymbol{a}_1 \le 4\pi$ and $T^a(\boldsymbol{x}) \cdot \boldsymbol{a}_j \le 2\pi$ for $j=2,3$ for each $\boldsymbol{x}\in\Rz^{3\times 8}$. Indeed, the sum of the interior angles in a hexagon is always smaller or equal to $4\pi$ and exactly $4\pi$ if the hexagon is planar. Likewise one argues for a triple junction. 

}

\begin{lemma}[Properties of $T$]\label{lemma: T}
The mapping $T$ is smooth in a neighborhood of $\boldsymbol{x}^0$. There is a  constant $c_{\rm kink} >0$ such that  
\begin{align*} 
1.& \ \ {\rm Ker}(D T(\boldsymbol{x}^0)) = {\rm span}(\mathcal{V}_{\rm degen} \cup \mathcal{V}_{\rm bad}), \ \ \ \ {\rm dim}({\rm Ker}(D T(\boldsymbol{x}^0))) = 11,\\
2.& \ \ {\rm dim}({\rm Ker}(D T^a(\boldsymbol{x}^0))) = 17,\\
3.& \ \ (\boldsymbol{v}^T D^2 T^a(\boldsymbol{x}^0) \boldsymbol{v})\cdot \boldsymbol{a}_j \le 0  \text{ for } \ j=1,2,3, \ \ \text{ for all } \ \ \boldsymbol{v} \in \Rz^{3 \times 8},\\
4.& \ \  \sum_{j=1}^{3}(\boldsymbol{v}^T D^2 T^a(\boldsymbol{x}^0)
\boldsymbol{v}) \cdot \boldsymbol{a}_j\le -c_{\rm kink}|\boldsymbol{v}
- \boldsymbol{v}_{\rm degen}|^2  \ \ \text{ for all } \ \
\boldsymbol{v} \in {\rm span}(\mathcal{V}_{\rm degen} \cup
\mathcal{V}_{\rm bad}), \\ & \ \ \ \text{ where $\boldsymbol{v}_{\rm
    degen}$ is the orthogonal  projection of $\boldsymbol{v}$ onto
  ${\rm span}(\mathcal{V}_{\rm degen})$.} 
\end{align*}

\end{lemma}


\begin{proof}
First, to see Property 1, we note that ${\rm span}(\mathcal{V}_{\rm
  degen} \cup \mathcal{V}_{\rm bad})$ is a subset of ${\rm Ker}(D
T(\boldsymbol{x}^0))$ since each vector in $\mathcal{V}_{\rm degen}
\cup \mathcal{V}_{\rm bad}$ does not change bond lengths and angles to
first order. On the other hand, each vector in $\mathcal{V}_{\rm
  good}$ changes bond lengths or angles to first order and is
therefore not contained in the kernel of $D
T(\boldsymbol{x}^0)$. Indeed, the first six vectors of
$\mathcal{V}_{\rm good}$ are directions of perturbations that do not
change angles \PPP to first order, \EEE but bond lengths. Vectors $\boldsymbol
u_7, \dots, \boldsymbol u_{11}$ are perturbations that do not change
bond lengths in first order, but angles. Vectors $ \boldsymbol u_{12}$ and
$\boldsymbol u_{13}$ are in-plane displacements of a single atom and
change both bond lengths and angles to first order.  \RRR More precisely, \vspace{0.2cm} for the changes of bond lengths we get\\ \begin{tabular}{ll}\vspace{0.1cm}
$DT^b(\boldsymbol{x}^0) \boldsymbol{u}_1 \,\|\, (1,1,1,1,1,1,-1,-1)$, & $DT^b(\boldsymbol{x}^0) \boldsymbol{u}_2 \,\|\, (0,-1,1,1,0,0,0,0)$,\\ \vspace{0.1cm}
$DT^b(\boldsymbol{x}^0) \boldsymbol{u}_3 \,\|\, (1,1,0,0,0,0,0,-1)$, & $DT^b(\boldsymbol{x}^0) \boldsymbol{u}_4 \,\|\, (2,-2,2,4,-2,0,0,-1)$,\\ \vspace{0.1cm}
$DT^b(\boldsymbol{x}^0) \boldsymbol{u}_5 \,\|\, (0,0,0,0,0,0,1,0)$, & $DT^b(\boldsymbol{x}^0) \boldsymbol{u}_6 \,\|\, (0,0,0,0,0,0,1,1)$,\\ \vspace{0.2cm}
$DT^b(\boldsymbol{x}^0) \boldsymbol{u}_{12} \,\|\, (0,0,-1,0,0,-1,2,0)$, & $DT^b(\boldsymbol{x}^0) \boldsymbol{u}_{13} \,\|\, (0,0,-1,0,0,1,0,0)$,
  \end{tabular}\\ 
where $\boldsymbol{w}_1 \,\|\, \boldsymbol{w}_2$ indicates that $\boldsymbol{w}_1$ and $\boldsymbol{w}_2$ are linearly dependent. Likewise, for the changes of angles we have  \vspace{0.2cm}  \\   
\begin{tabular}{ll}\vspace{0.1cm}
$DT^a(\boldsymbol{x}^0) \boldsymbol{u}_7 \,\|\, (4,0,-3,1,1,-3,-2,-2,0,0),$ & $DT^a(\boldsymbol{x}^0) \boldsymbol{u}_8 \,\|\, (-1,1,2,-2,0,0,1,0,-1,0)$,  \\  \vspace{0.1cm}
  $DT^a(\boldsymbol{x}^0) \boldsymbol{u}_9 \,\|\, (-2,-2,1,1,1,1,1,1,1,1)$, & $DT^a(\boldsymbol{x}^0) \boldsymbol{u}_{10} \,\|\, (0,0,0,0,0,0,0,0,-1,1)$,  \\ \vspace{0.1cm}
 $DT^a(\boldsymbol{x}^0) \boldsymbol{u}_{11} \,\|\, (0,0,0,0,0,0,-1,1,-1,1)$, & $DT^a(\boldsymbol{x}^0) \boldsymbol{u}_{12} \,\|\, (2,0,-1,0,0,-1,-1,-1,0,0)$,\\ \vspace{0.1cm}
$DT^a(\boldsymbol{x}^0) \boldsymbol{u}_{13} \,\|\, (0,0,0,0,0,0,1,-1,0,0).$ & \vspace{0.2cm}
  \end{tabular}\\
 \EEE (We prefer not to give
details of the computation, but rather refer the \PPP reader  \EEE to Figures \ref{Eigenvectors_angle_fixed}-\ref{Eigenvectors_nofixed} where the situation of the different directions is indicated.). \RRR It is elementary to check that the vectors $DT(\boldsymbol{x}^0) \boldsymbol{u}_i$, $i =1,\ldots, 13$, are linearly independent which concludes the proof of Property 1 \PPP by dimension counting. \EEE 
\EEE

Since  $ {\rm dim}({\rm Ker}(D T(\boldsymbol{x}^0))) =11$ and in $\mathcal{V}_{\rm good}$ only the first six vectors do not change angles to first order,   Property 2 holds. 

Property 3 follows from the fact that the mapping $t \mapsto T^a(\boldsymbol{x}^0 + t\boldsymbol{v})\cdot \boldsymbol{a}_j$ has a local maximum at $t= 0$  for $j=1,2,3$ and for all $\boldsymbol{v} \in \Rz^{3 \times 8}$ as noticed before the statement of the lemma.

To see Property 4,  we first consider the special case $\boldsymbol{v} \in \mathcal{V}_{\rm bad}$. In this situation the property follows from an elementary computation, which we detail only in the case $\boldsymbol{v} = (e_3 |0| \ldots|0)$. In this case, after some calculations, we obtain $(T^a(\boldsymbol{x}^0 + t\boldsymbol{v}))_i = \arccos(-1/2 + 3t^2/2)  + {\rm O}(t^3) \le  2\pi/3 -ct^2$ for some $c>0$ for $i=1,7,8$, \PPP i.e.,  for the angles at the triple junction at point $x_1$. \EEE Using also  Property 1, this indeed implies $(\boldsymbol{v}^T D^2 T^a(\boldsymbol{x}^0) \boldsymbol{v}) \cdot \boldsymbol{a}_2 \le -c$, \PPP i.e., \EEE by a perturbation out of the plane the sum of the angles is reduced to second order. For the other triple junction and the interior angles of the hexagon we argue analogously. This shows the property for perturbations in the directions $\mathcal{V}_{\rm bad}$. \BBB Likewise, we proceed for directions in ${\rm span}(\mathcal{V}_{\rm bad})$. \EEE   

Now consider the general case $\boldsymbol{v} = \boldsymbol{v}_{\rm trans} + \boldsymbol{v}_{\rm rot} + \boldsymbol{v}_{\rm bad} \in {\rm span}(\mathcal{V}_{\rm degen} \cup \mathcal{V}_{\rm bad})$ for $\boldsymbol{v}_{\rm trans} \in {\rm span}(\mathcal{V}_{\rm trans})$, $\boldsymbol{v}_{\rm rot} \in {\rm span}(\mathcal{V}_{\rm rot})$, and $\boldsymbol{v}_{\rm bad} \BBB \in {\rm span}(\mathcal{V}_{\rm bad}) \EEE $. 

First, since $T(\boldsymbol{x} + \boldsymbol{w}) =  T(\boldsymbol{x})$ for all $\boldsymbol{x} \in \Rz^{3 \times 8}$ and all $\boldsymbol{w} \in \mathcal{V}_{\rm trans}$, we get $DT(\boldsymbol{x}) \boldsymbol{w} = 0$ and $\boldsymbol{w}^TD^2 T(\boldsymbol{x}) \boldsymbol{w}' = 0$ for all $\boldsymbol{w} \in {\rm span}(\mathcal{V}_{\rm trans})$, $\boldsymbol{w}' \in \Rz^{3 \times 8}$, and $\boldsymbol{x} \in \Rz^{3 \times 8}$. Consequently, we deduce  $\boldsymbol{v}^TD^2 T^a(\boldsymbol{x}^0) \boldsymbol{v} = (\boldsymbol{v}_{\rm rot} + \boldsymbol{v}_{\rm bad})^TD^2 T^a(\boldsymbol{x}^0) (\boldsymbol{v}_{\rm rot} + \boldsymbol{v}_{\rm bad})$. 

Moreover, let $A \in \Rz^{3 \times 3}_{\rm skew}$ \PPP be \EEE    such that $\boldsymbol{v}_{\rm rot} = A \boldsymbol{x}^0$ and observe that there is a rotation $R_t \in SO(3)$ such that 
$\boldsymbol{x}^0_t := R_t (\boldsymbol{x}^0 + t \boldsymbol{v}_{\rm rot})$ is contained in the plane $\Rz^2 \times \lbrace 0 \rbrace$ and one has $|R_t - (\mathbf{I} - tA)| = {\rm O}(|tA|^2)$, cf. \BBB \cite[(3.20)]{FrieseckeJamesMueller:02}. \EEE (Here $\mathbf{I} \in \Rz^{3 \times 3}$ denotes the identity matrix.) Consequently, we get $|\boldsymbol{x}^0 - \boldsymbol{x}^0_t| = {\rm O}(|tA|^2)$. This  implies
\begin{align*}
T^a(\boldsymbol{x}^0 + t(\boldsymbol{v}_{\rm rot} + \boldsymbol{v}_{\rm bad}))& = T^a\big(R_t(\boldsymbol{x}^0 + t(\boldsymbol{v}_{\rm rot} + \boldsymbol{v}_{\rm bad}))\big) = T^a(\boldsymbol{x}^0_t + t R_t\boldsymbol{v}_{\rm bad}) \\
&= T^a(\boldsymbol{x}^0 + t \boldsymbol{v}_{\rm bad} + t^2 \boldsymbol{w} + {\rm O}(t^3))
\end{align*}
for some $\boldsymbol{w} \in \Rz^{3 \times 8}$ with $|\boldsymbol{w}| \le c|A|^2$ and the property that the third component of each vector in $\boldsymbol{w}$ is zero. A Taylor expansion and Property 1 of the lemma then yield
\begin{align*}
T^a(\boldsymbol{x}^0 + t(\boldsymbol{v}_{\rm rot} + \boldsymbol{v}_{\rm bad}))& = T^a(\boldsymbol{x}^0) + t^2DT^a(\boldsymbol{x}^0)\boldsymbol{w} + \frac{t^2}{2}\boldsymbol{v}_{\rm bad}^TD^2T^a(\boldsymbol{x}^0)\boldsymbol{v}_{\rm bad} + {\rm O}(t^3).
\end{align*}
As the sum of the angles in the hexagon and at the triple junctions remains invariant under perturbation $\boldsymbol{w}$, we then deduce
$$\sum_{j=1}^{3} T^a(\boldsymbol{x}^0 + t(\boldsymbol{v}_{\rm rot} + \boldsymbol{v}_{\rm bad})) \cdot \boldsymbol{a}_j = 8\pi + \sum_{j=1}^{3}\frac{t^2}{2}\boldsymbol{v}_{\rm bad}^TD^2T^a(\boldsymbol{x}^0)\boldsymbol{v}_{\rm bad} \cdot \boldsymbol{a}_j + {\rm O}(t^3).
$$
The desired result now follows from the fact that $\sum_{j=1}^3\boldsymbol{v}_{\rm bad}^TD^2T^a(\boldsymbol{x}^0)\boldsymbol{v}_{\rm bad} \cdot \boldsymbol{a}_j \le -c|\boldsymbol{v}_{\rm bad}|^2$ has already been established in the first part of the proof, \BBB where we also note that $|\boldsymbol{v}_{\rm bad}| \ge c |\boldsymbol{v}
- \boldsymbol{v}_{\rm degen}|$ with $\boldsymbol{v}_{\rm
    degen}$ being the orthogonal  projection of $\boldsymbol{v}$ onto
  ${\rm span}(\mathcal{V}_{\rm degen})$. \EEE
\end{proof}

For later purpose we also introduce  the mapping $\tilde{E}:  {\RRR [0,2\pi]^{10}\times [0,+\infty)^8} \to \Rz$ defined by 
$$
\tilde{E}(\boldsymbol y) = \sum_{i=1}^{10} \kappa^a_i v_3(y_{i})  + \sum_{i=1}^{8} \kappa^b_i v_2(y_{i+10})
$$
for $\boldsymbol y \in {\BBB [0,2\pi]^{10}\times [0,+\infty)^8}$.  Note that $E_{\rm cell}(\boldsymbol{x}) = \tilde{E}(T(\boldsymbol{x}))$ for all $\boldsymbol{x} \in \Rz^{3\times 8}$.

\begin{lemma}[Properties of $\tilde{E}$]\label{tildeE}
The mapping $\tilde{E}$ is \RRR smooth \EEE and there are constants $0 < c_{E,1}< c_{E,2}$ and $\ell_0 \in \Nz$ depending only on $v_2$ and $v_3$ such that for $\ell \ge \ell_0$
\begin{align*}
\begin{split}
1.& \ \ (D \tilde{E}(T(\boldsymbol{x}_{\rm kink}^\ell)))_i =0 \ \ \ \text{ for } \ i=11,\ldots,18,\\
2.& \ \  -c_{E,2} \ell^{-2} \le (D\tilde{E}(T(\boldsymbol{x}_{\rm kink}^\ell)))_i \le -c_{E,1} \ell^{-2} \ \ \ \text{ for } \ i=1,\ldots,10,\\
3.& \ \  c_{E,1} \le (D^2 \tilde{E}(T(\boldsymbol{x}_{\rm kink}^\ell)))_{ii} \le c_{E,2}     \ \  \text{ for } \ i=1,\ldots,18, \ \ \ (D^2 \tilde{E}(T(\boldsymbol{x}_{\rm kink}^\ell)))_{ij} = 0 \ \  \text{ for } i\neq j. 
\end{split}
\end{align*}
\end{lemma}

\begin{proof}
Property 1 follows from the fact that $T^b(\boldsymbol{x}^\ell_{\rm kink}) = (1,\ldots,1) \in \Rz^8$ and $v_2'(1) = 0$. To see Property 2, we apply Lemma \ref{aus} to find $(T^a(\boldsymbol{x}^\ell_{\rm kink}))_i \in (2\pi/3 - c_2\ell^{-2},   2\pi/3 - c_1\ell^{-2})$ for $i=1,\ldots,10$ and the fact that $v_3 \in C^2$ with $v_3'(2\pi/3) = 0$, $v_3''(2\pi/3)>0$. Likewise, Property 3 follows from $v_2''(1)>0$ and $v''_3(2\pi/3)>0$, respectively. 
\end{proof}

\subsection{Convexity of the cell energy} \label{sec: convexity}
The following theorem gives a first property of the Hessian of $E_{\rm cell}$ at the kink configuration $\boldsymbol{x}^\ell_{\rm kink}$.

\begin{theorem}[Convexity of $E_{\rm cell}$ in good directions]\label{th: cell convexity1}
Let $0 < r <1$. Then there exist $\ell_0 \in \Nz$ and a constant $c>0$ depending only on $v_2$, $v_3$, and $r$ such that for $\ell \ge \ell_0$ and each $\boldsymbol{v} \in \Rz^{3 \times 8}$ with 
$$|\boldsymbol{v} \cdot \boldsymbol{w}| \le r|\boldsymbol{w}||\boldsymbol{v}|  \ \ \ \text{ for all } \ \ \ \boldsymbol{w}  \in {\rm span}(\mathcal{V}_{\rm degen} \cup \mathcal{V}_{\rm bad})$$
 one has
$$\boldsymbol{v}^TD^2E_{{\rm cell}}(\boldsymbol{x}^\ell_{\rm kink})\boldsymbol{v} \ge c|\boldsymbol{v}|^2.$$
\end{theorem}

\begin{proof} First, by the  regularity of the mapping $T$,  Property 1 in Lemma \ref{lemma: T}, and the fact that $\boldsymbol{x}^\ell_{\rm kink} \to \boldsymbol{x}^0$ for $\ell \to \infty$, we find  $\ell_0 \in \Nz$ sufficiently large such that for $\ell \ge \ell_0$ the kernel of $D T(\boldsymbol{x}^\ell_{\rm kink})$ has  dimension at most $11$. Then we find universal constants $0< c_1 < c_2$  such that for all $\ell \ge \ell_0$, possibly for a larger $\ell_0$, we have
\begin{align}\label{eq: conv1}
\begin{split}
 &  c_1|\boldsymbol{v}| \le |D T(\boldsymbol{x}^\ell_{\rm kink}) \boldsymbol{v}| \le c_2|\boldsymbol{v}|   \ \text{ for all } \boldsymbol{v}\in  {\rm span}(\mathcal{V}_{\rm degen} \cup \mathcal{V}_{\rm bad})^\bot,\\
 &|D T(\boldsymbol{x}^\ell_{\rm kink}) \boldsymbol{v}| \le c_2| \boldsymbol{v}| \ell^{-1}  \ \text{ for all } \boldsymbol{v}\in  {\rm span}(\mathcal{V}_{\rm degen} \cup \mathcal{V}_{\rm bad}).
  \end{split}
  \end{align}
  For the second property we  used \eqref{eq: distance}.   Let  be given  $\boldsymbol{v} \in \Rz^{3 \times 8}$ with $|\boldsymbol{v} \cdot \boldsymbol{w}| \le r|\boldsymbol{w}||\boldsymbol{v}|  \ \text{ for all } \ \boldsymbol{w}  \in {\rm span}(\mathcal{V}_{\rm degen} \cup \mathcal{V}_{\rm bad})$. The vector can be written as $\boldsymbol{v} = \boldsymbol{v}_{\rm good} + \boldsymbol{v}_{\rm good}^\bot$ with two orthogonal vectors $\boldsymbol{v}_{\rm good}, \boldsymbol{v}_{\rm good}^\bot$ satisfying $\boldsymbol{v}_{\rm good}^\bot \in {\rm span}(\mathcal{V}_{\rm degen} \cup \mathcal{V}_{\rm bad})$ and $|\boldsymbol{v}_{\rm good}| \ge \sqrt{1-r^2}|\boldsymbol{v}|$. Consider the mapping $f_{\boldsymbol{v}}:\Rz \to \Rz$ defined by $f_{\boldsymbol{v}}(t) = \tilde{E}(T(\boldsymbol{x}^\ell_{\rm kink} + t\boldsymbol{v}))$. We compute
\begin{align}
f'_{\boldsymbol{v}}(t) & = D \tilde{E}(T(\boldsymbol{x}^\ell_{\rm kink} + t\boldsymbol{v}))  \big(D T(\boldsymbol{x}^\ell_{\rm kink} + t\boldsymbol{v}) \boldsymbol{v}\big), \notag\\
f''_{\boldsymbol{v}}(t) & =  \big(D T(\boldsymbol{x}^\ell_{\rm kink} + t\boldsymbol{v}) \boldsymbol{v}\big)^T D^2 \tilde{E}(T(\boldsymbol{x}^\ell_{\rm kink} + t\boldsymbol{v}))  \big(D T(\boldsymbol{x}^\ell_{\rm kink} + t\boldsymbol{v}) \boldsymbol{v}\big) \notag \\& \ \ \ + D \tilde{E}(T((\boldsymbol{x}^\ell_{\rm kink} + t\boldsymbol{v})) \big(\boldsymbol{v}^TD^2 T(\boldsymbol{x}^\ell_{\rm kink} + t\boldsymbol{v}) \boldsymbol{v}\big). \label{eq: conv3-b}
\end{align}
We further observe that by Lemma \ref{tildeE}, Property 1 and 2, there is a constant $c_3$ only depending on $c_{E,2}$   such that 
\begin{align}\label{eq: conv2}
|D \tilde{E}(T(\boldsymbol{x}^\ell_{\rm kink})) \big(\boldsymbol{v}^TD^2 T(\boldsymbol{x}^\ell_{\rm kink}) \boldsymbol{v}\big)| \le c_3|\boldsymbol{v}|^2 \ell^{-2}.
\end{align}
Then collecting    \eqref{eq: conv1}-\eqref{eq: conv2} and using Property 3 of Lemma \ref{tildeE}   we derive
\begin{align*}
\boldsymbol{v}^TD^2E_{{\rm cell}}(\boldsymbol{x}^\ell_{\rm kink})\boldsymbol{v} & = f_{\boldsymbol{v}}''(0) \\&\ge c_{E,1}c_1^2|\boldsymbol{v}_{\rm good}|^2 -2 c_{E,2} c^2_2|\boldsymbol{v}_{\rm good}||\boldsymbol{v}_{\rm good}^\bot| \ell^{-1}  - c_3|\boldsymbol{v}|^2 \ell^{-2} \\
& \ge |\boldsymbol{v}|^2 \big(c_{E,1}c_1^2(1-r^2) - 2c_{E,2}c^2_2 \ell^{-1}  - c_3 \ell^{-2}\big).
\end{align*}
For $\ell_0$ large enough \RRR (depending also on $r$) \EEE this implies the assertion of the lemma for $\ell \ge \ell_0$.  
\end{proof}

To investigate the convexity properties in the directions $\mathcal{V}_{\rm bad}$, we need some further preparations. Recall the reflections introduced in \eqref{reflexion}. The following lemma is a consequence of Theorem \ref{th: cell convexity1} and shows that variations in  the directions $\mathcal{V}_{\rm good}$ decrease the energy only to higher order.

\begin{lemma}[Energy decrease in good directions]\label{lemma: in plane}
There exist $\ell_0 \in \Nz$ and a constant $C>0$ depending only on $v_2$ and $v_3$ such that for $\ell \ge \ell_0$ and each $\boldsymbol{v} \in {\rm span}(\mathcal{V}_{\rm good})$ 
$$D \tilde{E}(T(\boldsymbol{x}^\ell_{\rm kink}))  \big(D T(\boldsymbol{x}^\ell_{\rm kink}) \boldsymbol{v}\big) \ge - C|\boldsymbol{v}| \ell^{-3}.$$
\end{lemma}
\begin{proof}
Let $\boldsymbol{v} \in {\rm span}(\mathcal{V}_{\rm good})$ be given and define a perturbation of $\boldsymbol{v}$ by 
\begin{align}\label{eq:sv'}
\boldsymbol{v}' = \boldsymbol{v} \RRR + \EEE s\ell^{-1}|\boldsymbol{v}|(0,0,{e}_3,  {e}_{3}, {e}_{3}, {e}_{3},0,0) \in \Rz^{3 \times 8}
\end{align}
 for some universal $s>0$ to be specified below. (Note that the direction $\boldsymbol{v}' - \boldsymbol{v}$ increases the third components of the points $x_3,\ldots,x_6$ {\BBB of the basic cell}). By Property 1 and 2 of Lemma \ref{tildeE} and the fact that $|\boldsymbol{v}-\boldsymbol{v}'| \le 4s|\boldsymbol{v}|\ell^{-1}$ it clearly suffices to show
\begin{align}\label{eq: deriv=0}
D \tilde{E}(T(\boldsymbol{x}^\ell_{\rm kink}))  \big(D T(\boldsymbol{x}^\ell_{\rm kink}) \boldsymbol{v}'\big) \ge 0.
\end{align}
To this end, we will show that 
\begin{align}\label{eq: deriv=0.2}
\tilde{E}(T(\boldsymbol{x}^\ell_{\rm kink} + t\boldsymbol{v}')) \ge \tilde{E}(T(\boldsymbol{x}^\ell_{\rm kink}))
\end{align}
for all $t >0$ small. Then \eqref{eq: deriv=0} follows by taking the limit $t \to 0$.

Consider  $\boldsymbol{x} = \boldsymbol{x}_{\rm kink}^\ell + t \boldsymbol{v}'$ for $t >0$ small.  Possibly after applying a rigid motion we can assume that the second  and third components of \RRR $(x_1 + x_7)/2 $ and $(x_2 + x_8)/2$ are zero, \EEE the points $x_1,x_2,x_7,x_8$ lie in the plane $\Rz^2 \times \lbrace 0\rbrace$ and that the points $x_3,x_4,x_5,x_6$ lie in a plane parallel to $\Rz^2 \times \lbrace 0\rbrace$. (Recall that $\boldsymbol{v}$ induces an in-plane perturbation, \PPP i.e., \EEE the third component of each vector in $\boldsymbol{v}$ is zero.) We replace $\boldsymbol{x}$ by a symmetrized version as follows.

Define {\BBB $\boldsymbol{x}_{S_1} $ by \eqref{s1s2} } and note that $E_{\rm cell}(\boldsymbol{x}_{S_1} ) =  E_{\rm cell}(\boldsymbol{x})$.   Moreover, it is elementary to see that the third component of each vector in $\boldsymbol{w}_1 := \boldsymbol{x}_{S_1}  - \boldsymbol{x}$ is zero. Consequently, $\boldsymbol{w}_1$ is perpendicular to $\mathcal{V}_{\rm bad}$, $\mathcal{V}_{\rm trans}$, and the rotations $\boldsymbol{v}_2, \boldsymbol{v}_3$. Clearly, as the reflection $S_1$ leaves the points \RRR $(x_1 + x_7)/2 $ and $(x_2 + x_8)/2$ \EEE unchanged, we also have that $\boldsymbol{w}_1$ is not parallel to the rotation $\boldsymbol{v}_1$. Consequently, by Theorem \ref{th: cell convexity1} and a continuity argument with $t$ small enough, the mapping $t' \mapsto E_{\rm cell}(\boldsymbol{x} + t' \RRR \boldsymbol{w}_1 \EEE )$ is convex on $[0,1]$. This implies for  $\boldsymbol{x}' = \frac{1}{2}(\boldsymbol{x} + \boldsymbol{x}_{S_1} )$ (see \eqref{reflection2-a}) that $E_{\rm cell}(\boldsymbol{x}') \le \frac{1}{2}(E_{\rm cell}(\boldsymbol{x}) + E_{\rm cell}(\boldsymbol{x}_{S_1} )) = E_{\rm cell}(\boldsymbol{x})$.

Likewise, we consider {\BBB $\boldsymbol{x}'_{S_2}:  = \boldsymbol{x}_{\rm kink}^\ell  + S_2(\boldsymbol{x}'- \boldsymbol{x}_{\rm kink}^\ell)$ and note that $E_{\rm cell}(\boldsymbol{x}'_{S_2}) =  E_{\rm cell}(\boldsymbol{x}')$. Similarly as before, the vector $\boldsymbol{w}_2 := \boldsymbol{x}'_{S_2} - \boldsymbol{x}'$ is perpendicular to the vectors $\mathcal{V}_{\rm bad}$ and not parallel to $\mathcal{V}_{\rm degen}$. Using Theorem \ref{th: cell convexity1} we get $E_{\rm cell} (\mathcal{S}(\boldsymbol{x})) \le E_{\rm cell}(\boldsymbol{x}') \le E_{\rm cell}(\boldsymbol{x}) $  for  $\mathcal{S}(\boldsymbol{x}) = \frac{1}{2}(\boldsymbol{x}' +\boldsymbol{x}'_{S_2})$ (see \eqref{reflection2-b}).}

By this symmetrization procedure we get that the eight points $\mathcal{S}(\boldsymbol{x})$ are contained in two {\BBB kinked} planes (similarly as $\boldsymbol{x}_{\rm kink}^\ell$).  We denote the incidence angle of  the two planes by $\gamma \le \pi$ and note that $\gamma \le \gamma_\ell$ if the constant $s>0$ \RRR in \eqref{eq:sv'} \EEE is chosen sufficiently large.   The bond lengths satisfy $b_1 = b_2$, $b_3 = b_4 = b_5= b_6$ and $b_7 = b_{8}$. For the angles $\varphi_1 = \varphi_2$ and $\varphi_3 = \ldots = \varphi_{10}$ holds.

Recalling \eqref{betaz} and  \eqref{eq: cell-energy}    we find $\alpha$ in a small neighborhood of $\alpha_\ell^{\rm us}$ such that 
$$E_{\rm cell}(\mathcal{S}(\boldsymbol{x})) \ge - 3+ 4 v_3(\alpha) + 2v_3\big(2\arcsin(\sin\alpha \sin(\gamma/2)) \big).$$
Now taking $\gamma \le \gamma_\ell$ into account and recalling that  $\alpha_\ell^{\rm us}$ is {\BBB optimal angle from}  Proposition \ref{eq: old main result}, we find 
\begin{align*}
E_{\rm cell}(\boldsymbol{x}) & \ge E_{\rm cell}(\mathcal{S}(\boldsymbol{x})) \ge - 3 + 4 v_3(\alpha) + 2v_3\big(2\arcsin(\sin\alpha \sin(\gamma_\ell/2)) \big) \\& \ge  - 3 + 4 v_3(\alpha_\ell^{\rm us}) + 2v_3\big(2\arcsin(\sin\alpha_\ell^{\rm us} \sin(\gamma_\ell/2)) \big)  = E_{\rm cell}(\boldsymbol{x}_{\rm kink}^\ell),
\end{align*}
\RRR where the last step \PPP follows from \EEE \eqref{kink}. \EEE This shows \eqref{eq: deriv=0.2} and concludes the proof.
\end{proof}

The next lemma shows that a perturbation of the angles, which does not change the sum of the angles, essentially does not decrease the energy to first order. 

\begin{lemma}\label{lemma: conv}
There exist $\ell_0 \in \Nz$ and a constant $C>0$ depending only on $v_2$ and $v_3$ such that for $\ell \ge \ell_0$ and each \PPP $\boldsymbol{w} =(w_1,\ldots,w_{10}) \in \Rz^{10}$ \EEE with $\boldsymbol{w} \cdot \boldsymbol{a}_j = 0$ for $j=1,2,3$ we have
$$\sum_{i=1}^{10} \big( D \tilde{E}(T(\boldsymbol{x}^\ell_{\rm kink})) \big)_i   \PPP w_i\EEE  \ge -  C|\boldsymbol{w}|\ell^{-3}.$$
\end{lemma}

\begin{proof}
From Property 2 of Lemma \ref{lemma: T} we have that the image of the affine mapping $DT^a(\boldsymbol{x}^0)$ has dimension 7.  Moreover, we have $(DT^a(\boldsymbol{x}^0) \boldsymbol{v}) \cdot \boldsymbol{a}_j = 0$ for $j=1,2,3$ and all $\boldsymbol{v} \in \Rz^{3 \times 8}$. Indeed, \RRR write $\boldsymbol{v} =  \boldsymbol{v}_{\rm good} + \boldsymbol{v}_{\rm bad}$ with $\boldsymbol{v}_{\rm good} \in {\rm span}(\mathcal{V}_{\rm good})$ and $v_{\rm bad} \in {\rm span}(\mathcal{V}_{\rm degen} \cup \mathcal{V}_{\rm bad})$. \EEE Note that $DT^a(\boldsymbol{x}^0) \boldsymbol{v} = DT^a( \boldsymbol{x}^0)  \boldsymbol{v}_{\rm good}$ \RRR by Property 1 of Lemma \ref{lemma: T}. \EEE For each $t \in \Rz$ the eight points $ \boldsymbol{x}^0 + t \boldsymbol{v}_{\rm good}$ are contained in the plane $\Rz^2 \times \lbrace 0 \rbrace$. This implies $T^a( \boldsymbol{x}^0 + t  \boldsymbol{v}_{\rm good}) \cdot  \boldsymbol{a}_j \in \lbrace 2\pi,4\pi \rbrace$ for all $t \in \Rz$ and $j=1,2,3$, which gives $(DT^a( \boldsymbol{x}^0)  \boldsymbol{v}_{\rm good}) \cdot  \boldsymbol{a}_j = 0$ for $j=1,2,3$, as desired. 

The dimension of the image of $DT^a(\boldsymbol{x}^0)$ together with   the fact that $\boldsymbol{w} \cdot \boldsymbol{a}_j = 0$ for $j=1,2,3$ show that there exists a vector $\boldsymbol{v}' \in {\rm span}(\mathcal{V}_{\rm good})$ such that $DT^a(\boldsymbol{x}^0) \boldsymbol{v}' = \boldsymbol{w}$. \PPP By applying Lemma \EEE \ref{lemma: in plane} we get 
$$
D \tilde{E}(T(\boldsymbol{x}^\ell_{\rm kink}))  \big(D T(\boldsymbol{x}^\ell_{\rm kink}) \boldsymbol{v}'\big)  \ge -  C'|\boldsymbol{v}'| \ell^{-3},
$$
where $C'$ is the constant from Lemma \ref{lemma: in plane}. By a continuity argument and \eqref{eq: distance} we get $|D T(\boldsymbol{x}^\ell_{\rm kink}) - D T(\boldsymbol{x}^0)| \le c\ell^{-1}$. This together with Property 2 of Lemma \ref{tildeE} shows  
$$
D \tilde{E}(T(\boldsymbol{x}^\ell_{\rm kink}))  \big(D T(\boldsymbol{x}^0) \boldsymbol{v}'\big)   \ge -  C|\boldsymbol{v}'| \ell^{-3}
$$
for $C=C(C',c_{E,2},c)$. The fact that $DT^a(\boldsymbol{x}^0) \boldsymbol{v}' = \boldsymbol{w}$, $|\boldsymbol{v}'| \le c|\boldsymbol{w}|$ for a   constant $c>0$ (depending on $DT^a(\boldsymbol{x}^0)$) and Property 1 of Lemma \ref{tildeE} conclude the proof.
\end{proof}

We now  improve Theorem \ref{th: cell convexity1} and prove convexity  of $E_{\rm cell}$ at the kink configuration $\boldsymbol{x}^\ell_{\rm kink}$.

\begin{theorem}[Convexity of $E_{\rm cell}$]\label{th: cell convexity3}
Let $0 < r <1$. Then there exist $\ell_0 \in \Nz$ and a constant $c>0$ depending only on $v_2$, $v_3$, and $r$ such that for $\ell \ge \ell_0$ and each $\boldsymbol{v} \in \Rz^{3 \times 8}$ with 
$$|\boldsymbol{v} \cdot \boldsymbol{w}| \le r|\boldsymbol{w}||\boldsymbol{v}|  \ \ \ \text{ for all } \  \ \ \boldsymbol{w}  \in {\rm span}(\mathcal{V}_{\rm degen})$$ one has
$$\boldsymbol{v}^TD^2E_{\rm cell}(\boldsymbol{x}^\ell_{\rm kink})\boldsymbol{v} \ge c|\boldsymbol{v}|^2\ell^{-2}.$$
\end{theorem}

\begin{proof} As in the proof of Theorem \ref{th: cell convexity1} we consider the mapping $f_{\boldsymbol{v}}$ as defined before \eqref{eq: conv3-b}. The goal is to show $f''_{\boldsymbol{v}}(0) \ge c |\boldsymbol{v}|^2 \ell^{-2}$.   We write $\boldsymbol{v} = \boldsymbol{v}_{\rm degen}+ \boldsymbol{v}_{\rm bad} + \boldsymbol{v}_{\rm good} $ with three orthogonal vectors, where  $\boldsymbol{v}_{\rm degen}+ \boldsymbol{v}_{\rm bad} \in {\rm span}(\mathcal{V}_{\rm degen} \cup \mathcal{V}_{\rm bad})$, $\boldsymbol{v}_{\rm degen} \in {\rm span}(\mathcal{V}_{\rm degen})$, $\boldsymbol{v}_{\rm bad} \in {\rm span}(\mathcal{V}_{\rm degen})^\bot$, and $\boldsymbol{v}_{\rm good} \in {\rm span}(\mathcal{V}_{\rm degen} \cup \mathcal{V}_{\rm bad})^\bot$. By assumption we obtain after a short calculation
\begin{align}\label{givenumber}
|\boldsymbol{v}_{\rm good}|^2 + |\boldsymbol{v}_{\rm bad}|^2 \ge (1-r^2)|\boldsymbol{v}|^2.
\end{align}
Set $c_* := \max\lbrace 2 c_{2}/c_1, (8c_3/(c_{E,1}c^2_1))^{1/2}  \rbrace$ with $c_1,c_2$ from \eqref{eq: conv1}, $c_3$ from \eqref{eq: conv2}, and $c_{E,1}$ from  Lemma \ref{tildeE}.  First,  we suppose $|\boldsymbol{v}_{\rm good}| \ge c_*|\boldsymbol{v}|\ell^{-1}$. We use \eqref{eq: conv1} and   $\boldsymbol{v}_{\rm good} \in {\rm span}(\mathcal{V}_{\rm degen} \cup \mathcal{V}_{\rm bad})^\bot$ to find
$$ |D T(\boldsymbol{x}^\ell_{\rm kink})\boldsymbol{v}| \ge c_1|\boldsymbol{v}_{\rm good}| - c_{2}|\boldsymbol{v}| \ell^{-1} \ge \frac{c_1}{2}|\boldsymbol{v}_{\rm good}|.$$ 
 Then by Property 3 of Lemma \ref{tildeE}, \eqref{eq: conv3-b}, and  \eqref{eq: conv2} we get 
 \begin{align*}
 f''_{\boldsymbol{v}}(0)  &= \boldsymbol{v}^TD^2E_{\rm cell}(\boldsymbol{x}_{\rm kink})\boldsymbol{v}  \ge \big(D T(\boldsymbol{x}^\ell_{\rm kink}) \boldsymbol{v}\big)^T D^2 \tilde{E}(T(\boldsymbol{x}^\ell_{\rm kink}))  \big(D T(\boldsymbol{x}^\ell_{\rm kink}) \boldsymbol{v}\big) - c_3 |\boldsymbol{v}|^2 \ell^{-2}  \\
 & \ge c_{E,1}|D T(\boldsymbol{x}^\ell_{\rm kink})\boldsymbol{v}|^2 - c_3 |\boldsymbol{v}|^2 \ell^{-2} \ge  \frac{c_{E,1} c^2_1}{4}|\boldsymbol{v}_{\rm good}|^2- c_3 |\boldsymbol{v}|^2 \ell^{-2} \ge  \frac{c_{E,1} c^2_1c_*^2}{8\ell^2} |\boldsymbol{v}|^2.
 \end{align*}
 
 Now suppose $|\boldsymbol{v}_{\rm good}| < c_*|\boldsymbol{v}|\ell^{-1}$. Since the first term of $f_{\boldsymbol{v}}''(0)$ given in \eqref{eq: conv3-b} is nonnegative, it suffices to consider the second term of $f_{\boldsymbol{v}}''(0)$. First, using Property 1 of Lemma \ref{tildeE} we have
\begin{align}\label{eq: conv6}
\sum_{i=11}^{18} \big( \PPP D \tilde{E}(T(\boldsymbol{x}^\ell_{\rm kink}) \EEE \big)_i  \, \big(\boldsymbol{v}^TD^2 T(\boldsymbol{x}^\ell_{\rm kink}) \boldsymbol{v}\big)_i =0.
\end{align}
Define for brevity  $\boldsymbol{w} =(\boldsymbol{v}_{\rm degen} + \boldsymbol{v}_{\rm bad})^T D^2 T^a(\boldsymbol{x}^\ell_{\rm kink}) (\boldsymbol{v}_{\rm degen} + \boldsymbol{v}_{\rm bad}) \in \Rz^{10}$ and note that $|\boldsymbol{v}_{\rm good}| < c_*\ell^{-1}|\boldsymbol{v}|$   implies
\begin{align}\label{new estimate}
\Big|(\boldsymbol{v}^T D^2 T^a(\boldsymbol{x}^\ell_{\rm kink}) \boldsymbol{v})_i -  \boldsymbol{w}_i  \Big|    \le c_4|\boldsymbol{v}|^2\ell^{-1}, \ \ \ \ i=1,\ldots,10,
\end{align}
for $c_4$  depending on $c_*$.  By Properties 3 and 4 in Lemma \ref{lemma: T}, \eqref{eq: distance},  and a continuity argument  we obtain constants $0 < c_5 < c_6$ (depending on $c_{\rm kink}$) such that for $\ell$ sufficiently large
\begin{align*}
\boldsymbol{w} \cdot \boldsymbol{a}_j \le  c_6|\boldsymbol{v}|^2 \ell^{-1},   \ \ \ j=1,2,3, \ \ \ \ \ \  
  \sum_{j=1}^{3} \boldsymbol{w} \cdot \boldsymbol{a}_j \le -c_5|\boldsymbol{v}_{\rm bad}|^2  + c_6|\boldsymbol{v}|^2 \ell^{-1}.
\end{align*}
Consequently, we can find a decomposition $\boldsymbol{w} = \boldsymbol{w}' + \boldsymbol{w}''$ with the property
\begin{align*}
&\boldsymbol{w}' \cdot \boldsymbol{a}_j = 0, \ \ \ j=1,2,3,  \ \ \ \  |\boldsymbol{w}'| \le c_7|\boldsymbol{v}|^2,\\
&  \sum_{i=1}^{10} w''_i \le -c_5|\boldsymbol{v}_{\rm bad}|^2  + c_6|\boldsymbol{v}|^2 \ell^{-1},  \  \ \ \ w''_i \le c_6|\boldsymbol{v}|^2 \ell^{-1}, \ \  \ i=1,\ldots,10
\end{align*}
for a universal constant $c_7>0$. \BBB (Choose, e.g., $w_3' = w_3 - \boldsymbol{w} \cdot \boldsymbol{a}_1$, $w_7' = w_7 - \boldsymbol{w} \cdot \boldsymbol{a}_2$, $w_9' = w_9 - \boldsymbol{w} \cdot \boldsymbol{a}_3$, and $w_i' = w_i$ else.) \EEE Let $I  = \lbrace i=1,\ldots,10| \ \boldsymbol{w}_i'' \le 0 \rbrace$ and note $\sum_{i \in I}  \boldsymbol{w}_i'' \le  \sum_{i=1}^{10} \boldsymbol{w}''_i $. Then using Property 2 of Lemma \ref{tildeE} and Lemma \ref{lemma: conv}  we derive
\begin{align*}
\sum_{i=1}^{10} & \big( \PPP D \tilde{E}(T(\boldsymbol{x}^\ell_{\rm kink}) \EEE \big)_i  \boldsymbol{w}_i   \\&= \sum_{i=1}^{10} \big( \PPP D \tilde{E}(T(\boldsymbol{x}^\ell_{\rm kink}) \EEE \big)_i   \boldsymbol{w}_i' +\sum_{i \in I} \big( \PPP D \tilde{E}(T(\boldsymbol{x}^\ell_{\rm kink}) \EEE \big)_i  \boldsymbol{w}_i'' + \sum_{i \notin I} \big( \PPP D \tilde{E}(T(\boldsymbol{x}^\ell_{\rm kink}) \EEE \big)_i  \boldsymbol{w}_i''\\
&\ge - C|\boldsymbol{w}'| \ell^{-3} + c_{E,1}\ell^{-2} \sum_{i \in I} - \boldsymbol{w}_i'' - \PPP 10 c_{E,2} \EEE c_6|\boldsymbol{v}|^2 \ell^{-3} 
\\
&\ge - Cc_7|\boldsymbol{v}|^2 \ell^{-3} + c_{E,1}\ell^{-2}\big(c_5|\boldsymbol{v}_{\rm bad}|^2  - c_6|\boldsymbol{v}|^2 \ell^{-1}\big) - \PPP 10 c_{E,2} \EEE c_6|\boldsymbol{v}|^2 \ell^{-3} ,
\end{align*}
where $C$ is the constant from Lemma \ref{lemma: conv}. Moreover, again using  Lemma \ref{tildeE} and \eqref{new estimate} we get
$$\sum_{i=1}^{10} {\BBB \left|   \big( \PPP D \tilde{E}(T(\boldsymbol{x}^\ell_{\rm kink}) \EEE \big)_i   \Big(\boldsymbol{w}_i - \big(\boldsymbol{v}^T D^2 T(\boldsymbol{x}^\ell_{\rm kink}) \boldsymbol{v}\big)_i \Big)  \right|}\le \PPP 10 c_{E,2} \EEE c_4 |\boldsymbol{v}|^2\ell^{-3}.$$ 
We then use \eqref{eq: conv3-b}, \eqref{eq: conv6}, and the previous two estimates to find
\begin{align*}
f''_{\boldsymbol{v}}(0) &= \boldsymbol{v}^TD^2E_{\rm cell}(\boldsymbol{x}_{\rm kink})\boldsymbol{v} \ge D \tilde{E}(T(\boldsymbol{x}^\ell_{\rm kink})) \big(\boldsymbol{v}^TD^2 T(\boldsymbol{x}^\ell_{\rm kink}) \boldsymbol{v}\big)  \\&
\ge c_{E,1}c_5|\boldsymbol{v}_{\rm bad}|^2 \ell^{-2} - c' |\boldsymbol{v}|^2 \ell^{-3},
\end{align*}
\PPP for \EEE $c' =c'(C,c_{E,1},c_{E,2},c_4,c_5,c_6,c_7)$ large enough. Since $|\boldsymbol{v}_{\rm good}| < c_*|\boldsymbol{v}|\ell^{-1}$, we get $|\boldsymbol{v}_{\rm bad}|^2 \ge \frac{1}{2}(1-r^2)|\boldsymbol{v}|^2$ for $\ell_0$ large enough by 
 \eqref{givenumber}. Then $f''_{\boldsymbol{v}}(0) \ge c\ell^{-2}|\boldsymbol{v}|^2$  follows when we choose $\ell_0 \in \Nz$ sufficiently large (depending also on $r$).
\end{proof}  

\subsection{Proof of  Theorem \ref{th: Ered}}
As a last preparation for the proof of Theorem \ref{th: Ered}, we need to investigate how the angles between planes behave under reflection of a configuration (see \eqref{reflexion}-\eqref{reflection2}). Let  a center $z_{i,j,k}$ be given and, as before, denote by  $\boldsymbol{x} \in \Rz^{3 \times 8}$ the atoms of the corresponding cell. We introduce the angles between the planes as in Section \ref{sec: main proof}.   By $\theta_l(\boldsymbol{x})$ we denote the angle between the planes $\{x_1 x_3 x_4\}$ and $\{x_1 x_6 x_5\}$. By $\theta_r(\boldsymbol{x})$ we denote the angle between the planes $\{x_3 x_4 x_2\}$ and $\{x_2 x_5 x_6\}$. Moreover, we let   $\theta^{\rm dual}_l(\boldsymbol{x}) = \theta(x_1)$ and $\theta^{\rm dual}_r(\boldsymbol{x}) = \theta(x_2)$ with $\theta(x_i)$, $i=1,2$, as defined in \eqref{eq: thetaangle}.  Recall also the definition of  $\Delta(z_{i,j,k})$ in \eqref{delta}.

\begin{lemma}[Symmetry defect controls angle defect]\label{lemma: angle invariance}
There \PPP exist \EEE a universal constant $C>0$ and $\ell_0 \in \Nz$, and for each $\ell \ge \ell_0$  there  \PPP exists \EEE $\eta_\ell >0$ such that for all $\tilde{\mathcal{F}} \in  \mathscr{P}_{\eta_\ell}(\mu)$,  $\mu \in (2.6,3.1)$, \PPP and  all \EEE centers $z_{i,j,k}$ we have
\begin{align*} 
&\theta_l(\mathcal{S}(\boldsymbol{x})) +  \theta_r(\mathcal{S}(\boldsymbol{x})) \le \theta_l(\boldsymbol{x}) +  \theta_r(\boldsymbol{x})+ C\Delta(z_{i,j,k}),\\
 &\theta^{\rm dual}_l(\mathcal{S}(\boldsymbol{x}))  + \theta^{\rm dual}_r(\mathcal{S}(\boldsymbol{x})) \le  \theta^{\rm dual}_l(\boldsymbol{x})  + \theta^{\rm dual}_r(\boldsymbol{x}) + C\Delta(z_{i,j,k}),
\end{align*}
where $\boldsymbol{x} \in \Rz^{3 \times 8}$  \PPP denotes \EEE the position of the atoms in the cell with center $z_{i,j,k}$ and $\mathcal{S}(\boldsymbol{x})$ as in \eqref{reflection2-b}. 
\end{lemma}

We postpone the proof of this lemma to the end of the section and now continue with the proof of   Theorem \ref{th: Ered}.

\begin{proof}[Proof of Theorem \ref{th: Ered}]
Let \PPP $\tilde{\mathcal{F}} \in \mathscr{P}_{\eta_\ell}(\mu)$ be a given configuration, where $\eta_\ell$ is specified below, \EEE and let $\boldsymbol{x} \in \Rz^{3 \times 8}$ be the points of one cell as introduced in Section \ref{sec: main proof}. {\BBB As usual}, possibly after a rigid motion we can assume that the second and third components of \RRR $(x_1 + x_7)/2$, $(x_2+x_8)/2$ are zero \EEE and the points $x_4$, $x_5$ lie in a plane parallel to $\Rz^2 \times \lbrace 0 \rbrace$. We now perform a symmetrization argument as in the proof of Lemma \ref{lemma: in plane}.

 We define $\boldsymbol{x}_{S_1}$ by \PPP \eqref{s1s2}. Clearly \EEE the  vector $\boldsymbol{w}_1 := \boldsymbol{x}_{S_1} - \boldsymbol{x}$ is  perpendicular to $\mathcal{V}_{\rm trans}$. Moreover, we have $|\boldsymbol{w}_1\cdot \boldsymbol{v}_i| \le r |\boldsymbol{w}_1||\boldsymbol{v}_i|$ for $i=1,2,3$ for \PPP a universal constant  $r\in(0,1)$. In particular, $r$ is independent \EEE of the perturbation $\boldsymbol{x}$. Indeed, for $\boldsymbol{v}_1$ and $\boldsymbol{v}_2$ this follows from the fact that the points \RRR $(x_1 + x_7)/2 $ and $(x_2 + x_8)/2$ \EEE are left unchanged. For $\boldsymbol{v}_3$ it follows from the assumption that the points $x_4$, $x_5$ lie in a plane parallel to $\Rz^2 \times \lbrace 0 \rbrace$.

Consequently, by Theorem \ref{th: cell convexity3}, a continuity argument, and the definition of the the perturbations $\mathscr{P}_{\eta_\ell}(\mu)$, the mapping $t \mapsto E_{\rm cell}(\boldsymbol{x} + t \boldsymbol{w}_1)$ is strictly convex on $[0,1]$ if $\eta_\ell$  is chosen small enough (independent of $\boldsymbol{x}$). This implies for  $\boldsymbol{x}' = \frac{1}{2}(\boldsymbol{x} + \boldsymbol{x}_{S_1})$ (see \eqref{reflection2-a}) that $E_{\rm cell}(\boldsymbol{x}') + c\ell^{-2}|\boldsymbol{w}_1|^2  \le \frac{1}{2}(E_{\rm cell}(\boldsymbol{x}) + E_{\rm cell}(\boldsymbol{x}_{S_1})) = E_{\rm cell}(\boldsymbol{x})$, where $c$ only depends on the constant from   Theorem \ref{th: cell convexity3}. 

Likewise, we consider {\BBB $\boldsymbol{x}'_{S_2} := \boldsymbol{x}_{\rm kink}^\ell  + S_2(\boldsymbol{x}'- \boldsymbol{x}_{\rm kink}^\ell)$ } \PPP and, similarly as before, \EEE the vector {\BBB  $\boldsymbol{w}_2 := \boldsymbol{x}'_{S_2} - \boldsymbol{x}'$ } is perpendicular to  $\mathcal{V}_{\rm trans}$ and satisfies $|\boldsymbol{w}_2\cdot \boldsymbol{v}_i| \le r |\boldsymbol{w}_2||\boldsymbol{v}_i|$ for $i=1,2,3$ for \PPP  a universal constant  $r\in(0,1)$. \EEE Indeed, for $\boldsymbol{v}_1$ and $\boldsymbol{v}_2$ this follows as before and for $\boldsymbol{v}_3$ it suffices to note that also for the configuration $\boldsymbol{x}' = (x_1',\ldots,x_8')$ the points $x'_4$, $x'_5$ lie in a plane parallel to $\Rz^2 \times \lbrace 0 \rbrace$. \RRR Using again   Theorem \ref{th: cell convexity3} we get $E_{\rm cell} (\mathcal{S}(\boldsymbol{x})) + c\ell^{-2}|\boldsymbol{w}_2|^2\le  E_{\rm cell}(\boldsymbol{x}') $   with $\mathcal{S}(\boldsymbol{x})$ from \eqref{reflection2-b}. Possibly passing to a smaller \PPP constant $c>0$ (not relabeled) \EEE and using \eqref{delta}, we observe
$$E_{\rm cell}(\mathcal{S}(\mathbf{x})) + c\ell^{-2}\Delta(z_{i,j,k}) \le   E_{\rm cell}(\boldsymbol{x}).$$
\EEE By this symmetrization procedure we get that the eight points $\mathcal{S}(\mathbf{x})$ satisfy the symmetry conditions stated in \eqref{sym-assumption}. {\BBB  In particular, $\widetilde\mu$ from \eqref{sym-assumption} is here equal to $|z^{\rm dual}_{i,j,k} - z^{\rm dual}_{i,j-1,k}|$, the latter quantity being unchanged after symmetrization \RRR since the second and third component of $z^{\rm dual}_{i,j,k}, z^{\rm dual}_{i,j-1,k}$ are assumed to be zero. \EEE} {\RRR Choose $M^\ell$ and $\eta_\ell$   small enough   such that  $|\lambda_1 - 1|  + |\lambda_3 - 1| \le \ell^{-4}$, and  $|\gamma_1 - \gamma_2| \le \ell^{-2}$ with $\lambda_1, \lambda_3,\gamma_1,\gamma_2$ from \eqref{sym-assumption}. This choice of $M^\ell$ is possible thanks to Property 2 in  Proposition \ref{th: main2}}. Consequently, by Lemma \ref{lemma: sym-energy} we obtain 
$$
E_{\rm cell}(\boldsymbol{x}) = E_{\rm cell}(z_{i,j,k}) \ge  E_{{\BBB \widetilde\mu},\gamma_1,\gamma_2}^{{\rm sym}}(\lambda_2,\alpha_1,\alpha_2)  + c\ell^{-2}\Delta(z_{i,j,k})  - c_0 \ell^{-4} (\gamma_1 - \gamma_2)^2.
$$
Using Property 2 of Proposition \ref{th: mainenergy} and \eqref{red}  we get for $\ell_0$ sufficiently large  
\begin{align}\label{sym-red} 
E_{\rm cell}(z_{i,j,k}) \ge E_{\rm red}({\BBB \widetilde\mu}, \bar{\gamma},\bar{\gamma} )  + c\ell^{-2}\Delta(z_{i,j,k}),
\end{align}
where $\bar{\gamma} = (\gamma_1 + \gamma_2)/2$. By Lemma  \ref{lemma: angle invariance} we obtain $\bar{\gamma} \le \bar{\theta}(z_{i,j,k}) + C\Delta(z_{i,j,k})$, where  $\bar{\theta}(z_{i,j,k}) = \big(\theta_l(z_{i,j,k}) + \theta_r(z_{i,j,k}) + \theta_l(z^{\rm dual}_{i,j,k}) + \theta_r(z^{\rm dual}_{i,j-1,k}) \big)/4$. Thus, by the monotonicity of the reduced energy (see Property 3 of Proposition \ref{th: mainenergy}) and a Taylor expansion for the mapping $\gamma \mapsto E_{\rm red}({\BBB \widetilde\mu},\gamma,\gamma)$ we get
\begin{align} 
E_{\rm red}({\BBB\widetilde\mu},\bar{\gamma},\bar{\gamma})&  \ge E_{\rm red}\Big({\BBB\widetilde\mu},\bar{\theta}(z_{i,j,k}) , \bar{\theta}(z_{i,j,k})  \Big) -C\ell^{-3} \Delta(z_{i,j,k})  + {\rm O}\big((\Delta(z_{i,j,k}) )^2 \big) \notag\\
& \ge E_{\rm red}\Big({\BBB \widetilde\mu},\bar{\theta}(z_{i,j,k}) , \bar{\theta}(z_{i,j,k})  \Big) -\RRR 2 \EEE C\ell^{-3} \Delta(z_{i,j,k}) \label{sym-red2} 
\end{align} 
for $C>0$ large enough \RRR depending on $v_3$, \EEE where the last step follows for $\eta_\ell$ sufficiently small.  The assertion of the theorem now follows for $\ell_0$ sufficiently large and $\ell \ge \ell_0$  from  \eqref{sym-red}, \eqref{sym-red2},   {\BBB { and the fact that $\widetilde\mu = |z^{\rm dual}_{i,j,k} - z^{\rm dual}_{i,j-1,k}|$.}}
\end{proof}

Finally, we give the proof of Lemma \ref{lemma: angle invariance}.

\begin{proof}[Proof of Lemma \ref{lemma: angle invariance}]
The proof is mainly based on a careful Taylor expansion for the angles under the symmetrization of the atomic positions in the cell, which is induced by the reflections \eqref{reflexion}. In particular, the argumentation for the angles $\theta_l, \theta_r$ and the dual angles $\theta_l^{\rm dual}$, $\theta_r^{\rm dual}$, respectively, is very \PPP similar. \EEE Therefore, we concentrate on the first inequality in the following. 

Let the configuration $\boldsymbol{x}$ be given  for a center $z_{i,j,k}$. Let $n^l_{1}(\boldsymbol{x})$ and $n^l_2(\boldsymbol{x})$ be \RRR unit \EEE normal vectors of the planes  $\{x_1 x_3 x_4\}$ and $\{x_1 x_6 x_5\}$. Likewise, let $n^r_{1}(\boldsymbol{x})$ and $n^r_2(\boldsymbol{x})$ be normal vectors of the planes  $\{ x_2 x_4 x_3 \}$ and $\{ x_2 x_5 x_6\}$. Let $n_l(\boldsymbol{x})$ and $n_r(\boldsymbol{x})$ be  unit vectors perpendicular to $n^l_{1}(\boldsymbol{x}), n^l_2(\boldsymbol{x})$ and  $n^r_{1}(\boldsymbol{x}), n^r_2(\boldsymbol{x})$, respectively.

Let $s_1^l(\boldsymbol{x})$ be a unit vector perpendicular to $n_l(\boldsymbol{x})$, $n_1^l(\boldsymbol{x})$ and let $s_2^l(\boldsymbol{x})$ be a unit vector perpendicular to $n_l(\boldsymbol{x})$, $n_2^l(\boldsymbol{x})$ such that $s_1^l(\boldsymbol{x}) \cdot s_2^l(\boldsymbol{x})$ is near $-1$. \PPP We define $s_1^r(\boldsymbol{x})$, $s_2^r(\boldsymbol{x})$ in a similar fashion. \EEE Note that these objects can be \PPP chosen to depend  smoothly \EEE with respect to $\boldsymbol{x}$ and that the angle in \eqref{eq: thetaangle} can be expressed as 
$$\theta_k(\boldsymbol{x}) = \arccos\big(s_1^k(\boldsymbol{x}) \cdot s_2^k(\boldsymbol{x})\big) \ \ \ \text{for } \ \ k=l,r. $$
We also introduce the mapping
 \begin{align}\label{concave0}
 g(\boldsymbol{x}) = \arccos\big(s_1^l(\boldsymbol{x}) \cdot s_2^l(\boldsymbol{x})\big) + \arccos \big( s_1^r(\boldsymbol{x}) \cdot s_2^r(\boldsymbol{x}) \big).
 \end{align}
\emph{Step I.} Recall from \PPP the definition \EEE in \eqref{reflection2}, \eqref{delta} that there are two vectors $\boldsymbol{w}_1,\boldsymbol{w}_2 \in \Rz^{3 \times 8}$ such that  the symmetrized configurations can be expressed as $\boldsymbol{x}' = \boldsymbol{x} + \boldsymbol{w}_1$ and $\mathcal{S}(\boldsymbol{x}) = \boldsymbol{x}' + \boldsymbol{w}_2$ with  
\begin{align}\label{bad/good}
|\boldsymbol{w}_1|^2 + |\boldsymbol{w}_2|^2 \RRR = \EEE \Delta(z_{i,j,k})
\end{align}
for a universal constant $C>0$. The goal will be to investigate the Hessian of $g$ and to show
 \begin{align}\label{concave3}
 \boldsymbol{w}_1^TD^2g(\boldsymbol{x}')\boldsymbol{w}_1 +   \boldsymbol{w}_2^TD^2g(\mathcal{S}(\boldsymbol{x}))\boldsymbol{w}_2 \ge -C(|\boldsymbol{w}_1|^2 +  |\boldsymbol{w}_2|^2) 
 \end{align} 
for $C>0$ universal. We defer the proof of \eqref{concave3} and first show that the assertion follows \PPP from it. \EEE We consider the mappings 
\begin{align}\label{eq: f1,f2}
f_1(t) = g(\boldsymbol{x}' + t\boldsymbol{w}_1 ), \ \ \ \  f_2(t) =  g(\mathcal{S}(\boldsymbol{x}) + t\boldsymbol{w}_2 ) \ \ \ \ \text{for} \ \ t\in [-1,1] 
\end{align}
and observe that $f_1(-1) = g(\boldsymbol{x})$, $f_2(-1) = g(\boldsymbol{x}')$, {\BBB $f_1(1)=g(\boldsymbol{x_{S_1}})$, $f_2(1)=g(\boldsymbol{x}'_{S_2})$, where $\boldsymbol{x}_{S_1}=\boldsymbol{x}_{\rm kink}^\ell+S_1(\boldsymbol{x}-\boldsymbol{x}_{\rm kink}^\ell)$ and $\boldsymbol{x}'_{S_2}=\boldsymbol{x}_{\rm kink}^\ell+S_2(\boldsymbol{x}'-\boldsymbol{x}_{\rm kink}^\ell)$, see \eqref{reflexion}-\eqref{s1s2}.}  Moreover, due to the fact that the symmetrized configurations are obtained by applying the reflections $S_1,S_2$, see \eqref{reflexion}, we get that $f_1,f_2$ are smooth, even functions, in particular, $f'_1(0) = f'_2(0) = 0$.  Thus, by a Taylor expansion we find $\xi \in (-1,0)$ such that
$$
g(\boldsymbol{x})  = f_1(-1)  = f_1(0) -  f_1'(0) + \frac{1}{2} f_1''(0) - \frac{1}{6}f_1'''(\xi)    \ge  g(\boldsymbol{x}' )  + \frac{1}{2}\boldsymbol{w}_1^TD^2g(\boldsymbol{x}')\boldsymbol{w}_1 - C|\boldsymbol{w}_1|^3, 
$$
where $C>0$ is a universal constant. Indeed, the constant is independent of $\boldsymbol{x}$ as all admissible $\boldsymbol{x}$  \RRR lie in a compact neighborhood of $\boldsymbol{x}_{\rm kink}^\ell$ where $g$ is smooth. \EEE  Applying Taylor once more, we get
$$ g(\boldsymbol{x}) \ge  g(\mathcal{S}(\boldsymbol{x}) )  + \frac{1}{2}\boldsymbol{w}_1^TD^2g(\boldsymbol{x}')\boldsymbol{w}_1 + \frac{1}{2}\boldsymbol{w}_2^TD^2g(\mathcal{S}(\boldsymbol{x}))\boldsymbol{w}_2  - C|\boldsymbol{w}_1|^3  - C|\boldsymbol{w}_2|^3.$$
Then we conclude  for  $\eta_\ell$ sufficiently small (and thus $|\boldsymbol{w}_1|, |\boldsymbol{w}_2|$ small) by \eqref{bad/good}-\eqref{concave3}
 $$  g(\boldsymbol{x})    \ge g(\mathcal{S}(\boldsymbol{x}) ) - C (|\boldsymbol{w}_1|^2 +  |\boldsymbol{w}_2|^2)  \RRR = \EEE g(\mathcal{S}(\boldsymbol{x}) ) - C\Delta(z_{i,j,k}).$$
Recalling \eqref{concave0} we obtain the assertion of the lemma.

\emph{Step II.} It remains to confirm \eqref{concave3}. We first concern ourselves with the Hessian of the mapping $f_1$ as defined in \eqref{eq: f1,f2}. For $t \in [-1,1]$ we let $u^k_j(t) = s^k_j(\boldsymbol{x}' + t\boldsymbol{w}_1)$ for $j=1,2$ and $k=l,r$. By a Taylor expansion we obtain
\begin{align}\label{expansion}
 u^k_j(t) = s^k_j(\boldsymbol{x}') + \big( v^{1,k}_j + w^{1,k}_j \big) t + \big( v^{2,k}_j + w^{2,k}_j\big) t^2     + {\rm O}(|\boldsymbol{w}_1|^3t^3)  \ \ \ \ \text{ with $ |u^k_j(t)| = 1$},
 \end{align}
where $v^{1,k}_j,v^{2,k}_j$ are perpendicular to $n_k(\boldsymbol{x}')$ and $w^{1,k}_j,w^{2,k}_j$ are parallel to $n_k(\boldsymbol{x}')$ such that $\sum_{j=1,2} \sum_{k=l,r}(|v^{1,k}_j| +  |w^{1,k}_j|) \le C|\boldsymbol{w}_1|$ and $\sum_{j=1,2} \sum_{k=l,r}(|v^{2,k}_j| +  |w^{2,k}_j|) \le C|\boldsymbol{w}_1|^2$. (The constant $C$ is again universal as all admissible $\boldsymbol{x}$ lie in a compact set and the mappings $s^k_j$ are smooth.) {\BBB For $j=1,2$ and $k=l,r$, the two vectors $w_j^{1,k}$ and $w_{j}^{2,k}$ are  orthogonal to $s_j^k(\boldsymbol{x}')$, and taking the first and the second derivative of the constraint $|s^k_j(\boldsymbol{x}'+t\boldsymbol{w}_1)|^2 = |u^{k}_j(t)|^2 = 1$ with respect to $t$ }  yields by an elementary computation
\begin{align}\label{u2}
(a)  \ \ v^{1,k}_j \cdot s^k_j(\boldsymbol{x}') = 0, \ \ \ \ \ \ (b) \ \   |v^{1,k}_j|^2 + |w^{1,k}_j|^2 + 2s^{k}_j(\boldsymbol{x}') \cdot v^{2,k}_j = 0.
\end{align}
Then we compute by \eqref{eq: f1,f2}
\begin{align*}
f_1(t)  & = \sum_{k=l,r} \arccos\Big(s^k_1(\boldsymbol{x}')  \cdot s^k_2(\boldsymbol{x}')  + \big( v_1^{1,k} \cdot s^k_2(\boldsymbol{x}') + v_2^{1,k} \cdot s^k_1(\boldsymbol{x}') \big)t \\& \ \ \ \ \ \ \ \ \  + \big( v_1^{2,k} \cdot s^k_2(\boldsymbol{x}') + v_2^{2,k} \cdot s^k_1(\boldsymbol{x}')+ v_1^{1,k} \cdot v_2^{1,k} + w_1^{1,k} \cdot w_2^{1,k}  \big) t^2 + {\rm O}(|\boldsymbol{w}_1|^3t^3) \Big).
\end{align*}
 \RRR A Taylor expansion and the fact that $f_1$ is even yield $ f_1(t) -f_1(0) = f''_1(0)t^2/2 +  {\rm
  O}(|\boldsymbol{w}_1|^3t^3)$. \PPP More precisely,  \EEE 
  \EEE  we get recalling   $s^k_1(\boldsymbol{x}')  \cdot s^k_2(\boldsymbol{x}') = \cos(\theta_k(\boldsymbol{x}'))$  for $k=l,r$
\begin{align} f_1(t) -f_1(0) & = \sum_{k=l,r} \arccos'(\cos(\theta_k(\boldsymbol{x}')))
 \big( v_1^{2,k} \cdot s^k_2(\boldsymbol{x}') + v_2^{2,k} \cdot s^k_1(\boldsymbol{x}')+ v_1^{1,k} \cdot v_2^{1,k} + w_1^{1,k} \cdot w_2^{1,k}  \big) t^2 \notag \\& 
\ \ \ +  \sum_{k=l,r}\frac{1}{2}\arccos''(\cos(\theta_k(\boldsymbol{x}')))  \big(
v_1^{1,k} \cdot s^k_2(\boldsymbol{x}') + v_2^{1,k} \cdot
s^k_1(\boldsymbol{x}') \big)^2t^2      + {\rm
  O}(|\boldsymbol{w}_1|^3t^3). \label{taylor1} 
\end{align}
We get  $|v_1^{1,k}  \cdot s^k_2(\boldsymbol{x}')| =|v_1^{1,k}|\sin(\theta_k(\boldsymbol{x}'))  $ by  \eqref{u2}(a). This together with \eqref{u2}(b)  and $|v_1^{2,k}| \le C|\boldsymbol{w}_1|^2$  yields for $k= l,r$
\begin{align*}
v_1^{2,k} \cdot s^k_2(\boldsymbol{x}') &= \Big( (v_1^{2,k}\cdot s^k_1(\boldsymbol{x}'))  s^k_1(\boldsymbol{x}') +  |v_1^{1,k} |^{-2} (v_1^{2,k}\cdot v_1^{1,k})v_1^{1,k} \Big) \cdot s^k_2(\boldsymbol{x}')\\& = -\frac{1}{2}(|v_1^{1,k}|^2 + |w_1^{1,k}|^2)\cos(\theta_k(\boldsymbol{x}')) +   |v_1^{1,k} |^{-2}(v_1^{2,k}\cdot v_1^{1,k}) (v_1^{1,k}  \cdot s^k_2(\boldsymbol{x}')) \\
& \le  -\frac{1}{2}(|v_1^{1,k}|^2 + |w_1^{1,k}|^2)\cos(\theta_k(\boldsymbol{x}')) + C\sin(\theta_k(\boldsymbol{x}'))|\boldsymbol{w}_1|^2,
\end{align*}
and repeating the same calculation for $v_2^{2,k}$, we derive for $k=l,r$
\begin{equation}\label{cross12} 
\big( v_1^{2,k} \cdot s^k_2(\boldsymbol{x}') + v_2^{2,k} \cdot
s^k_1(\boldsymbol{x}') \big) \le  \sum_{j=1,2}
-\frac{1}{2}(|v_j^{1,k}|^2 +
|w_j^{1,k}|^2)\cos(\theta_k(\boldsymbol{x}'))
+C\sin(\theta_k(\boldsymbol{x}'))|\boldsymbol{w}_1|^2. 
\end{equation}
Note that  $v_1^{1,k} \cdot v_2^{1,k} = |v_1^{1,k}||v_2^{1,k}| q \cos(\theta_k(\boldsymbol{x}'))$ for $q \in \lbrace -1 , 1 \rbrace$  by \eqref{u2}(a). An elementary computation then yields
\begin{align}\label{taylor3}
\big(v_1^{1,k} \cdot s^k_2(\boldsymbol{x}') + v_2^{1,k} \cdot s^k_1(\boldsymbol{x}')\big)^2 \ =   \sin^2(\theta_k(\boldsymbol{x}'))  (|v_1^{1,k}|-q|v_2^{1,k}| )^2.
\end{align}
Combining \PPP {\BBB \eqref{taylor1}--\eqref{taylor3}} \EEE and using that  $\arccos'(x) = -(1-x^2)^{-1/2}$ and that $\arccos''(x)=-x(1-x^2)^{-3/2}$,  we find
\begin{align}
f_1&(t) -f_1(0) \notag\\
& \ge \sum_{k=l,r}  -\sin(\theta_k(\boldsymbol{x}'))^{-1}\Big( \sum_{j=1,2}  - \frac{1}{2}(|v_j^{1,k}|^2 + |w_j^{1,k}|^2)\cos(\theta_k(\boldsymbol{x}')) + C\sin(\theta_k(\boldsymbol{x}'))|\boldsymbol{w}_1|^2\notag\\&
\ \ \  + w_1^{1,k} \cdot w_2^{1,k}  + |v_1^{1,k}||v_2^{1,k}| q \cos(\theta_k(\boldsymbol{x}')) \Big)t^2 \notag \\& \ \ \ - \frac{1}{2}\cos(\theta_k(\boldsymbol{x}')) (1-\cos^2(\theta_k(\boldsymbol{x}')))^{-3/2}\sin^2(\theta_k(\boldsymbol{x}'))     (|v_1^{1,k}| - q|v_2^{1,k}|)^2 t^2      + {\rm O}(|\boldsymbol{w}_1|^3t^3)\notag \\
&  =   \sum_{k=l,r}  -\sin(\theta_k(\boldsymbol{x}'))^{-1}\Big( \sum_{j=1,2} - \frac{1}{2}|w_j^{1,k}|^2\cos(\theta_k(\boldsymbol{x}'))  + w_1^{1,k} \cdot w_2^{1,k}   \Big)t^2   - C|\boldsymbol{w}_1|^2t^2   + {\rm O}(|\boldsymbol{w}_1|^3t^3)\notag \\
& \ge \sum_{k=l,r}  -\sin(\theta_k(\boldsymbol{x}'))^{-1}\Big( \sum_{j=1,2} \frac{1}{2}|w_j^{1,k}|^2  + w_1^{1,k} \cdot w_2^{1,k}   \Big)t^2   - C|\boldsymbol{w}_1|^2t^2   + {\rm O}(|\boldsymbol{w}_1|^3t^3). \label{taylor7}
\end{align}
In the last step we used that \RRR $\cos\theta\ge -1$. \EEE  Before we proceed let us note that the same computation can be repeated for the second mapping $f_2$ defined in \eqref{eq: f1,f2}: considering an expansion as in \eqref{expansion} with $s^k_j(\mathcal{S}(\boldsymbol{x}))$ in place of $s^k_j(\boldsymbol{x}')$ and indicating the vectors by $\hat{v}^{i,k}_j$ and $\hat{w}^{i,k}_j$ (perpendicular and parallel to $n_k(\mathcal{S}(\boldsymbol{x}))$, respectively)  we also obtain
\begin{align}\label{taylor8}
f_2&(t) -f_2(0)  \notag\\
&\ge \sum_{k=l,r}  -\frac{1}{\sin(\theta_k(\mathcal{S}(\boldsymbol{x})))}\Big( \sum_{j=1,2}  \frac{1}{2}|\hat{w}_j^{1,k}|^2  + \hat{w}_1^{1,k} \cdot \hat{w}_2^{1,k}   \Big)t^2  -C|\boldsymbol{w}_2|^2 t^2   + {\rm O}(|\boldsymbol{w}_2|^3t^3).
\end{align}

\emph{Step III.} We now investigate \eqref{taylor7}-\eqref{taylor8} \PPP in more detail. \EEE Consider first $f_1$. Due to the symmetry of the setting induced by the reflection $S_1$ (recall \eqref{reflexion}) we find $u^k_1(t)\cdot n_k(\boldsymbol{x}') = u^k_2(-t)\cdot n_k(\boldsymbol{x}') $ for $k=l,r$. In particular, {\BBB  taking the derivative in $t$ and using \eqref{u2}(a),}  this implies $w_1^{1,k} = -w_2^{1,k}$. Then by \eqref{taylor7} we obtain $$f_1(t) -f_1(0) \ge -C|\boldsymbol{w}_1|^2t^2 +   {\rm O}(|\boldsymbol{w}_1|^3t^3)$$ and therefore taking $t \to 0$ we get $ \boldsymbol{w}_1^TD^2g(\boldsymbol{x}')\boldsymbol{w}_1 \ge -C|\boldsymbol{w}_1|^2$, which establishes the first part of \eqref{concave3}. Now consider $f_2$. Notice that one can show $\hat{w}_1^{1,k} = \hat{w}_2^{1,k} $ for $k=l,r$ by symmetry, \PPP i.e., \EEE we cannot repeat the same argument as for $f_1$. However, in this case we can show
\begin{align}\label{taylor9}
|\hat{w}_1^{1,l}| + |\hat{w}_1^{1,r}|  +|\hat{w}_2^{1,l}| + |\hat{w}_2^{1,r}| \le C|\boldsymbol{w}_2| \ell^{-1}.
\end{align}
Once this is proved, the assertion follows. Indeed, due to symmetry of $\mathcal{S}(\boldsymbol{x})$ we observe that $\theta_l(\mathcal{S}(\boldsymbol{x})) = \theta_r(\mathcal{S}(\boldsymbol{x}))$, denoted by $\varphi$ in the following. Recalling \eqref{kink} and the fact that $\mathcal{S}(\boldsymbol{x})$ is near $\boldsymbol{x}_{\rm kink}^\ell$, we get $\varphi \le  \pi- c\ell^{-1}$   and $\sin(\varphi) \ge c\ell^{-1}$ for some $c>0$. Then by  \eqref{taylor8} we have $$f_2(t) -f_2(0)  \ge - C|\boldsymbol{w}_2|^2t^2 - C\ell \cdot   |\boldsymbol{w}_2|^2 \ell^{-2} t^2 +  {\rm O}(|\boldsymbol{w}_2|^3t^3),$$ which shows the second part of \eqref{concave3}. 

Let us finally show  \eqref{taylor9}. Recall the definition of the \RRR unit \EEE normal vectors $n_1^k(\boldsymbol{x}), n_2^k(\boldsymbol{x})$, and $n_k(\boldsymbol{x})$ introduced before \eqref{concave0} for $k=l,r$. Observe that \PPP by symmetry reasons \EEE we have $n_k(\mathcal{S}(\boldsymbol{x})) = \pm e_1$ and $|n_j^k(\mathcal{S}(\boldsymbol{x})) \cdot e_2| = \sin(\frac{\pi - \varphi}{2})$ for $j=1,2$, $k=l,r$. Then a continuity argument gives $|n_k(\boldsymbol{x}') \cdot e_3| \le C|\boldsymbol{w}_2|$ and $|n^k_j(\boldsymbol{x}') \cdot e_2| \le \sin(\frac{\pi - \varphi}{2}) + C|\boldsymbol{w}_2|$ for $k=l,r$ and $j=1,2$.  Moreover, as $\boldsymbol{x}'$ is invariant under the reflection $S_1$ (recall \eqref{reflexion}), we get $n_k(\boldsymbol{x}') \cdot e_2 = 0$. 
By definition of $s^k_j(\boldsymbol{x}')$ this implies   $$|s^k_j(\boldsymbol{x}') \cdot e_1|  = \big|\big(n_k(\boldsymbol{x}') \times n^k_j(\boldsymbol{x}') \big) \cdot e_1\big| {\BBB =|n_k(\boldsymbol{x}') \cdot e_3| |n^k_j(\boldsymbol{x}') \cdot e_2|} \le C\sin(\frac{\pi - \varphi}{2})|\boldsymbol{w}_2| + C|\boldsymbol{w}_2|^2.$$  For  a small enough perturbation parameter $\eta_\ell$  we get $|\boldsymbol{w}_2| \le \ell^{-1}$ and thus $|s^k_j(\boldsymbol{x}') \cdot e_1| \le   C|\boldsymbol{w}_2|\ell^{-1}$ \RRR since $\sin(\frac{\pi - \varphi}{2}) \le c\ell^{-1}$ by \eqref{kink}. \EEE  As $s^k_j(\boldsymbol{x}')\cdot e_1 = s^k_j(\mathcal{S}(\boldsymbol{x}))\cdot e_1 - \hat{w}_j^{1,k} + {\rm O}(|\boldsymbol{w}_2|^2) = - \hat{w}_j^{1,k} + {\rm O}(|\boldsymbol{w}_2|^2) $ \RRR (see \eqref{expansion} and use the fact that $s^k_j(\mathcal{S}(\boldsymbol{x}))\cdot e_1 = 0$), \EEE this shows \eqref{taylor9} and concludes the proof.
\end{proof}

\section*{Acknowledgements} 
\RRR M.F. acknowledges support from the Alexander von Humboldt Stiftung. 
 E.M. acknowledges support from the Austrian
Science Fund (FWF) project M 1733-N20.  \PPP
 P. P. acknowledges support
from the Austrian Science Fund (FWF) project P~29681, and from the Vienna Science and Technology Fund (WWTF), the
City of Vienna, and the Berndorf Private Foundation through Project MA16-005. 
\EEE
 U.S. acknowledges support from
the Austrian Science Fund (FWF) projects  P~27052, I~2375, and F~65
and from the  Vienna Science and Technology Fund (WWTF)
through project MA14-009.
The \PPP authors  \EEE would like to acknowledge
the kind hospitality of the Erwin
Schr\"odinger International Institute for Mathematics and Physics,
where part of this research was developed under the frame of the \PPP thematic
 program \EEE {\it Nonlinear Flows}. \EEE

Conflict of Interest: The authors declare that they have no conflict of interest.

\vspace{15mm}

\bibliographystyle{alpha}

\end{document}